\documentclass[letterpaper]{amsart}
\usepackage[utf8]{inputenc}
\usepackage{amsmath, amssymb, amsthm, verbatim, hyperref, tikz}
\usepackage{mathtools,bm,bbm,framed,xcolor,graphicx,subcaption,dsfont,booktabs,pst-node,tikz-cd,cleveref}

\usepackage{tikz-cd}
\usepackage[margin=1in]{geometry}
\usepackage{ytableau}
\usepackage{hyperref}
\usepackage{comment}
\usepackage{listings}
\usepackage{algorithm}
\usepackage{array}
\usepackage{longtable}
\usepackage{multirow}
\usepackage[all,cmtip]{xy}
\usepackage[mathscr]{euscript}

\newtheorem{theorem}{Theorem}[section]
\newtheorem{lemma}[theorem]{Lemma}
\newtheorem{prop}[theorem]{Proposition}
\newtheorem{cor}[theorem]{Corollary}

\newtheorem{definition}[theorem]{Definition}

\newtheorem{example}[theorem]{Example}
\newtheorem{remark}[theorem]{Remark}
\numberwithin{equation}{section}

\newcommand{\GL}[1]{\operatorname{GL}_{#1}(\mathbb{F}_q)}

\newcommand{\F}{\mathcal{F}}
\newcommand{\FF}{\mathbb{F}}

\newcommand{\HH}{\mathcal{H}}

\newcommand{\syt}{{\sf SYT}}
\newcommand{\B}{{\mathcal{B}}}

\newcommand{\CC}{{\mathbb{C}}}
\newcommand{\R}{{\mathbb{R}}}

\newcommand{\Z}{{\mathbb{Z}}}
\newcommand{\D}{{\mathcal{D}}}
\newcommand{\RR}{{\mathcal{R}}}
\newcommand{\TT}{{\mathcal{T}}}
\newcommand{\setofcontent}{{\mathfrak{C}}}

\newcommand{\symm}{{\mathfrak{S}}}
\DeclareMathOperator{\im}{im}
\DeclareMathOperator{\shape}{sh}

\newcommand{\spann}{\operatorname{span}}

\newcommand{\frakt}{\mathfrak{t}}
\newcommand{\fraks}{\mathfrak{s}}
\newcommand{\frakq}{\mathfrak{q}}

\newcommand{\arinj}{\ar@{_{(}->}}
\newcommand{\arinjrev}{\ar@{^{(}->}}
\newcommand{\arsurj}{\ar@{->>}}

\newcommand{\word}{\mathrm{word}}

\newcommand{\sm}{\setminus}
\newcommand{\basis}{\mathfrak{B}}
\newcommand{\content}{\mathfrak{c}}
\newcommand{\eigen}{\mathcal{E}}

\newcommand{\tT}{\widetilde{T}}

\newcommand{\gtdom}{\mathrel{\triangleright}}
\newcommand{\geqdom}{\mathrel{\trianglerighteq}}
\newcommand{\leqdom}{\mathrel{\trianglelefteq}}
\newcommand{\ledom}{\mathrel{\triangleleft}}
\newcommand{\res}{{\sf Res}}
\newcommand{\ind}{{\sf Ind}}
\newcommand{\End}{{\sf End}}
\newcommand{\DD}{\mathfrak{D}}

\usetikzlibrary{shapes.symbols}

\title[Spectrum of random-to-random shuffling in the Hecke algebra]{Spectrum of random-to-random shuffling \\ in the Hecke algebra}

\author[Axelrod-Freed]{Ilani Axelrod-Freed}
\email{ilani\_af@mit.edu}
\address{Massachusetts Institute of Technology, Cambridge MA, 02139}
\author[Brauner]{Sarah Brauner}
\email{sarahbrauner@gmail.com}
\address{Brown University, Providence RI, 02906}
\author[Chiang, Commins]{Judy Hsin-Hui Chiang, Patricia Commins}
\email{hchiang@umn.edu, commi010@umn.edu}
\address{University of Minnesota, Minneapolis MN 55455}
\author[Lang]{Veronica Lang}
\email{langv@umich.edu}
\address{University of Michigan, Ann Arbor MI 48109}

\keywords{Hecke algebra, Markov chain, card shuffling, random-to-random, Jucys-Murphy elements, Young seminormal forms, Okounkov-Vershik approach to representation theory, integral probability}
\subjclass{
05E10, 
60J10 
20C08 
}

\begin{document}
\begin{abstract}
We generalize random-to-random shuffling from a Markov chain on the symmetric group to one on the Type $A$ Iwahori Hecke algebra, and show that its eigenvalues are polynomials in $q$ with non-negative integer coefficients. Setting $q=1$ recovers results of Dieker and Saliola, whose computation of the spectrum of random-to-random in the symmetric group resolved a nearly 20 year old conjecture by Uyemura-Reyes. Our methods simplify their proofs by drawing novel connections to the Jucys-Murphy elements of the Hecke algebra, Young seminormal forms, and the Okounkov-Vershik approach to representation theory. 
\end{abstract}
\maketitle
\section{Introduction}
In this paper, we generalize the well-known but mysterious shuffling process \emph{random-to-random} $\RR_n$ from a Markov chain on the symmetric group $\symm_n$ to a Markov chain on the Type A Iwahori Hecke algebra $\HH_n(q)$. Building on seminal work by Dieker and Saliola \cite{DiekerSaliola}, we compute the complete spectrum of $\RR_n$ in $\HH_n(q)$. Our methods simplify the proof for $q=1$ by adopting the Okounkov-Vershik approach to the representation theory of the symmetric group and Hecke algebra, and drawing connections to the Jucys-Murphy elements and Young seminormal basis of $\symm_n$ and $\HH_n(q)$. 

This project is motivated by a growing interest in studying random walks on $\HH_n(q)$ from a combinatorial perspective. There is a rich connection between $\HH_n(q)$ and interacting particle systems, beginning with the work of Alcaraz--Droz--Henkel--Rittenberg \cite{alcaraz1994reaction} who realized that the generators of \emph{asymmetric simple exclusion processes (ASEPs)} satisfy the algebra relations of $\HH_n(q)$. Bufetov then showed in \cite{bufetov2020interacting} that this connection could be generalized to numerous important interacting particle systems with multiple species arising from statistical mechanics, including  ASEP, $M$-exclusion, TASEP, and stochastic vertex models \cite{bufetov2020interacting}. Many of these systems have been studied using algebraic combinatorics with great success; see for example \cite{ayyer2020modified, ayyer2022modified, corteel2022compact, corteel2017combinatorics, corteel2021cylindric,corteel2022multiline, diaconis2000analysis, diaconis2023double}. 

On the other hand, there is a beautiful theory of random walks on hyperplane arrangements (and more generally, left regular bands) pioneered by Bidigare--Hanlon--Rockmore \cite{bidigare1999combinatorial} and Brown \cite{BrownOnLRBs}, which was originally built as a way of understanding and computing the mixing times of card shuffling processes, i.e. Markov chains on the symmetric group. This approach forges important links between combinatorial representation theory, probability, statistical physics and dynamic data storage; see \cite{ assaf2009riffle, athanasiadis2010functions,ayyer2014combinatorial, ayyer2017spectral, diaconis2003mathematical, diaconis1995riffle, diaconis2000analysis, diaconis2023double, DiekerSaliola, Lafreniere-thesis}.  It has since been generalized to a broad class of random walks on semigroups in work of  Ayyer--Schilling--Steinberg--Thi\'ery \cite{ASST} and Rhodes--Schilling \cite{rhodes2019unified}.

Our work serves to unite these two perspectives, by defining and studying one of the most important shuffling processes arising in the latter setting---random-to-random---as a Markov chain on $\HH_n(q)$.

\subsubsection*{Random-to-random shuffling in the symmetric group}
The random-to-random shuffling process $\RR_n$ acts on a permutation $(w_1, \cdots, w_n) \in \symm_n$ by removing a ``card'' $w_i$ with uniform probability then re-inserting it with uniform probability to a new position in the deck. One can think of this as a two-step process: 
\begin{enumerate}
    \item Apply \emph{random-to-bottom} $\B^*_n$, which moves $w_i$ to the end of the word;
    \item Apply \emph{bottom-to-random} $\B_n$, which moves the last letter of the word to a new position $j$. 
\end{enumerate}
Formally, $\RR_n$ is the linear map $\CC[\symm_n] \longrightarrow \CC[\symm_n]$ acting on $w \in \symm_n$ by \emph{right} multiplication (i.e. by position):
\begin{align*}
    \RR_n(w):= w \cdot \underbrace{\left( \sum_{i=1}^{n} s_{i} \cdots s_{n-1}\right)}_{=:\B_n^*} \underbrace{\left( \sum_{j = 1}^{n} s_{n-1} \cdots s_{j} \right)}_{=:\B_n},
\end{align*}
where $s_k$ denotes the transposition $(k,k+1)$ in $\symm_n$ swapping $k$ and $k + 1.$ To obtain a random walk on $\CC[\symm_n]$, normalize both $\B_n^*$ and $\B_n$ by $\frac{1}{n}$; after normalization, the coefficient $[u]\RR_n(w)$ of $u \in \symm_n$ is the probability of obtaining $u$ from $w$ after one iteration of $\RR_n$. 

Random-to-bottom shuffling $\B_n^*$ is very well-studied and has numerous interesting connections to combinatorics.  Bidigare--Halon--Rockmore showed in \cite{bidigare1999combinatorial} that the eigenvalues of $\B_n^*$ acting on $\CC[\symm_n]$ are $0, 1, \cdots, n-2, n$, recovering a result of Phatarfod \cite{Phatarfod1991ONTM}. The multiplicity of the eigenvalue $j$ is $\binom{n}{j} d_{n-j},$
where $d_{n-j}$ is the $(n-j)$-\emph{derangement number} counting the number of permutations in $\symm_{n-j}$ with zero fixed points. The kernel of $\B^*_n$ carries the so-called \emph{derangement representation} introduced by D\'esarm\'enien and Wachs \cite{FrenchDesarmenienWachs}, which is related to well-loved objects such as Gessel's fundamental quassisymmetric function \cite{Gessel}, the free left regular band \cite{brauner2023invariant}, the complex of injective words \cite{ReinerWebb}, and the configuration space of $n$ points in $\R^3$ \cite{hersh2017representation}. 

Random-to-\emph{random} shuffling is significantly harder to understand. It was first defined by Diaconis (see \cite[p100]{Uyemura-Reyes}), and studied by Uyemura-Reyes in his thesis \cite{Uyemura-Reyes}. Uyemura-Reyes conjectured that the eigenvalues of $\RR_n$ were non-negative integers, and proved this to be true in several cases. The full conjecture was open for almost two decades, until it was resolved by Dieker and Saliola in 2018 \cite{DiekerSaliola}. Random-to-random is a special case of a broader family of ``symmetrized shuffling operators'' whose spectral properties are still quite mysterious; see \cite{kqr2r2025,Lafreniere-thesis, RSW}.

To state Dieker and Saliola's solution, recall that a skew shape (i.e., skew Young diagram) $\lambda \sm \mu$ is a \emph{horizontal strip} if it has at most one box in each column. The \emph{content} $ \content_{\lambda\sm \mu}$ of a skew shape $\lambda \sm \mu$ is defined by summing the difference of the column and row number for each box in $\lambda \sm \mu.$ Formally, letting $(i,j)$ indicate the coordinates of the box in the $i$-th row (ordered top-to-bottom, in English notation) and $j$-th column (ordered left-to-right),
\[ \content_{\lambda\sm \mu}:= \sum_{(i,j) \in \lambda \sm \mu} (j-i).\]  

Dieker and Saliola showed that the eigenvalues of $\RR_n$ acting on $\CC[\symm_n]$ are indexed by  horizontal strips $\lambda \sm \mu$, where $|\lambda| = n$ and $\mu \subseteq \lambda.$
The horizontal strip $\lambda \sm \mu$ corresponds to the eigenvalue 
\begin{equation}\label{eq:q=1_eigenvalues} 
\eigen_{\lambda \sm \mu} := \content_{\lambda \sm \mu} + \sum_{k=|\mu|+1}^{n} k. \end{equation}
Equation \eqref{eq:q=1_eigenvalues} implies that, remarkably,  $\eigen_{\lambda \sm \mu} \in \Z_{\geq 0}$, thereby proving Uyemura-Reyes's conjecture. Using \eqref{eq:q=1_eigenvalues}, Bernstein--Nestoridi \cite{bernstein2019cutoff} proved that $\RR_n$ exhibits cutoff behavior at 
\[ \frac{3}{4} n \log(n) - \frac{1}{4} \log(\log(n)). \]

At the heart of the Dieker--Saliola's proofs in \cite{DiekerSaliola} is the representation theory of the symmetric group, which they use to inductively construct eigenvectors of $\RR_n$ from the kernels of $\RR_j$ for $j < n$. Our work will follow a similar strategy, but utilize different tools that both simplify their arguments and deepen the connections between $\RR_n$ and fundamental concepts in representation theory.

\subsubsection*{Generalization to the Hecke algebra}
Given $q \in \CC$, the Type $A$ Iwahori Hecke algebra $\HH_n(q)$ is the associative $\CC$-algebra on the generators $T_{s_1}, \cdots, T_{s_{n-1}}$, subject to the relations
\begin{enumerate}
    \item $T_{s_i}^2 = (q-1)T_{s_i} + q$ for all $1 \leq i \leq n-1$, \smallskip
    \item $T_{s_i} T_{s_j} = T_{s_j} T_{s_i}$ when $|i-j| \geq 2$, and \smallskip
    \item $T_{s_i} T_{s_{i+1}} T_{s_i} = T_{s_{i+1}} T_{s_i} T_{s_{i+1}}$ for all $1 \leq i \leq n-2$. \smallskip
\end{enumerate} 
Note that $\HH_n(q)$ is a $q$-deformation of the symmetric group algebra $\CC[\symm_n]$, where $\HH_n(1) = \CC[\symm_n]$. As in the case of $\CC[\symm_n]$, the Hecke algebra has a $\CC$-basis given by $\{ T_w \}$ for $w$ in $\symm_n$. 

We define the following $q$-deformation of $\RR_n$.

\begin{definition}\label{def:qr2r_intro}
For any $q \in \CC$, define $q$-random-to-random shuffling $\RR_n(q): \HH_n(q) \longrightarrow \HH_n(q)$ by linearly extending 
\[ {\RR_n(q)}(T_w):= T_w \cdot \underbrace{\left( \sum_{i=1}^{n} T_{s_i} \cdots T_{s_{n-1}} \right)}_{=: {\B^*_n(q)}} \underbrace{\left( \sum_{j=1}^{n} \  T_{s_{n-1}} \cdots T_{s_j} \right)}_{=:{\B_n(q)}}. \]
 \end{definition}

Putting additional assumptions on $q$ allows us to define a random walk $\widetilde{\RR}_n(q)$ on $\HH_n(q)$ using a construction by Diaconis--Ram \cite[Theorem 4.3]{diaconis2000analysis}. Assume $q \geq 1 \in \R$, so that $q^{-1} \in (0,1] \subseteq \R$ can be understood as a probability. Define $\tT_{s_i}:= q^{-1}T_{s_i}$, and more generally let $\tT_w:= q^{-\ell(w)}T_w$, where $\ell(w)$ is the Coxeter length of the reduced word $w \in \symm_n$. Then $\tT_{s_i}$ acts by right multiplication on $\tT_w$: 
\begin{equation}\label{eq:heckeaction}
    \tT_w \cdot \tT_{s_i} := \begin{cases}
        \tT_{w s_i} & \ell(w s_i) > \ell(w) \\
        q^{-1} \tT_{w s_i} + (1-q^{-1}) \tT_w & \ell(w s_i) < \ell(w),
    \end{cases}
\end{equation}
thereby defining a Markov chain on $\HH_n(q)$.

Recall that the $q$-integer $[n]_q$ is defined for any $n \in \Z$ by 
\[ [n]_q:= \frac{1-q^n}{1-q} = \begin{cases}
    1 + q+ \cdots + q^{n-1} & n > 0 \\
    0 & n = 0\\
-q^{-1} - q^{-2} - \cdots - q^{n} & n < 0. 
\end{cases}\]
We can thus define a random walk $\widetilde{\RR}_n(q)$ by rewriting $\RR_n(q)$ in terms of the $\tT_{s_i}$ and normalizing: 
\begin{equation}\label{eq:probr2rinto} {\widetilde{\RR}_n(q)}(\tT_w):= \tT_w \cdot \underbrace{{\left( \frac{1}{[n]_q} \sum_{i=1}^{n} q^{n-i} \  \tT_{s_i} \cdots \tT_{s_{n-1}} \right)}}_{=: \widetilde{\B}_n^*(q)} \underbrace{{\left(\frac{1}{[n]_q} \sum_{j=1}^{n} q^{n-j} \  \tT_{s_{n-1}} \cdots \tT_{s_j} \right)}}_{=: \widetilde{\B}_n(q)}. \end{equation}

The goal of our work is to characterize the spectrum of ${\RR_n(q)}$ acting by right multiplication on $\HH_n(q)$. Observe that the eigenvalues of $\widetilde{\RR}_n(q)$ can be obtained immediately from those of $\RR_n(q)$ by restricting to the case where $q^{-1} \in (0,1] \subseteq \R$ and dividing by $([n]_q)^2$.   

Our spectral formula for $\RR_n(q)$ uses two combinatorial statistics. First, given a skew shape $\lambda \sm \mu$, we define the $q$-content 
\[ \content_{\lambda \sm \mu}(q):= \sum_{(i,j) \in \lambda \sm \mu} [j-i]_q. \]
Note that $\content_{\lambda \sm \mu}(q)$ is a Laurent polynomial in $q$ with integer coefficients; however, $q^{|\lambda|} \content_{\lambda \sm \mu}(q) \in \Z[q]$. 

Second, a standard Young tableau $\frakt$ has a \emph{descent} at position $i \in [n-1]$ if $i+1$ appears south and weakly west of $i$ in $\frakt$. Let ${\sf Des}(\frakt)$ be the set of descents of $\frakt$. A standard tableau $\frakt$ of size $n$ is a \emph{desarrangement tableau} if the minimum element of $[n] \setminus {\sf Des}(\frakt)$ is even. Let $d^\mu$ be the number of desarrangement tableaux of shape $\mu$ and $f^\mu$ be the number of standard Young tableaux of shape $\mu$. Our main result is Theorem \ref{thm:intromainthm}.

\begin{theorem}\label{thm:intromainthm} For any $q \in \CC,$ the right action of $\RR_n(q)$ on $\HH_n(q)$ has the following properties:
\begin{enumerate}
    \item All eigenvalues of $\RR_n(q)$ are of the form 
            \[ \eigen_{\lambda \sm \mu}(q) =  q^{n} \content_{\lambda \sm \mu}(q) + \sum_{k=|\mu|+1}^{n} q^{n-k} \  [k]_q \]
    where $\lambda \sm \mu$ is a horizontal strip with $|\lambda| = n$ and $0 \leq |\mu| \leq n$. \medskip
    
    \item The (algebraic) multiplicity of a fixed eigenvalue $\eigen(q)$ is given by \[ \sum_{\substack{\lambda \sm \mu \textrm{ a horizontal strip}: \\ \eigen_{\lambda \sm \mu}(q) = \eigen(q)}} f^\lambda  \ d^\mu. \]

\item Every eigenvalue $\eigen_{\lambda \sm \mu}(q)$ is a polynomial in $q$ with non-negative integer coefficients.
\medskip
\item If $q \in \R_{>0}$, then $\RR_n(q)$ is diagonalizable.
\end{enumerate}
\end{theorem}

Theorem \ref{thm:intromainthm} has the following special cases. 

\begin{cor}\label{cor:introsecondeigen}
The case $\lambda = (n-k,1^k)$ and $\mu = (j-k,1^k)$ gives the eigenvalue
\[ \eigen_{(n-k,1^k) \sm (j-k,1^k)}(q) = [n-j]_q \ [n+j-k]_q.\]
Setting $j=k=0$ gives
$\eigen_{(n) \sm \emptyset}(q) = [n]_q  [n]_q$, which is the largest eigenvalue of $\RR_n(q)$ when $q \in \R_{>0}$.
The corresponding stationary distribution for $\widetilde{\RR}_n(q)$ is 
the \emph{Mallows measure} of $\symm_n$: 
\[ \mathcal{M}(\symm_n, q^{-1}) = \sum_{w \in \symm_n} q^{\ell(w)} \tT_w. \]
\noindent Setting $j = 2$ and $ k=1$ gives $ \eigen_{(n-1,1) \sm (1,1)}(q) = [n-2]_q [n+1]_q$, which is the second largest eigenvalue of $\RR_n(q)$ when $q \in \R_{>0}$.

\end{cor}




Our analysis also has implications for $\B_n^*(q)$ and $\B_n(q)$.

\begin{theorem}\label{thm:introb2r}
For any $q \in \CC$, the right actions of $\B_n^*(q)$ and $\B_n(q)$ on $\HH_n(q)$ have characteristic polynomial 
\[ \chi(y,q) = \prod_{j=0}^{n} (y - [n-j]_q)^{\binom{n}{j}d_{j}},  \]
 where $d_{j}$ is the $j$-derangement number counting permutations in $\symm_{j}$ with zero fixed points. 
 
 Moreover, whenever $\HH_n(q)$ is semisimple, both $\B_n^*(q)$ and $\B_n(q)$ are diagonalizable. 
\end{theorem}

In a follow-up paper to this one \cite{kqr2r2025}, the second and fourth authors, along with Grinberg and Saliola, define a deformation to $\HH_n(q)$ of a family of operators called $k$-random-to-random $\RR_{n,k}$ introduced in \cite{RSW}. The $\RR_{n,k}$ pairwise commute \cite{Lafreniere-thesis}, and $\RR_{n,1} = \RR_{n}$. 

\subsubsection*{Methods}
Our overall strategy is similar in spirit to \cite{DiekerSaliola}, in that we inductively construct eigenvectors for the action of $\RR_n(q)$ on the irreducible representations of $\HH_n(q)$ when $\HH_n(q)$ is semisimple using the kernels of $\RR_j(q)$ for $j < n$. However, our approach differs from \cite{DiekerSaliola} in technique. 

The novelty of our method is to recursively relate $\RR_n(q)$ to the \emph{Jucys-Murphy elements of} $\HH_n(q)$
\[ J_k(q):= \sum_{i=1}^{k-1} q^{i-k} \ T_{(i,k)} \]
and their many wonderful properties\footnote{Sometimes these elements are referred to as the \emph{additive Jucys-Murphy elements}, see \cite{elias2017categorical}.}. In particular, we prove the following recursion for $\RR_n(q)$.

\begin{theorem}\label{thm:introrecursion}
For any $q \in \CC$, the following recursion holds in $\HH_n(q)$:
\[ \B_{n}(q) \RR_{n}(q) = \bigg( q \RR_{n-1}(q) + [n]_q + q^{n} J_{n}(q) \bigg) \B_{n}(q). \]
\end{theorem}
This connection with $J_k(q)$ unlocks a wealth of tools put forth by Dipper--James \cite{DipperJames}, Mathas \cite{Mathas}, Hoefsmit \cite{Hoefsmit_1974}, and others, which develop the representation theory of the Hecke algebra as a parallel to that of the symmetric group. Chief among these are \emph{Young's seminormal units} which give an elegant basis for the irreducible representation $S^\lambda$ of $\HH_n(q)$ in terms of standard Young tableaux of shape $\lambda$. The seminormal units are a simultaneous eigenbasis for the Jucys-Murphy elements $J_1(q), \cdots, J_n(q)$, with eigenvalues determined by the $q$-contents of the standard Young tableaux indexing each element. 

We briefly summarize how our argument for Theorem \ref{thm:intromainthm} builds from Theorem \ref{thm:introrecursion}:

\begin{enumerate}
  
\item We first prove a connection between $\B_n^*(q)$ and a Markov chain on the space of flags of the finite field $\mathbb{F}_q^n$ studied in \cite{brauner2023invariant, BrownOnLRBs} (Proposition \ref{xisrantomtotop}). We use this connection to compute the kernel of $\B_n^*(q)$ when $\HH_n(q)$ is semisimple (Theorems \ref{thm:B-polys} and \ref{pf.cor:prime-eigenspace.dimker}), and prove that if $q \in \R_{>0}$, we have $\ker \B_n^*(q) = \ker \RR_n(q)$ (Lemma~\ref{lemma:whenqisreal} (2)). This implies that the nullity of $\RR_n(q)$ in this setting is the derangement number $d_n$.

\medskip

    \item When $\HH_n(q)$ is semisimple, for each $\lambda \vdash n$ we construct a set of eigenvectors of $\RR_n(q)$ in the irreducible Specht module $S^\lambda$ of $\HH_n(q)$. We explain in Theorem \ref{thm:recursive_eigenvector_construction} how given an eigenvector of $\RR_{n-1}(q)$ acting on $S^{\lambda'}$ with eigenvalue $\eigen,$ one can use Theorem \ref{thm:introrecursion} to obtain an eigenvector for $\RR_n(q)$ acting on $S^\lambda$ for $\lambda' \subseteq \lambda$ with eigenvalue
\begin{equation*}\label{eq:introeigenvectorrecursion} q  \eigen + [n]_q + q^{n} \content_{\lambda \sm \lambda'}(q). \end{equation*}
We then iterate this process in Theorem \ref{thm:eigenvectors}, starting with kernel eigenvectors of $\RR_j(q)$ for $j< n$. In particular, the eigenvalue $\eigen_{\lambda \sm \mu}(q)$ in Theorem \ref{thm:intromainthm} corresponds to an eigenvector in $S^\lambda$ constructed from a $\RR_{|\mu|}(q)$-kernel eigenvector in $S^\mu$.
\medskip

\item We show in Theorem \ref{thm:orderedspanningset} that when $q \in \R_{>0}$, all eigenvectors of $\RR_n(q)$ can be obtained by the process in (2), and define a set $\basis_\lambda$ that spans $S^\lambda$.  Our proof relies on a \emph{Straightening Lemma} (Lemma \ref{lem:horiz-strip-diff-by-cnostant}) that draws upon the remarkable properties of Young's seminormal units.
We prove and later apply a \emph{Horizontal Strip Lemma} (Lemma \ref{horizontalstriplemma}) to show in Lemma \ref{lemma:xtoy} that the elements of $\basis_\lambda$ are indexed by horizontal strips.

\medskip

\item We combine points (1), (2), and (3) above to show that the spanning set $\basis_\lambda$ of $S^\lambda$ is indeed an $\RR_n(q)$-eigenbasis for $q \in \R_{>0}$. We use this fact to compute the characteristic polynomial of $\RR_n(q)$ acting on $\HH_n(q)$ for any $q \in \CC$, thereby obtaining Theorem \ref{thm:summarythm}.
\end{enumerate}

The remainder of the paper proceeds as follows. In Section \ref{section:reptheoryofhecke} we explain the necessary semisimple representation theory of $\HH_n(q)$, including the construction of the Young idempotents and seminormal forms. In Section \ref{sec:branching-rules-Hecke} we describe branching rules for $\HH_n(q)$, and prove the Horizontal Strip Lemma (Lemma \ref{horizontalstriplemma}). In Section \ref{sec:flagintroduction}, we study $\B_n^*(q)$ using the action on flags, as described in (1) above, and prove Theorem \ref{thm:introb2r}. We prove Theorem \ref{thm:introrecursion} and recursively construct $\RR_n(q)$-eigenvectors in Section \ref{section:constructionofeigenvectors}, as outlined in (2) above. We then turn our attention to the case when $q \in \R_{>0}$ in Section \ref{section:spanning}, and prove the results from (3) above. Finally, in Section \ref{section:maintheorem}, we obtain an $\RR_n(q)$-eigenbasis for $S^\lambda$ when $q$ is positive (Theorem \ref{thm:eigenbasisSlambda}) as in (4). We use this to prove Theorem \ref{thm:intromainthm} and Corollary \ref{cor:introsecondeigen}. Data for $\RR_n(q)$ up to $n=5$ can be found in Appendix \ref{appendix:data}.

\subsection*{Acknowledgements}
The authors are extremely appreciative of Darij Grinberg for providing extensive edits, feedback and fruitful discussions about this work, which greatly improved the exposition and corrected several issues and subtleties in earlier drafts of the paper. The authors are also grateful to Alexey Bufetov for illuminating discussions about random walks on Hecke algebras and for pointing us to \cite{diaconis2000analysis}; to Ben Elias for first introducing us the perspective in \cite{elias2017categorical}; to Vic Reiner for suggesting the problem and insightful guidance throughout the project; to Franco Saliola for invaluable discussions about Theorem \ref{thm:orderedspanningset}; and to Peter Webb for help with Lemma \ref{horizontalstriplemma}. The authors benefited from conversations with Anne Schilling, Arvind Ayyer, Jan de Gier, Zachary Hamaker, Andy Hardt, Arun Ram, Jeanne Scott, and Nadia Lafreni\`ere. 

All computations were done using SageMath. This project began as part of the University of Minnesota, Twin Cities REU in Algebra and Combinatorics, and all authors received partial support from NSF DMS-1745638. The second author was supported an NSF Mathematical Sciences Research Postdoctoral Fellowship under DMS 2303060, and the fourth author was supported by an NSF Graduate Research Fellowship.

\section{(Semisimple) Representation theory of $\HH_n(q)$}\label{section:reptheoryofhecke}
In this section we review the representation of the Type $A$ Iwahori Hecke algebra $\HH_n(q)$, with a view towards the Okounkov-Vershik perspective. For an excellent contextual overview of this approach, see Elias--Hogancamp \cite[\S 1]{elias2017categorical}. We assume for the remainder of this section that $\HH_n(q)$ is semisimple, which occurs if and only if $q$ is neither $0$ nor a primitive $k$th root of unity for some $2 \leq k \leq n.$ Note that for these values of $q,$ if $-n \leq m \leq n,$ then $[m]_q \neq 0$ if $m \neq 0$. For a treatment of the non-semisimple case, see \cite{Mathas}.

In what follows, we will write $\syt(\lambda)$ for the set of standard Young tableaux of shape $\lambda$, and
\[ \syt(n) := \bigcup_{\lambda \vdash n} \syt(\lambda), \quad \quad \quad \syt := \bigcup_{n \geq 0} \syt(n).\] 

 Given a tableau $\frakt \in \syt(\lambda)$, the shape of $\frakt$ will be written $\shape(\frakt) = \lambda$. We will also work with skew diagrams $\lambda \sm \mu$, and many of our definitions above can be adapted in this case. Let 
\[ \syt(\lambda \sm \mu) \]
 be the set of skew tableaux of shape $\lambda \sm \mu$ filled with the letters $|\mu|+1, \cdots, |\lambda|$ with entries increasing across rows and down columns.

\subsection{Basics of the Hecke algebra and conventions}\label{section:basicsofhecke}
As discussed in the introduction, for any $q \in \CC$, the Hecke algebra $\HH_n(q)$ is defined as the associative $\CC$-algebra on $T_{s_1}, \cdots, T_{s_{n-1}}$, subject to the relations 
\begin{enumerate}
    \item $T_{s_i}^2 = (q-1)T_{s_i} + q$ for all $1 \leq i \leq n-1$, \medskip
    \item $T_{s_i} T_{s_j} = T_{s_j} T_{s_i}$ when $|i-j| \geq 2$, and \medskip
    \item $T_{s_i} T_{s_{{i+1}}} T_{s_i} = T_{s_{i+1}} T_{s_i} T_{s_{i+1}}$ for all $1 \leq i \leq n-2$.\medskip
\end{enumerate}
The elements $T_{s_i}$ should be thought of as $q$-deformations of the simple transpositions $s_i:=(i,i+1)$, which generate $\symm_n$. Recall that a reduced word $w \in \symm_n$ is a minimal expression of $w$ in the generators $s_1, \cdots, s_{n-1}$. The number of generators $s_i$ used to write a reduced word (counting multiplicity) is independent of the choice of reduced word, and is called the length $\ell(w)$ of $w$.

It is well known that $\HH_n(q)$ has a linear basis indexed by $w \in \symm_n$, 
which we shall write as 
\[ T_w := T_{s_{i_1}} \cdots T_{{s_{i_\ell}}},\]
where $s_{i_1}\cdots s_{i_\ell}$ is any reduced expression for $w.$

Using the generating relations for $\HH_n(q)$, one can deduce the right action of $\HH_n(q)$ on itself:
\begin{equation}\label{eq:heckeactiononitself} T_w T_{s_i} = \begin{cases} 
        T_{w s_i} & \ell(w s_i) > \ell(w) \\
      q  T_{w s_i} + (q-1)T_w & \ell(w s_i) < \ell(w).
    \end{cases}
\end{equation}

Recall from the introduction that our operator of interest is the element $\RR_n(q) = \B_n^*(q) \B_n(q)$ of $\HH_n(q)$, which acts by right multiplication on $\HH_n(q)$. As in the introduction,  $\B_n^*(q)$ and $ \B_n(q)$ are defined to be: 
 \begin{align}
    \label{eq:b2rdef} q\textrm{-bottom-to-random}:= \B_n(q) &:=  \sum_{i = 1}^n T_{s_{n - 1}} T_{s_{n - 2}} \cdots T_{s_{i}},\\
    \label{eq:r2bdef} q\textrm{-random-to-bottom}:=\B_n^\ast(q)& := \sum_{j = 1}^n T_{s_{j}}T_{s_{ j + 1}}\cdots T_{s_{n - 1}}.
 \end{align}
Before proceeding, we make a few remarks on  conventions for the remainder of the paper.
\begin{remark} \rm
In the introduction, we also introduced $\tT_{w} := q^{-\ell(w)} T_w$. Diaconis and Ram show \cite[Theorem 4.3]{diaconis2000analysis} that one can deduce the right action of $\tT_{s_i}$ on $\tT_w$ given in equation \eqref{eq:heckeaction} using \eqref{eq:heckeactiononitself}. 
Everything that follows can be rephrased in terms of the $\tT_w$ by assuming $q \neq 0$ and substituting $T_w$ with $q^{\ell(w)}\tT_w$.
\end{remark}

\begin{remark}[Anti-isomorphism $*$]\label{rmk:antiiso}\rm
There is an anti-isomorphism on $\HH_n(q)$ defined by:
\begin{align*}
    *: \HH_n(q) &\longrightarrow \HH_n(q)\\
    T_w &\longmapsto (T_w)^* := T_{w^{-1}}.
\end{align*}
Note that $*$ is also an involution.
The map $*$ is useful in studying properties of random walks on $\HH_n(q)$ (see \cite{bufetov2020interacting}).
As the notation suggests, $\B_n^*(q)$ is the image of $\B_n(q)$ under $*$.
\end{remark}

\begin{remark}[Bottom versus Top] \label{rmkbottomvtop}\rm
We have made the choice to let the end of a word or permutation be its bottom (read left to right), and the beginning be its top. In some of the literature \cite{bidigare1999combinatorial, brauner2023invariant, Phatarfod1991ONTM} the authors prefer to use \emph{random-to-top} and \emph{top-to-random}, defined in $\HH_n(q)$ as 
\begin{align*}
         q\textrm{-top-to-random}:= \TT_n(q) &:=  \sum_{i = 1}^n T_{s_{1}} T_{s_{2}} \cdots T_{s_{i-1}},\\
     q\textrm{-random-to-top}:=\TT_n^\ast(q)& :=  \sum_{j = 1}^n T_{s_{j-1}}T_{s_{ j - 1}}\cdots T_{s_{1}}.
\end{align*}
We move between these perspectives by applying the $\CC$-algebra isomorphism $\tau: \HH_n(q) \longrightarrow \HH_n(q)$, 
sending
\[ \tau(T_{s_i}) = T_{s_{n-i}}, \quad \quad \tau(\TT_n(q)) = \B_n(q), \quad \quad \tau(\TT^*_n(q)) = \B^*_n(q).\]
 It is straightforward to check that $\tau$ preserves the eigenvalues and multiplicities of an element in $\HH_n(q)$ acting via multiplication on $\HH_n(q).$  Hence studying the spectrum of $\B_n^\ast(q)$ is equivalent to studying that of $\TT^\ast_n(q)$. Note that we make the choice to study $\B_n^\ast(q)$ rather than $\TT^\ast_n(q)$ because its indexing is more compatible with our recursion (Theorem \ref{thm:introrecursion}). We will need to use the equivalence of $\TT_n^*(q)$ and $\B_n^*(q)$ to understand the kernel of $\RR_n(q)$ in Section \ref{section:kernelofrr}.
 \end{remark}

Recall that in the representation theory of the symmetric group, the irreducible representations---called \emph{Specht modules}---are indexed by partitions $\lambda$ of $n$. The Specht module indexed by $\lambda$ has dimension given by $f^\lambda$, the number of standard Young tableaux of shape $\lambda$. The number $f^\lambda$ also appears when decomposing $\CC[\symm_n]$ into irreducible representations: there are $f^\lambda$ copies of the Specht module indexed by $\lambda$ inside $\CC[\symm_n]$.

In the case that $\HH_n(q)$ is semisimple, these facts remain true: 
\begin{itemize}
    \item the irreducible representations of $\HH_n(q)$ are indexed by $\lambda \vdash n$, and will be denoted by $S^\lambda$;
    \item the dimension of each $S^\lambda$ is $f^\lambda$; and
    \item for every $\lambda \vdash n$, the multiplicity of $S^\lambda$ in $\HH_n(q)$ is $f^\lambda$.
\end{itemize}

As in the case of the symmetric group, one way of understanding the second and third points is by decomposing $\HH_n(q)$ as a $(\HH_n(q), \HH_n(q))$-bimodule, where $\HH_n(q)$ acts on itself by multiplication on both the right and left. In this case, we obtain that 
\begin{equation}\label{eq:leftrightbimoduledecomposition}
    \HH_n(q) \cong \bigoplus_{\lambda \vdash n} (S^{\lambda})^* \otimes S^\lambda,
\end{equation}
where $S^\lambda$ is an irreducible \emph{right} module of $\HH_n(q)$ and $(S^\lambda)^*$ is an irreducible \emph{left} module of $\HH_n(q)$; see \cite{Halverson-Ram} for more details on this perspective. 

The bimodule decomposition of $\HH_n(q)$ in \eqref{eq:leftrightbimoduledecomposition} will be essential for our analysis. Given an element $\varphi \in \HH_n(q)$, suppose the right action of $\varphi$ on $\HH_n(q)$ is diagonalizable (i.e. $\varphi$ acts semisimply) with eigenspaces $V(\eigen_1), \cdots, V(\eigen_k)$ corresponding to eigenvalues $\eigen_1, \cdots, \eigen_k$. Each $V(\eigen_j)$ is a left $\HH_n(q)$-module, since 
\[ y \cdot (x \cdot \varphi) = (y \cdot x) \cdot \varphi \] for any $x, y, \varphi \in \HH_n(q)$. The decomposition in \eqref{eq:leftrightbimoduledecomposition} can be used to deduce the left module structure of each $V(\eigen_j)$ from the spectrum of $\varphi$ acting on  $S^\lambda$.

\begin{cor}\label{cor:righteigenleftmodule}

Let $\HH_n(q)$ be semisimple and suppose $\varphi \in \HH_n(q)$ acts semisimply on $\HH_n(q)$ by right multiplication with eigenvalues $\eigen_1, \cdots, \eigen_k$. 

\begin{enumerate}
    \item Then the $\eigen_j$-eigenspace is a left $\HH_n(q)$-module $V(\eigen_j)$; and \medskip
\item The (representation) multiplicity of $(S^\lambda)^*$ in $V(\eigen_j)$ is $m_{\eigen_j}(\lambda)$ if and only if the (eigenvalue) multiplicity of $\eigen_j$ in $S^\lambda$ is $m_{\eigen_j}(\lambda)$.

\end{enumerate}
\end{cor}

Corollary \ref{cor:righteigenleftmodule} assumes that $\varphi$ is diagonalizable. However, under certain assumptions, if $\varphi$ is diagonalizable for infinitely many values of $q$, we may deduce the characteristic polynomial of $\varphi$ for \emph{any} $q \in \CC$.

\begin{lemma}\label{lem:suffices-to-prove-for-infinite-q}
   Suppose $\varphi_q : \HH_n(q) \longrightarrow \HH_n(q)$ is a linear transformation that can be represented by an $n! \times n!$ matrix with entries in $\Z[q]$. 
   \begin{enumerate}
   \item Let $\chi(y,q)$ be a polynomial in $\Z[y,q]$. If the characteristic polynomial 
   \begin{equation}\label{eqn:char-poly}
       \det\left(y\cdot 1 - \varphi_q\right) = \chi(y,q)
   \end{equation}
   for infinitely many $q \in \CC$, then $\det\left(y\cdot 1 - \varphi_q\right) = \chi(y,q)$ for all $q \in \CC$. \medskip
   \item
   If some $P \in \Z[y,q]$ satisfies $P(\varphi_q,q) = 0$ for infinitely many $q \in \CC$, then $P(\varphi_q,q) = 0$ for all $q \in \CC$.
   \end{enumerate}
\end{lemma}
\begin{proof}
(1) Both sides of (\ref{eqn:char-poly}) belong to $\Z[y, q].$ By comparing the coefficients of $y^i$ on both sides, it suffices to prove the following polynomial identities in $\Z[q]$ hold for all $i$:
\[[y^i] \, \det\left(y\cdot 1 - \varphi_q\right) = [y^i]\ \chi(y,q) .\]
By assumption, these identities are true for infinitely many $q.$ Hence,  the ``polynomial identity trick'' (i.e., the fact that two polynomials with complex coefficients are equal if they agree on infinitely many points) implies it holds for all $q.$ This proves (1). \medskip

\noindent (2) This is similar: both sides of the equation $P(\varphi_q,q) = 0$ belong to $\HH_n(q)$. By comparing the coefficients of a given basis element $T_w$ on both sides, we can rewrite this equation as a family of identities between polynomials in $q$:
\[
[T_w] \ P(\varphi_q,q) = 0 \qquad \text{for each } w \in \symm_n.
\]
Again, if such an equation holds for infinitely many $q$, then it will hold for all $q$.
\end{proof}

In what follows, we will let $\varphi = \RR_n(q)$ and study its right action on $S^\lambda$. We will show that $\RR_n(q)$ is diagonalizable when $q \in \R_{> 0}$ by constructing an eigenbasis for each $S^\lambda$, and then appeal to Lemma \ref{lem:suffices-to-prove-for-infinite-q} to deduce the characteristic polynomial of $\RR_n(q)$ for any $q \in \CC$.

\subsection{Explicit constructions of Specht modules using the Jucys-Murphy elements }\label{sec:okounkovvershik}

As in the case of the symmetric group, there are several ways to explicitly construct Specht modules of $\HH_n(q)$. Here, we will adopt what is sometimes referred to as the ``Okounkov-Vershik approach'' described in the influential paper \cite{okounkov1996new}, which reframes tools developed earlier by Murphy \cite{MURPHY1981287} and Jucys \cite{jucys1971factorization} for $\symm_n$, and Dipper--James \cite{DipperJames} for $\HH_n(q)$; see \cite[\S 1]{elias2017categorical} for a nice overview of this story. While we only discuss $\HH_n(q)$, setting $q = 1$ recovers the classical $\symm_n$-theory.

We will be interested in explicit presentations of right Specht modules $S^\lambda$, but the corresponding left module $(S^\lambda)^*$ can be constructed analogously. Note that there is also a way to realize $(S^\lambda)^* \otimes S^\lambda$ inside $\HH_n(q)$; this is the perspective taken in Mathas \cite{Mathas} (see Remark \ref{remark:mathasdictionary}).

The key ingredient will be the Jucys-Murphy elements of $\HH_n(q)$, defined below. 

\begin{definition}[Jucys-Murphy elements]\label{def:Jucys-Murphy}
For $1 \leq k \leq n,$ and $q \neq 0$, the $k$-th Jucys-Murphy element of $\HH_n(q)$ is defined as 
    \[ J_k(q):= \sum_{i = 1}^{k - 1} q^{i - k} \ T_{(i, k)}. \]
\end{definition}
\begin{example}\rm
    In $\HH_6(q),$
    \begin{align*}
        J_4(q) = q^{-3}T_{(1,4)} + q^{-2}T_{(2, 4)} + q^{-1}T_{(3, 4)} = q^{-3}T_{s_1s_2s_3s_2s_1}  + q^{-2}T_{s_2s_3s_2} + q^{-1}T_{s_3.}
    \end{align*}
\end{example}
The Jucys-Murphy elements satisfy several miraculous properties. First, they pairwise commute (see \cite[Proposition 3.26(iii)]{Mathas}). Second, an element $x$ is in the center of $\HH_n(q)$ if and only if it is a symmetric polynomial in the $J_k(q)$. The forwards direction of this fact is a result of Francis--Graham \cite{francis2006centres}, solving a conjecture of Dipper--James \cite{DipperJames}; the converse appeared much earlier (see, for example,  \cite[Corollary 3.27]{Mathas}). Finally---and most importantly for us---the Jucys-Murphy elements have a simultaneous eigenbasis known as
\emph{Young's seminormal basis}, which is a particularly nice basis for the Specht modules $S^\lambda$. The first step in constructing the seminormal basis is to define \emph{Young's idempotents}.

\subsubsection{Young's idempotents} \label{sec:idempotents}

For any semisimple algebra $A$ one can define canonical projectors onto the isotypic components of $A$; these projectors form a family of mutually orthogonal, central idempotents in $A$. In the case of $\HH_n(q)$, these projectors will be written as 
\[\{  p_\lambda \in \HH_n(q): \lambda \vdash n \}. \]  As discussed in Section \ref{section:basicsofhecke}, the $\lambda$-isotypic component of $\HH_n(q)$ contains $f^\lambda$-many copies of $S^\lambda$. Further, for any right $\HH_n(q)$-module $M,$ right multiplication by $p_\lambda$ projects $M$ onto a submodule isomorphic to a direct sum of $S^\lambda$'s. However, in general there is no canonical way to project from $A$ onto a single, irreducible representation of $A$.

Remarkably, when $A = \CC[\symm_n]$ or $\HH_n(q)$, there exists a family of idempotents that project onto the individual Specht modules in $\HH_n(q)$: these are Young's idempotents, which we shall write as 
\[ \{ p_{\frakt} \in \HH_n(q): \frakt \in \syt(n) \}.\] 
Right multiplication by $p_{\frakt}$ projects onto a copy of the left module $(S^\lambda)^*$ in \eqref{eq:leftrightbimoduledecomposition}. Since there are $f^\lambda$ distinct $p_{\frakt}$, this gives a method of projecting onto each copy of $(S^\lambda)^*$ in $\HH_n(q)$.

The Young idempotent $p_{\frakt}$ will also be---by design---a projector onto a simultaneous eigenspace of the Jucys-Murphy elements. The $q$-content of $k$ in $\frakt$ determines the eigenvalue with which each $J_k(q)$ acts. Write $x_k(\frakt)$ to be the index of the row of $\frakt$ containing the entry $k$ (counting from top to bottom), and $y_k(\frakt)$ to be the column containing $k$ (counting from left to right). 
\begin{definition}
Given a tableau $\frakt \in \syt(\lambda)$, the $q$-\emph{content of $\frakt$} at $k$ is 
\[ \content_{\frakt,k}(q):= [y_k(\frakt) - x_k(\frakt)]_q. \]    
\end{definition}
Setting $q=1$, this recovers the classical content $\content_{\frakt,k}:= \content_{\frakt,k}(1)$.
We will soon see that the Young idempotent $p_\frakt$ is an eigenvector of $J_k(q)$ with eigenvalue $\content_{\frakt,k}(q)$ for every $1 \leq k \leq n$.

\ytableausetup{boxsize=1em}

\begin{example}\rm
    Consider the standard Young tableau $\frakt \in \syt(3,3,1,1)$:
     \begin{align*}
     \frakt =
        \begin{ytableau}
            1 & 3 & 4\\
            2 & 5 & 8\\
            6\\
            7
        \end{ytableau}.
    \end{align*}
    Then, 
    \begin{align*}
        \content_{\frakt, 7}(q) = [-3]_q, \ \ \content_{\frakt, 6}(q) = [-2]_q, \  \content_{\frakt, 2}(q) = [-1]_q,  \ \ \content_{\frakt, 1}(q) = \content_{\frakt, 5}(q) = [0]_q,  \ \ \content_{\frakt, 3}(q) = \content_{\frakt, 8}(q)  = [1]_q,  \text{ and } \content_{\frakt, 4}(q) = [2]_q.
    \end{align*}
\end{example}

Denote by $\setofcontent(m)$ the set of all possible values $\content_{\frakt,m}$ for any $\frakt \in \syt$:
  \begin{align*}
        \setofcontent(m) := \begin{cases}
            \{k: -m < k < m\} &m \neq 2, 3\\
             \{k: -m < k < m\}\setminus \{0\} &m = 2, 3.\\
        \end{cases}
    \end{align*}

\noindent We will use $\setofcontent(m)$ to define the Young idempotents
using \emph{Lagrange interpolation}; see \cite[p504]{MURPHY1992492}\footnote{The $p_\frakt$ are written as $E_t$ in \cite{MURPHY1992492}, but there is a minor typo in the definition in \cite[p504]{MURPHY1992492}: in the product index, the set $\setofcontent(n)$ should be replaced by $\setofcontent(m)$.}.

\begin{definition}\label{def:youngidempotent}
For $\HH_n(q)$ semisimple and $\frakt \in \syt(n)$, define the \emph{Young idempotent} $p_\frakt$ to be 
 \begin{align*}
       p_\frakt := \prod_{m = 1}^n \quad \prod_{\substack{\mathfrak{d} \in \setofcontent(m) \sm \{\content_{\frakt, m}\} }} \quad \frac{J_m(q) - [\mathfrak{d} ]_q}{\content_{\frakt, m}(q) - [\mathfrak{d} ]_q}.
   \end{align*}
\end{definition}
  Lagrange interpolation guarantees the following properties, which can be found in \cite[p. 506]{MURPHY1992492}.

 \begin{prop}\label{lem:interaction-of-jucys-murphys-with-pts}
 The Young idempotents $p_\frakt$ for $\frakt \in \syt$ satisfy the following properties:
 \begin{enumerate}
     \item The collection of $p_{\frakt}$ for $\frakt \in \syt(n)$
form a family of mutually orthogonal, complete idempotents, meaning that 
     \[ p_{\frakt} \  p_{\frakq} = \delta_{\frakt, \frakq} p_{\frakt} \quad \quad \textrm{ and } \quad \quad \sum_{\frakt \in \syt(n)} p_{\frakt} = 1.\]
    \item Each $p_\frakt$ for $\frakt \in \syt(n)$ is a simultaneous (right) eigenvector for $J_m(q)$ for $1 \leq m \leq n$: 
    \[  p_\frakt \ J_m(q)   = \content_{\frakt, m}(q) \ p_\frakt.\] 
     \item The collection of $p_\frakt$ for all $\frakt \in \syt(n)$ diagonalize $J_n(q)$, so that $J_n(q)$ can be written as 
     \[J_n(q) = \sum_{\substack{\frakt \in \syt(n)}}\content_{\frakt, n}(q) \ p_\frakt.\] 
     \item If $\frakt \in \syt(\lambda),$ then as left $\HH_n(q)$-modules, \[\HH_n(q) \  p_{\frakt} \cong \left(S^\lambda\right)^\ast.\] 
 \end{enumerate}
 \end{prop}

The Young idempotents beautifully encode the connection between representations of $\HH_n(q)$ and $\syt(n)$ as follows. For $\frakt \in \syt(n)$, let $\frakt|_k$ be the subtableau of $\frakt$ obtained from restricting $\frakt$ to the boxes labeled $1, \cdots, k$ of $\frakt$, and let $\shape(\frakt|_k)$ be the shape of $\frakt|_k$. Then 
\[ \frakt|_k \in \syt \left(\shape(\frakt|_k)\right) \subseteq \syt(k), \] and we will think of each  $\frakt \in \syt(n)$ as being built by a nested sequence of tableaux:
\[ \frakt|_1 \subseteq \frakt|_2 \subseteq \cdots \subseteq \frakt|_n = \frakt. \] 

Crucially, each $p_\frakt \in \HH_n(q)$ can be built from a tower of inclusions in the same way. Algebraically, the nested tableaux correspond to the algebra embedding $\HH_{k}(q) \subseteq \HH_{k+1}(q)$ (which sends each $T_{s_i}$ to $T_{s_i}$). This idea is encapsulated in the \emph{Tower Rule} (Proposition \ref{lem:tower-rule}) below, which will allow us to move between Hecke algebras of different sizes in a precise way. This is essential to our inductive arguments.

Recall that $\{ p_\lambda: \lambda \vdash n\} $ is the collection of canonical central orthogonal idempotents that project onto the $S^\lambda$-isotypic component of $\HH_n(q)$. We sketch the proof of the Tower rule for completeness, though the result is not new; see for instance \cite[\S 1]{elias2017categorical}. 

\begin{prop}[Tower rule] \label{lem:tower-rule}
For any $\lambda \vdash n$, we have that 
\begin{equation}\label{tower1}
    p_\lambda = \sum_{\frakt \in \syt(\lambda)} p_{\frakt},
\end{equation}
and for $\frakt \in \syt(\lambda)$,
\begin{equation} \label{tower2}
p_{\frakt} = p_{\shape\left(\frakt\vert_1\right)} \ p_{\shape\left(\frakt\vert_2\right)} \ \cdots \ p_{\shape\left(\frakt\vert_{n-1}\right)}  \ p_{\shape\left(\frakt\vert_n\right)}. \end{equation}

\end{prop}
\begin{proof}
Since the isotypic decomposition of semisimple algebras is unique, the corresponding complete, orthogonal idempotent projectors $p_\lambda$ are unique. It is straightforward to check that
\[ \sum_{\frakt \in \syt(\lambda)} p_\frakt\]
form a family of complete, orthogonal idempotents, which implies that $\HH_n(q) \sum_{\frakt \in \syt(\lambda)} p_\frakt$ is a left submodule of $\HH_n(q)$ isomorphic to \[ \bigoplus_{\frakt \in \syt(\lambda)}\HH_n(q) p_\frakt \cong \bigoplus_{\frakt \in \syt(\lambda)}\left(S^\lambda\right)^\ast,\] which must be the $\lambda$-isotypic subspace. 
For the second claim, by induction on $n$ it suffices to show 
\[p_{\frakt'}\cdot p_{\lambda} = p_{\frakt}\]
where $\frakt \in \syt(\lambda)$ and $\frakt'= \frakt \vert_{n - 1}.$
By definition of $p_{\frakt}$, we can rewrite 
        \begin{align*}
        p_{\frakt}
        & = \left(\prod_{m = 1}^{n-1} \quad  \prod_{\mathfrak{d} \in \setofcontent(m) \sm \{\content_{\frakt', m}\}}\frac{J_m(q) - [\mathfrak{d}]_q}{\content_{\frakt', m}(q) - [\mathfrak{d}]_q}\right)\cdot \left(  \prod_{\mathfrak{d} \in \setofcontent(n) \sm \{\content_{\frakt, n}\}}\frac{J_{n}(q) - [\mathfrak{d}]_q}{\content_{\frakt, n}(q) - [\mathfrak{d}]_q}\right) \\
        & = p_{\frakt'} \left(  \prod_{\mathfrak{d} \in \setofcontent(n) \sm \{\content_{\frakt, n}\}}\frac{J_{n}(q) - [\mathfrak{d}]_q}{\content_{\frakt, n}(q) - [\mathfrak{d}]_q}\right).
        \end{align*}
        Combining this observation with the fact that $p_{\frakt'}$ is idempotent, we rewrite
        \begin{align}\label{eqn:pt'=ptpt'}
        p_{\frakt}= p_{\frakt'}\cdot p_{\frakt'}\left(  \prod_{\mathfrak{d} \in \setofcontent(n) \sm \{\content_{\frakt, n}\}}\frac{J_{n}(q) - [\mathfrak{d}]_q}{\content_{\frakt, n}(q) - [\mathfrak{d}]_q}\right)  = p_{\frakt'} \cdot p_{\frakt}.
        \end{align}

        Equation (\ref{eqn:pt'=ptpt'}) and orthogonality imply that 
        \[ p_{\frakt'} \cdot p_\frakq = p_{\frakt'} \cdot p_{\frakq \vert_{n - 1}} \cdot p_{\frakq} = 0 \] 
        for any $\frakq \in \syt(\lambda)$ with $\frakq \vert_{n - 1} \neq \frakt'.$ Since the only $\frakq \in \syt(\lambda)$ with $\frakq \vert_{n - 1} = \frakt'$ is $\frakq = \frakt,$ we can again apply equation (\ref{eqn:pt'=ptpt'}) to show 
\begin{align*}
    p_{\frakt'} \cdot p_{\lambda} = p_{\frakt'}  \sum_{\frakq\in \mathrm{SYT}(\lambda)}p_{\frakq}
   = p_{\frakt'}p_\frakt = p_{\frakt}.
\end{align*}
\end{proof}
\begin{example}\rm
Let \begin{align*}
    \frakt = \begin{ytableau}
            1 & 2 & 5\\
            3 & 4
        \end{ytableau}.
\end{align*}
Then, the shapes of the restricted tableaux $\frakt \vert_k$ are as follows.
\setlength{\extrarowheight}{.25cm}
\begin{center}
\begin{tabular}{c|ccccc}
  $k$   &  $1$ & $2$ & $3$ & $4$ & $5$\\ \hline
  
 $ \frakt \vert_k$  & $\begin{ytableau}
       1
   \end{ytableau}$  & $\begin{ytableau}
       1 & 2
   \end{ytableau}$  & $\begin{ytableau}
       1 & 2\\
       3
   \end{ytableau}$ & $\begin{ytableau}
       1 & 2\\
       3 & 4
   \end{ytableau}$  & $\begin{ytableau}
       1 & 2 & 5\\
       3 & 4
   \end{ytableau}$\\[1.25em] \hline
   $\shape \left(\frakt \vert_k \right)$ & $(1)$ & $(2)$ & $(2, 1)$ & $(2, 2)$ & $(3, 2)$
\end{tabular}
\end{center}
    Hence, by the tower rule:
    \begin{align*}
        p_{\frakt} = p_{(1)} \  p_{(2)} \  p_{(2,1)}  \ p_{(2,2)} \ p_{(3, 2)}.
    \end{align*}
\end{example}

We use the tower rule to extend our definition of Young idempotents to skew diagrams $\lambda \sm \mu$:

\begin{definition}
\label{def:skew-pt}
For $\frakt \in \syt(\lambda \sm \mu)$, let
\[
p_\frakt := p_{\shape\left(\frakt\vert_{|\mu|+1}\right)} \ p_{\shape\left(\frakt\vert_{|\mu|+2}\right)} \ \cdots \ p_{\shape\left(\frakt\vert_{|\lambda|} \right)}.
\]
\end{definition}

 It will also be useful to build a standard tableau in $\syt(\lambda)$ from elements of $\syt(\lambda \sm \mu)$ and $\syt(\mu)$.

\begin{definition}\label{defconstructtableau}
Given a tableau $\fraks \in \syt(\mu)$ and a skew tableau $\frakt \in \syt(\lambda \sm \mu)$, define $\frakt(\fraks) \in \syt(\lambda)$ to be the unique tableau for which
\[ \frakt(\fraks)|_{|\mu|} = \fraks \quad \quad \textrm{ and } \quad \quad \frakt(\fraks) \sm \fraks = \frakt.\]
\end{definition}

\begin{example}\rm
    If
    \begin{align*}
       \fraks =  \begin{ytableau}[*(lightgray)]
            1 & 2 & 5\\
            3& 4\\
        \end{ytableau} \quad \text{ and } \quad \frakt =  \begin{ytableau}
            \none & \none & \none & 7 & 9\\
            \none & \none & 8 & 10 \\
            6 & 11
        \end{ytableau}\ ,
    \end{align*}
    then 
    \begin{align*}
    \frakt(\fraks) =  \begin{ytableau}
           *(lightgray) 1 & *(lightgray) 2 & *(lightgray) 5 & 7 & 9\\
            *(lightgray) 3 & *(lightgray) 4 & 8 & 10 \\
            6 & 11
        \end{ytableau}.
    \end{align*}
\end{example}

\subsubsection{Specht modules and the seminormal basis}\label{subsec:spech-modules-seminormal-basis}

We will now use the Young idempotents to construct the seminormal basis for Specht modules of $\HH_n(q)$. 

In order to describe this basis and its properties, we will define a partial order on $\syt(\lambda)$ as follows. Recall the \emph{dominance} partial order on partitions of a fixed size, given by $\mu \leqdom \nu$ if for all $j,$
\[ \sum_{i = 1}^j \  \mu_i  \ \leq  \ \sum_{i = 1}^j \  \nu_i \] 

\begin{definition}[Dominance order on $\syt(\lambda)$] \label{def:dominanceorder}

 For any $\frakt, \frakq \in \syt(\lambda)$ we say
 $\frakq \leqdom \frakt$ if 
 \[ \shape \left(\frakq \vert_{k}\right) \leqdom \shape \left(\frakt \vert_k\right) \quad \quad \textrm{ for all} \quad \quad  1 \leq k \leq |\lambda|.  \]
\end{definition}
There is always a unique maximal element in $\syt(\lambda)$ with respect to $\leqdom$, which we denote by $\frakt^\lambda.$ The tableau $\frakt^\lambda$ is given by filling in a Young diagram of shape $\lambda$ with the entries $1, 2, \cdots, n$ starting with the box in the upper left corner, then continuing across each row left to right, proceeding from the top row to the bottom row.
\begin{example}\rm
We draw Hasse diagram for dominance order on $\syt(3,2).$ The top tableau is $\frakt^{(3,2)}$.

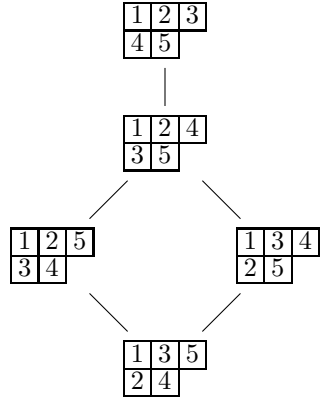
\begin{center}
    \begin{tikzpicture}
        \node (1) at (0, 4.5) {$\begin{ytableau}
            1 & 2 & 3\\ 4 & 5
        \end{ytableau}$};
        \node (2) at (0, 3) {$\begin{ytableau}
            1 & 2 & 4\\ 3 & 5
        \end{ytableau}$};
          \node (3) at (-1.5, 1.5) {$\begin{ytableau}
            1 & 2 & 5\\ 3 & 4
        \end{ytableau}$};
          \node (4) at (1.5, 1.5) {$\begin{ytableau}
            1 & 3 & 4\\ 2 & 5
        \end{ytableau}$};
           \node (5) at (0,0) {$\begin{ytableau}
            1 & 3 & 5\\ 2 & 4
        \end{ytableau}$};
        \draw (1)--(2);
        \draw (2) -- (3);
        \draw (2) -- (4);
        \draw (3)--(5);
        \draw(4)--(5);
    \end{tikzpicture}
\end{center}
\end{example}

 Definition \ref{def:dominanceorder} can easily be extended to an ordering on $\syt(\lambda \sm \mu)$ by $\leqdom$. We will use $\frakt^{\lambda \sm \mu}$ to denote the largest element of $\syt(\lambda \sm \mu)$ with respect to $\leqdom$. 

Given any $\frakt \in \syt(\lambda)$, one can define a word, $\word(\frakt),$ in the alphabet $\{ 1, 2, \cdots, \ell(\lambda)\}$ as follows. 

\begin{definition} Given $\frakt \in \syt(\lambda)$ with $|\lambda| = n$, define $\word(\frakt)$ to be the word \[ w_1\ w_2\cdots \  w_n, \]
where $w_i \in [\ell(\lambda)]$ has value given by the row in which $i$ appears in $\frakt$. 
\end{definition}

\begin{example}\rm
For $\lambda = (\lambda_1, \cdots, \lambda_k)$, we have that 
\[ \word(\frakt^\lambda) = (\underbrace{1, \cdots, 1}_{\lambda_1}, \underbrace{2, \cdots, 2}_{\lambda_2},\cdots, \underbrace{k, \cdots, k}_{\lambda_k}).    \]
\end{example}

By construction $\word(\frakt)$ has \textit{content} $\lambda$, which means it contains $\lambda_1$ $1$'s, $\lambda_2$ $2$'s, and so on. There is a natural right symmetric group action on any word ${\mathbf{w}} = w_1\ w_2 \cdots w_n$ by position, where $s_i$ swaps the positions of $w_i$ and $w_{i+1}$. Importantly, \cite[Chapter 4, Exercise 19, pg 67]{Mathas} shows this action generalizes
to a right action by $\HH_n(q)$ as well, defined as follows:
\begin{align} \label{eqn: word-action}
    \mathbf{w} \cdot T_{s_i} := \begin{cases}
        q\mathbf{w}, & \text{$w_i = w_{i + 1}$},\\
        \mathbf{w}s_i, & \text{$w_i < w_{i + 1}$},\\
        q\mathbf{w}s_i + (q - 1)\mathbf{w}, & \text{$w_i > w_{i + 1}.$}
    \end{cases}
\end{align}
Let $W^\lambda$ denote the $\CC$-span of all words of length $n$ with content $\lambda$. Note that by \eqref{eqn: word-action}, the $\HH_n(q)$-action on words preserves content, and so $W^\lambda$ is a $\HH_n(q)$-module.

We are at last ready to define Young's seminormal units:

\begin{definition}[Young seminormal units] \label{def:youngseminormalunits}
Given $\frakt \in \syt(\lambda)$, define the Young seminormal unit: 
\begin{align}
 w_{\frakt}:= \word(\frakt) \ p_{\frakt} \in W^\lambda.
 \label{eq.def:youngseminormalunits.wt=}
\end{align}
\end{definition}

The importance of the seminormal units is that they provide a particularly nice basis for the irreducible representations $S^\lambda$ of $\HH_n(q)$.
To state this, we recall the action of the symmetric group $\symm_n$ on the tableaux of shape $\lambda$ with entries $1,2,\ldots,n$ from the right by value --- that is, the entries of $\frakt \cdot \sigma$ are obtained by applying $\sigma^{-1}$ to the entries of $\frakt$.

\begin{theorem}[Dipper-James] \label{thm:dipperjamesyoungbasis}
Suppose $\HH_n(q)$ is semisimple. Then the following holds:
\begin{enumerate}
    \item The collection 
    \[ \{ w_{\frakt}: \frakt \in \syt(\lambda) \}\]
    give a $\CC$-basis for an irreducible representation $S^\lambda$ of $\HH_n(q)$. \medskip
    \item For every $1 \leq m \leq n$ and $\frakt \in \syt(\lambda)$,
     \begin{equation*}\label{eqn:JMscalingSN}
    w_\frakt \  J_m(q) = \content_{\frakt, m}(q) \  w_\frakt.\end{equation*}
    \item \label{thm:dipperjamesactiononseminormalunits}
   For $\frakt \in \syt(\lambda),$ $\frakq = \frakt \cdot s_i$ and $\rho_i:= \content_{\frakt, i}-\content_{\frakq, i},$ 
\begin{equation*}
    w_{\frakt} T_{s_i} = \begin{cases} q \ w_{\frakt} &  i \textrm{ and } i+1 \textrm{ are in the same row in } \frakt \bigskip \\
    -w_{\frakt} &  i \textrm{ and } i+1 \textrm{ are in the same column in } \frakt \bigskip \\
    -\frac{1}{[\rho_i]_q} w_{\frakt} +  w_{\frakq} &  \frakq \in \syt(\lambda) \textrm{ and } \frakq \ledom \frakt \bigskip \\ 
    - \frac{1}{[\rho_i]_q} \ w_{\frakt}  +  \frac{q[\rho_i+1]_q \ [\rho_i-1]_q}{([\rho_i]_q)^2} \ w_{\frakq} & \frakq \in \syt(\lambda) \textrm{ and } \frakt \ledom \frakq.
    \end{cases}
\end{equation*}
\end{enumerate}
\end{theorem}

\begin{remark}\rm
For a proof, see \cite[Prop 3.35 (ii) and Thm 3.36]{Mathas}. However,
Theorem \ref{thm:dipperjamesyoungbasis}(3) as stated in Mathas \cite[Thm 3.36 (ii), pg 44]{Mathas} has a minor typo in the final case, corrected above.
\end{remark}

\begin{example}\label{ex:triv-rep}\rm
The one-dimensional representation $S^{(n)}$ of $\HH_n(q)$ has basis element $w_{\frakt^{(n)}}$ and 
\[ w_{\frakt^{(n)}} \cdot T_w = q^{\ell(w)} w_{\frakt^{(n)}}. \]
We can see this as an example of Theorem \ref{thm:dipperjamesyoungbasis} (3), since $i$ and $i+1$ are always in the same row in 
\[ \frakt^{(n)} =  \begin{ytableau}
            1 & 2 & \cdot & \cdot & \cdot & n
        \end{ytableau}\]
Thus the \emph{character} of $T_w$ acting on $S^{(n)}$ is $q^{\ell(w)}$; see \cite{ram1991frobenius} for more on characters of $\HH_n(q)$. 

Similarly, the one-dimensional representation $S^{(1^n)}$ has basis element $w_{\frakt^{(1^n)}}$, and the action of $T_w \in \HH_n(q)$ on $S^{(1^n)}$ is given by: 
\[w_{\frakt^{(1^n)}} \cdot T_w = (-1)^{\ell(w)} w_{\frakt^{(1^n)}}. \]

This can also be deduced from Theorem \ref{thm:dipperjamesyoungbasis}(3), since for any $i$, one always has $i$ and $i+1$ in the same column in $\frakt^{(1^n)}$. 
\end{example}
\begin{example}\rm
Consider $i=2$, with  \[\frakt = \begin{ytableau}
    1 & 3 & 4\\
    2 & 5
\end{ytableau} \ \text{ so that } \frakt \cdot s_2 = \frakq = \begin{ytableau}
    1 & 2 & 4\\
    3 & 5
\end{ytableau}. \]
We see that $\frakt \ledom \frakq$, and $\rho_2 = \content_{\frakt, 2} - \content_{\frakq, 2} = -1 - 1 = -2$.
Thus the final case of Theorem \ref{thm:dipperjamesyoungbasis}(3) gives 
\begin{align*}
    w_{\frakt} \cdot T_{s_i} = -\frac{1}{[-2]_q}w_{\frakt} + \frac{q[-1]_q[-3]_q}{\left([-2]_q\right)^2}w_{\frakq}.
\end{align*} 
\end{example}

\begin{remark}\label{remark:mathasdictionary}\rm
Our presentation of $S^\lambda$ is slightly different (though equivalent) to the one presented by Mathas in \cite[Chapter 3]{Mathas}. While we use $w_{\frakt} = \word(\frakt) \ p_{\frakt}$, Mathas defines the corresponding basis elements of $S^\lambda$ as 
$m_{\frakt} \ p_{\frakt}$, where 
\[ m_{\frakt} = m_{\lambda} \  T_w,\]
with $w \in \symm_n$ being the minimal length permutation such that $\frakt  = \frakt^\lambda \cdot w$ and 
 \[m_{\lambda}:= \sum_{w \in \symm_\lambda}T_w\]
for the Young subgroup $\symm_\lambda \cong \symm_{\lambda_1} \times \symm_{\lambda_2} \times \cdots \times \symm_{\lambda_k}$. 

It is not hard to check that $W^\lambda$ is isomorphic to the right $\HH_n(q)$-module generated by the $m_{\frakt}$ for all \emph{row strict tableaux} $\frakt$, i.e. $\frakt$ whose entries increase across rows. 

The advantage of Mathas's approach is that the Specht modules he defines are actually in $\HH_n(q)$, and because of this he can construct a basis for the bimodule $(S^\lambda)^* \otimes S^\lambda$. The advantage of our approach is that the elements $\word(\frakt)$ are quite concrete, and work well with our inductive arguments.

\end{remark}

\begin{example}\label{ex:m1n-j}\rm
The element  
\[ m_{(1^j, n-j)} = \sum_{w \in \symm_{(1^j,n-j)}} T_w \]
will appear many times in subsequent sections. Note that $\symm_{(1^j,n-j)} \cong \symm_{n-j}$ in the generators  $s_{n-j+1}, \cdots, s_{n-1}$. Thus, as in Example \ref{ex:triv-rep}, for any $w \in \symm_{(1^j,n-j)}, $
\[ m_{(1^j,n-j)} T_w = q^{\ell(w)} m_{(1^j,n-j)}. \]

\end{example}

\section{Branching rules for $\HH_n(q)$ and the Horizontal Strip Lemma}\label{sec:branching-rules-Hecke}
We review the branching rules of $\HH_n(q)$. When $\HH_n(q)$ is semisimple, one can define a Frobenius characteristic map (Section \ref{sec:frob-char}). We prove the \emph{horizontal strip lemma} (Lemma \ref{horizontalstriplemma}) in Section \ref{subsec:horiz-strip}.

\subsection{Branching rules and the Frobenius characteristic map}\label{sec:frob-char}\label{section:pieri}

In the semisimple case, the branching rules (i.e. behavior of restriction and induction) for $\HH_n(q)$ mirror the symmetric group, as we now explain. 

Suppose $V$ is an $\HH_n(q)$-module. Then one can define $\res(V)$, the restriction of $V$ to $\HH_{n-1}(q)$, as the module $V$ viewed as an $\HH_{n-1}(q)$-module using the standard embedding $\HH_{n-1}(q) \hookrightarrow \HH_n(q)$.

The induction of $V$ to $\HH_{n+1}(q)$ is the $\HH_{n+1}(q)$-module 
\[ \ind(V): = V \otimes_{\HH_n(q)} \HH_{n+1}(q). \]

Given a partition $\lambda$, we say $\mu \lessdot \lambda$ if $\mu$ is a partition obtained from removing a single box from $\lambda$.

\begin{theorem}[Dipper-James]\label{thm:branchingrules}
When $\HH_n(q)$ is semisimple, for any irreducible $\HH_n(q)$-module $S^\lambda$, 
\[ \res(S^\lambda) \cong \bigoplus_{\mu \lessdot \lambda} S^\mu \quad \quad \quad \ind(S^\lambda) \cong \bigoplus_{\lambda \lessdot \nu} S^\nu.\]

\end{theorem}

More generally, for a subalgebra $B$ of an algebra $A,$ one can define the induction of a $B$-module $V$ to an $A$-module as \[ \ind_B^{A}\left(V\right) := V \otimes_{B}A.\] We will be most interested in the case $B = \HH_{j}(q) \otimes \HH_{n-j}(q)$ and $A = \HH_n(q).$

Note that $\HH_{j}(q) \otimes \HH_{n-j}(q)$ can be realized explicitly in $\HH_n(q)$ as the algebra generated by $T_{s_i}$ for $i \neq j$; write this realization as $\HH_{j, n-j}(q)$. Write $S^{\mu} \otimes S^{\nu}$ to be an irreducible representation of $\HH_{j, n - j}(q)$, where $S^\mu$ is a Specht module of $\HH_{j}(q)$ and $S^\nu$ is a Specht module of $\HH_{n-j}(q)$. 

\subsubsection{Frobenius characteristic map}
One of the most powerful tools to study the representation theory of the symmetric group is the \emph{Frobenius characteristic map} from the ring of virtual symmetric group representations to the ring of symmetric functions. When $\HH_n(q)$ is semisimple, there is an analogous map (\cite[\S 3.2]{krob1997noncommutative}). For a proof, see  \cite[Prop 1.2]{GOODMAN1990244}. 
Let \[{\sf{Rep}} [\HH(q)] = \bigoplus_{n \geq 0} {\sf{Rep}}[\HH_n(q)]\] be the graded ring of isomorphism classes of $\HH_n(q)$-modules, where addition corresponds to direct sum of representations, and the product of $M \in {\sf{Rep}} [\HH_m(q)]$ and $N \in {\sf{Rep}} [\HH_n(q)]$ corresponds to induction  
\[ \ind_{\HH_{m, n}(q)}^{\HH_{m + n}(q)}\left(M \otimes N\right). \] Write $\Lambda$ for the ring of symmetric functions, $s_\lambda$ for the Schur function indexed by a partition $\lambda$, and $h_n$ for the homogeneous symmetric function.  

\begin{prop}[Frobenius characteristic map]
    When $\HH_n(q)$ is semisimple, the map \begin{align*}
    \mathrm{ch}: {\sf{Rep}}[\HH(q)] &\to \Lambda\\ S^\lambda &\mapsto s_\lambda \\
    S^{(n)} & \mapsto h_n
\end{align*}
is a ring isomorphism.
\end{prop}
   We thus immediately obtain \emph{Pieri rules} for $\HH_n(q)$: 

\begin{cor}[Pieri rules for $\HH_n(q)$]\label{pierirules}
   When $\HH_n(q)$ is semisimple, 
    \begin{equation*} \ind_{\HH_{j, n- j}(q)}^{\HH_n(q)} \ \left( S^{\mu} \otimes S^{(n-j)}\right) \cong \bigoplus_{\substack{\nu \vdash n\\ \nu \sm \mu \textrm{ a horizontal strip}}} S^{\nu}.\end{equation*}
\end{cor}

\subsubsection{The operators $\Phi$}
Our proofs in Section \ref{section:constructionofeigenvectors} build concrete realizations of $\res(S^\lambda)$ and $\ind(S^\lambda)$.  We define concatenation operators $\Phi_\frakt$ and $\Phi_{\lambda \sm \mu}$ that will facilitate these constructions.

\begin{definition}\label{def:phi}\rm
Let $\mu \subseteq \lambda$ be partitions.
For any $\frakt \in \syt(\lambda \sm \mu)$, define a $\CC$-linear map $\Phi_{\frakt} : W^\mu \to W^\lambda$ by
\begin{equation}
\word(\fraks) \cdot  \Phi_{\frakt} := \word \left(\frakt(\fraks)\right)
\qquad \text{for all } \fraks \in \syt(\mu) .
\label{eq.def:phi.wordPhi}
\end{equation}
Note that the action of $\Phi_{\frakt}$ on a word $w_1 w_2 \cdots w_{|\mu|} \in W^\mu$ can also be described explicitly by
\[
w_1 w_2 \cdots w_{|\mu|} \cdot \Phi_{\frakt} = w_1 w_2 \cdots w_{|\lambda|} ,
\]
where $w_i$ for each $i > |\mu|$ is the number of the row of $\frakt$ that contains the entry $i$.

Furthermore, define a map $\Phi_{\lambda \sm \mu} : W^\mu \to W^\lambda$ by
\begin{align*}
	 \Phi_{\lambda \sm \mu} := \Phi_{\frakt^{\lambda \sm \mu}}.
\end{align*}
\end{definition}

\begin{example}\rm
Consider  \begin{align*}
       \fraks =  \begin{ytableau}
            1 & 2 & 5\\
            3& 4\\
        \end{ytableau} \quad \text{ and } \quad \frakt =  \begin{ytableau}
            \none & \none & \none & 7 & 9\\
            \none & \none & 8 & 10 \\
            6 & 11
        \end{ytableau}\ .
    \end{align*}
    Then, 
    \begin{align*}
    \word(\fraks) \ \Phi_\frakt = \word\left(\frakt(\fraks)\right) = 11221312123,
    \end{align*}
    whereas 
    \begin{align*}
    \word(\fraks) \ \Phi_{(5, 4,2) \sm (3,2)} = \word\left(\frakt^{(5, 4,2) \sm (3,2)}(\fraks)\right) = 11221112233.
    \end{align*}
\end{example}

The following lemma is straightforward to check and will be used throughout our paper. 

\begin{lemma}\label{lem:commuting-phi}
    Let $\mu \subseteq \lambda$ be partitions with $\left|\mu\right| = j$ and $\left|\lambda\right| = n$.
	Then, for each $\frakt \in \syt\left(\lambda \sm \mu\right)$, 
 \[ \Phi_\frakt : W^\mu \to W^\lambda \] 
 is an $\HH_j(q)$-module morphism, where $\HH_j(q)$ acts on $W^\lambda$ via the canonical embedding $\HH_j(q) \hookrightarrow \HH_n(q)$.
	In other words,
	\[
	{\bf w} \ \Phi_\frakt \cdot x = \left({\bf w}\cdot x\right) \Phi_\frakt
	\qquad \text{for all ${\bf w} \in W^\mu$ and $x \in \HH_j(q)$.}
	\]
\end{lemma}

We will also repeatedly make use of the following expression for $w_{\frakt(\fraks)}$.
\begin{lemma}\label{lem:wts}
Let $\mu \subseteq \lambda$ be partitions. Let $\fraks \in \syt(\mu)$ and $\frakt \in \syt(\lambda \sm \mu)$. Then, $w_{\frakt(\fraks)} = w_\fraks \ \Phi_\frakt \ p_\frakt$.
\end{lemma}

\begin{proof}
Combining \eqref{tower2} with \eqref{def:skew-pt}, we see that
$p_{\frakt(\fraks)} = p_\fraks \ p_\frakt$.
Since \eqref{eq.def:phi.wordPhi} yields $\word \left(\frakt(\fraks)\right) = \word(\fraks) \ \Phi_{\frakt}$, applying \eqref{eq.def:youngseminormalunits.wt=} yields
\[
w_{\frakt(\fraks)}
= \word(\frakt(\fraks)) \ p_{\frakt(\fraks)}
= \word(\fraks) \ \Phi_{\frakt} \ p_\fraks \ p_\frakt.
\]
By Lemma~\ref{lem:commuting-phi}, we can move the $\Phi_{\frakt}$ past the $p_\fraks$, thus rewriting the above as
\[
w_{\frakt(\fraks)}
= \word(\fraks) \ p_\fraks \ \Phi_{\frakt} \ p_\frakt
= w_\fraks \ \Phi_\frakt \ p_\frakt
\]
by \eqref{eq.def:youngseminormalunits.wt=} again.
\end{proof}

Importantly, the operators $\Phi$ behave well with restriction.
\begin{prop}\label{prop:restrictionrewrite}
Let $\frakt \in \syt(\lambda)$ and $\frakt' = \frakt|_{n-1}$. Then, $w_{\frakt}$ can be rewritten as 
\begin{equation}
w_{\frakt} = w_{\frakt'} \ \Phi_{\lambda \sm \shape(\frakt')} \ p_{\lambda}.
\label{eq:restrictionrewrite-wt}
\end{equation}

\end{prop}
\begin{proof}
    This follows from the fact that 
    \begin{align*} w_{\frakt} = \word(\frakt) \ p_{\frakt} &= \word(\frakt')  \ \Phi_{\lambda \sm \shape(\frakt')}  \ p_{\frakt} \\
    &= \word(\frakt') \  \Phi_{\lambda \sm \shape(\frakt')}  \ p_{\frakt'}  \ p_{\lambda}  \tag{Proposition \ref{lem:tower-rule}: \eqref{tower2}}\\
    & =  \word(\frakt') \ p_{\frakt'} \ \Phi_{\lambda \sm \shape(\frakt')} \   p_{\lambda}  \tag{Lemma \ref{lem:commuting-phi}}\\
    &= w_{\frakt'} \  \Phi_{\lambda \sm \shape(\frakt')} \  p_{\lambda}.
    \end{align*}
\end{proof}

The maps $\Phi$ also provide a method of constructing induced representations, as we will see in Section \ref{subsec:horiz-strip}.

\subsection{The horizontal strip lemma}\label{section:factorization}\label{subsec:horiz-strip}

Our $\RR_n(q)$-eigenvectors will be constructed from products of $\B_i(q)$ for $i \in (j,n]$ as $j$ varies. Write such a product as $\B_{n, n-j}(q)$.

\begin{definition}\label{def:Cjn}
For any $j \in [0, n]$, we set
\[\B_{n, n-j}(q):= \B_{j + 1}(q) \cdots \B_{n}(q).\]
\end{definition}

We will first show that $\B_{n, n-j}(q)$ can be factored in a useful way (Proposition \ref{prop:factoring-Cj-symm-coset}). We will apply this factorization to prove the \emph{Horizontal Strip Lemma} (Lemma \ref{horizontalstriplemma}), which will be used later to characterize when certain eigenvectors in our recursive construction are zero.

For a permutation $w \in \symm_n,$ recall that $\ell(w)$ is the \textit{Coxeter length} of $w.$ It is well-known (see, for instance, \cite[Section 2.4]{bjorner2005combinatorics}) that each right coset $\symm_\alpha w \in \symm_\alpha \sm \symm_n$ has a unique representative $w' \in \symm_\alpha w$ such that $\ell(w') < \ell(u)$ for all other $u \in \symm_\alpha w$. 
We denote the set of these minimal-length coset representatives as $X_\alpha$ and define the corresponding sum \[x_\alpha:= \sum_{w \in X_\alpha}T_w.\] 

Our first goal is to prove the following:
\begin{prop}\label{prop:factoring-Cj-symm-coset}
The element $\B_{n, n-j}(q) := \B_{j+1}(q) \ \B_{j+2}(q) \cdots \B_{n-1}(q) \ \B_n(q)$ factors as 
    \[ \B_{n, n-j}(q) = m_{(1^j, n - j)} \ x_{(j, n - j)}.\]
\end{prop}

We record here a few Coxeter-theoretic facts from \cite{bjorner2005combinatorics} which will be useful for this goal, translated to be statements about right cosets, since this is the setting relevant to us. 

\begin{prop} \label{lem:factoring-minimal}\label{lem:factoring-prods-of-xn-11}
Let $w \in \symm_n$. 
\begin{enumerate}
    \item \cite[Proposition 2.4.4]{bjorner2005combinatorics}: For any composition $\alpha$ of $n$, we can \textit{uniquely} factor $w$ as 
    \[ w = u \cdot v,\]
    where $u \in \symm_\alpha$ and $v \in X_\alpha.$ For this factorization,  $\ell(u) + \ell(v) = \ell(w)$.
    \item \cite[Corollary 2.4.6]{bjorner2005combinatorics}:
     Furthermore, $w$ can be written uniquely in the form \[w = y_1 \cdot y_2 \cdots y_{n - 1},\] where $y_i \in X_{(i,1)} \subseteq \symm_{i+1}$. For this factorization, $\ell(w) = \ell(y_1) + \ell(y_2) + \cdots + \ell(y_{n - 1}).$
\end{enumerate}
\end{prop}
The elements of $X_\alpha$ have a nice characterization in terms of descents. Recall that $1 \leq i \leq n - 1$ is a \textit{right} descent of a permutation $w \in \symm_n$ if $w(i) > w(i + 1).$ In contrast, $i$ is a \textit{left} descent of $w$ if $\ell(s_iw) < \ell(w);$ note that $i$ is a left descent of $w$ precisely if $i$ is a right descent of $w^{-1}.$ Define a subset $J(\alpha) \subseteq [n - 1]$ associated to the composition $\alpha = (\alpha_1, \alpha_2, \cdots, \alpha_k)$ as
\begin{align*}
   J(\alpha):= \{\alpha_1, \alpha_1 + \alpha_2, \cdots, \alpha_1 + \alpha_2 + \cdots + \alpha_{k - 1}\}.
\end{align*}
As explained in \cite[p41]{bjorner2005combinatorics}, the elements in $X_\alpha$ are precisely the permutations in $\symm_n$ whose left descents are contained in the set $J(\alpha).$ We will use this characterization to first realize $\B_{n, n-j}(q)$ as the element $x_{(j, 1^{n - j})}.$ 

\begin{lemma}\label{sequenceofbottomtorandoms}
For any  $q \in \CC$ the element $\B_{n, n-j}(q)=\B_{j+1}(q) \ \B_{j+2}(q) \cdots \B_{n-1}(q) \ \B_n(q)$ can be rewritten as \[\B_{n, n-j}(q) = x_{(j, 1^{n - j})}.\] 
\end{lemma}

\begin{proof}
Observe that $X_{(k, 1)} = \{s_{k}s_{k - 1}\cdots s_i: 1 \leq i \leq k + 1\}.$ Hence, $\B_{k + 1}(q) = x_{(k, 1)}$ and the terms in the expansion of $\B_{n, n-j}(q) = \B_{j + 1}(q)\cdots \B_n(q)$ biject with tuples $(y_{j}, \cdots, y_{n - 1})$ where $y_k \in X_{(k, 1)}.$ Notice that $1 \in X_{(k, 1)}$ for any $k,$ so by Proposition \ref{lem:factoring-prods-of-xn-11}(2), the expressions of the form $y_{j} y_{j + 1}\cdots y_{n - 1}$ are all  distinct and length-additive, implying that 
\[ T_{y_{j}\cdots y_{n - 1}} = T_{y_{j}} \ T_{y_{j + 1}}\cdots T_{y_{n - 1}}. \]

To complete the proof, it suffices to explain that there is an equality of sets \[\left\{y_{j}\ldots y_{n - 1}: (y_{j}, \ldots, y_{n - 1}) \in X_{(j, 1)} \times \ldots \times X_{(n - 1, 1)}\right\} = X_{(j, 1^{n - j})}.\]
Since \[\frac{|\symm_n|}{\left \vert \symm_{(j, 1^{n - j})}\right \vert} = n \cdot (n - 1) \cdots (j + 1),\] the two sets have the same size, so it is enough to show that the left descent set of each $y_{j}\cdots y_{n - 1}$ is contained in $J(j, 1^{n - j}) = \{j, j + 1, \cdots, n - 1\}.$ This is equivalent to proving that for $i < j,$ the element $s_i$ is not a left descent of $y_j\cdots y_{n - 1}.$

To this end, fix $i < j.$
Since $s_i \in X_{(i,1)}$,
it follows from Proposition \ref{lem:factoring-prods-of-xn-11}(2) that 
\[ \ell(s_i y_{j}\cdots y_{n - 1}) = 1 + \ell(y_{j}\cdots y_{n - 1}). \] 

Thus, $s_i$ is not a left descent of $y_{j}\cdots y_{n-1}$, as desired. 
\end{proof}

We are now able to prove Proposition \ref{prop:factoring-Cj-symm-coset}.

\begin{proof}[Proof of Proposition \ref{prop:factoring-Cj-symm-coset}]
 By Lemma \ref{sequenceofbottomtorandoms}, it suffices to prove that \[x_{(j, 1^{n - j})} = m_{(1^j, n - j)}\cdot x_{(j, n - j)}.\] Using the descent interpretation of $X_{(j, 1^{n - j})}$, note that \[X_{(j, 1^{n - j})} = \{w \in \symm_n: w^{-1}(1) < w^{-1}(2) < \cdots < w^{-1}(j)\}.\]

Suppose $w \in X_{(j, 1^{n - j})}$.
By Proposition \ref{lem:factoring-minimal}(1), there is a unique way to factor $w =u  \cdot v $ where $u \in \symm_{(1^j, n - j)}$ and $ v \in X_{(1^j, n - j)}$ with $\ell(w) = \ell(u) + \ell(v)$. Since $v \in X_{(1^j, n - j)}$, the left descents of $v$ are contained in $J(1^j, n - j) = \{1, 2, \cdots,j\},$ meaning $ v$ has the property that 
\[ v^{-1}(j+1) < v^{-1}(j + 2) < \cdots < v^{-1}(n).\]
However, since $w^{-1}(1) < \cdots < w^{-1}(j)$
and $u$ fixes each element of $\{ 1, 2, \cdots, j \}$, it follows that $v$ \textit{also} satisfies 
\[ v^{-1}(1) < v^{-1}(2) < \cdots < v^{-1}(j). \]
Therefore, $v \in X_{(j, n - j)}.$ Thus each element $w \in X_{(j, 1^{n - j})}$ can be factored uniquely as $u \cdot v$ where $u \in \symm_{(1^j, n - j)}$ and $v \in X_{(j, n - j)}.$ 

Since $\lvert X_{(j, 1^{n - j})}\rvert = \lvert\symm_{(1^j, n - j)}\rvert \cdot \lvert X_{(j, n - j)}\rvert,$ we can conclude that 

\[ x_{(j, 1^{n - j})} =  \sum_{\substack{u \in \symm_{(1^j, n - j)}\\ v \in X_{(j, n - j)}}} T_{uv}=  \left(\sum_{\substack{u \in \symm_{\left(1^j, n - j\right)}}}T_u \right)\cdot \left(\sum_{v \in X_{(j, n - j)}}T_v\right) = m_{\left(1^j, n - j\right)} \cdot x_{(j, n - j)}. \] 
\end{proof}

We will use Proposition \ref{prop:factoring-Cj-symm-coset} to prove the \emph{Horizontal Strip Lemma} (Lemma \ref{horizontalstriplemma}), which shows that when $\HH_n(q)$ is semisimple, the element $\B_{n, n-j}(q)$ behaves like the induction operators discussed in Section \ref{section:pieri}. This will eventually explain why our eigenvectors in Theorem \ref{thm:intromainthm} are indexed by horizontal strips.

\begin{lemma}[Horizontal strip lemma]\label{horizontalstriplemma}\label{cor:nonhorizontalstripgives0}
    Assume $\HH_n(q)$ is semisimple. Let $\lambda$ be a partition of $n$ and $\mu$ a partition of $0 \leq j \leq n$ such that  $\mu \subseteq \lambda$. Then for any $\frakt \in \syt(\lambda \sm \mu )$,
    \[ (S^\mu) \ \Phi_{\frakt} \ \B_{n, n-j}(q)\] is contained in a $\HH_n(q)$-module isomorphic to a quotient of \[ \bigoplus_{\substack{\nu \vdash n \\ \nu \sm \mu \textrm{ is a horizontal strip}}} S^\nu. \]
    Thus, if $\lambda \sm \mu$ is not a horizontal strip, then for any $u \in S^\mu$
        \[u \  \Phi_{\frakt}  \  \B_{n, n-j}(q) \ p_{\lambda} = 0. \]
\end{lemma}

\begin{proof}

We shall prove the first statement by
\begin{enumerate}
    \item  Showing $S^\mu \ \Phi_\frakt \ \B_{n, n-j}(q)$ is contained in a certain $\HH_n(q)$-module we call $M,$ 
    \item Defining an $\HH_n(q)$-module $N$ which is isomorphic to 
    \[ \bigoplus_{\substack{\nu \vdash n \\ \nu \sm \mu \textrm{ is a horizontal strip}}} S^\nu, \]  
   \item explaining that $M$ is isomorphic to a quotient of $N.$
\end{enumerate} 

We begin with item (1). Observe that $S^\mu\ \Phi_\frakt \  m_{(1^j, n - j)}$ is closed under the right action of $\HH_{j, n - j}(q).$ In fact, it is straightforward to check that as right $\HH_{j, n - j}(q)$-modules,
\begin{equation}\label{eqn:submodule}
   S^\mu \ \Phi_\frakt \  m_{(1^j, n - j)} \cong S^\mu \otimes S^{(n-j)}. 
\end{equation}
This follows by noting that any element of $\HH_{j, n - j}(q)$ can be written as $T_uT_v$ where $u \in \symm_{(j, 1^{n - j})}$  and $v \in \symm_{(1^j, n - j)}.$ Since $T_u$ commutes with $\Phi_{\frakt}\  m_{(1^j,n-j)}$, by Example \ref{ex:triv-rep},
\[  S^\mu \ \Phi_\frakt \  m_{(1^j, n - j)} \cdot T_{u}T_v = \left( S^\mu \cdot T_u  \right) \left( \Phi_{\frakt} m_{(1^j,n-j)} \cdot T_v\right) =  \left( S^\mu \cdot T_u  \right) \left( \Phi_{\frakt} \  m_{(1^j,n-j)}  \ q^{\ell(v)} \right).  \]

 Define a right $\HH_n(q)$-module $M$ by
 \[ M:= S^\mu \ \Phi_\frakt \  m_{(1^j, n - j)} \ \HH_n(q). \]

Proposition \ref{prop:factoring-Cj-symm-coset} shows that $\B_{n, n-j}(q) = m_{(1^j, n - j)}x_{(j, n-j)}$. Thus, for each $u \in S^\mu$, we have
  \[u \  \Phi_\frakt \  \B_{n, n-j}(q) = u\ \Phi_\frakt \  m_{(1^j, n - j)} \ x_{(j, n-j)} \in M,\] 
completing item (1).

We now explain item (2). Set $N$ to be the following right $\HH_n(q)$-module:
\[ N := S^\mu \ \Phi_\frakt \  m_{(1^j, n - j)} \otimes_{\HH_{j, n - j}(q)} \HH_n(q)
= \ind_{\HH_{j,n-j}(q)}^{\HH_n(q)} \left( S^\mu \ \Phi_\frakt \  m_{(1^j, n - j)} \right) . \]
Using \eqref{eqn:submodule}, we thus obtain
\[
N \cong \ind_{\HH_{j,n-j}(q)}^{\HH_n(q)} \left( S^{\mu} \otimes S^{(n-j)} \right)
\cong \bigoplus_{\substack{\nu \vdash n \\ \nu \sm \mu \textrm{ is a horizontal strip}}} S^\nu
\]
(by Corollary \ref{pierirules}).

It now suffices to prove item (3). Observe that the map linearly extending
\begin{align*}
    N &\longrightarrow M\\
    x \otimes T_v &\longmapsto x \cdot T_v
\end{align*}
is surjective and commutes with the right action of $\HH_n(q).$ Hence, $M$ is isomorphic to a quotient of $N,$ completing item (3) and thus the first claim. 

 For the second claim, if $\lambda \sm \mu$ is not a horizontal strip, then $S^\lambda$ does not appear in the irreducible decomposition of $N$. Since $p_\lambda$ is the projector onto the $\lambda$-isotypic component, this forces $ u \  \Phi_{\frakt }  \  \B_{n, n-j}(q) \ p_{\lambda}= 0.$ 

    \end{proof}

\section{Flags of $\mathbb{F}_q^n$ and the spectrum of $\B_n^*(q)$}\label{sec:flagintroduction}
The goal of this section is to relate the
$\B_n(q)$ and $\B_n^*(q)$ to Markov chains on flags over a finite field $\mathbb{F}_q^n$, and to use this connection to compute the spectrum of both operators. We also obtain the $\HH_n(q)$-module structure on the eigenspaces of $\B_n^*(q)$ whenever $\HH_n(q)$ is semisimple. These results will serve as the essential base case of our analysis of the full spectrum of $\RR_n(q)$ in Section \ref{section:maintheorem}.

We first define in Section \ref{section:kernelofrr} the \emph{derangement representation} of $\HH_n(q)$. We then explain the general connection between $\HH_n(q)$ and flags over a finite field in Section \ref{section:connectionheckeandflags}, and use this in Section \ref{section:r2tflagstohecke} to link  $\B_n^*(q)$ to a Markov chain on flags introduced by Brown \cite{BrownOnLRBs} and studied in \cite{brauner2023invariant}. Finally, we apply this analysis in Section \ref{sec:b2r} to show that the kernel of $\B_n^*(q)$ carries the derangement representation, and compute the characteristic polynomials of $\B_n^*(q)$ and $\B_n(q)$.

\subsection{The Derangement representation}\label{section:kernelofrr}

Recall that an integer $i \in [n-1]$ is a \textit{descent} of $\frakt \in \syt$ if $i + 1$ appears south and weakly west of $i$. We write ${\sf Des}(\frakt)$ for the set of descents of $\frakt$. Then $\frakt \in \syt(n)$ is a \emph{desarrangement tableau} if the smallest element of $[n] \setminus {\sf Des}(\frakt)$ is even. Write 
\[ \D_n:= \{ \frakt \in \syt(n): \frakt \textrm{ is a desarrangement tableau} \}.\]
We will keep track of the number of desarrangement tableaux of shape $\mu$ by
\[ d^\mu:= \# \{ \frakt \in \D_n: \shape(\frakt) = \mu  \} .\]

Let $d_n$ be the number of \textit{derangements} permutations, or permutations with no fixed points, in $\symm_n.$ Then work of D\'esarm\'enien and Wachs \cite{FrenchDesarmenienWachs} shows that 
\[d_n = \sum_{\mu \vdash n}d^\mu f^\mu.\] 
(For an explicit proof, evaluate \cite[Proposition 2.3]{ReinerWebb} at the identity.)

\begin{example}\rm The desarrangement tableaux in $\D_5$ are shown below.\\
\begin{center}
\ytableausetup{smalltableaux}
\begin{tabular}{c|c|c|c|c}
   \begin{ytableau} 1 & 3 & 4 & 5\\ 2 \end{ytableau}  &  \begin{ytableau} 1 & 3 & 4\\ 2 & 5 \end{ytableau} \ \ \ \begin{ytableau} 1 & 3 & 5\\ 2 & 4\end{ytableau} & \begin{ytableau} 1 & 3 & 4\\ 2\\ 5\end{ytableau} \ \ \ \begin{ytableau} 1 & 3 & 5\\ 2\\ 4\end{ytableau} &  \begin{ytableau} 1 & 3\\ 2 & 4\\ 5\end{ytableau} \ \ \ \begin{ytableau} 1 & 3\\ 2 & 5\\ 4\end{ytableau}  & \begin{ytableau} 1 & 3\\ 2\\ 4 \\ 5\end{ytableau} \ \ \ \begin{ytableau} 1 & 5\\ 2\\ 3 \\ 4\end{ytableau}\\[3em]
   $d^{(4,1)} = 1$  & $d^{(3,2)} = 2$ & $d^{(3,1,1)} = 2$ & $d^{(2,2,1)} = 2$ & $d^{(2,1,1,1)} = 2.$
\end{tabular}
\end{center}
Hence, in total, \[d_5 = \sum_{\mu \vdash 5}d^\mu f^\mu = 1 \cdot 4 + 2 \cdot 5 + 2 \cdot 6 + 2 \cdot 5 + 2 \cdot 4 = 44.\]
\end{example}

We define a representation from the set $\D_n$ as follows.

\begin{definition}[Derangement representation]\label{def:desrep}
Let $\HH_n(q)$ be semisimple. 
The derangement representation of $\HH_n(q)$ is the left-module
\begin{align*}
\DD_n(q) := \bigoplus_{\frakt \in \D_n} \left(S^{\ \shape(\frakt)}\right)^\ast.
\end{align*}
Equivalently, $\DD_n(q)$ is the representation satisfying 
\begin{equation}\label{eq:frobeniusderangement} \sum_{j = 0}^n \mathrm{ch} \left(\DD_j(q)\right)h_{n-j} = h_{1^n}, \ \ \mathrm{ch} \left(\DD_0(q)\right) = 1.\end{equation}
\end{definition}
In addition to the characterizations in Definition \ref{def:desrep}, there are several other equivalent ways to describe $\DD_n(q)$ in terms of symmetric functions; see \cite[Prop. 3.1]{brauner2023invariant}.

 By construction, the multiplicity of $(S^\mu)^*$ in $\DD_n(q)$ is $d^\mu$ and the total dimension of $\DD_n(q)$ is $d_n.$ We will soon show that when $q \in \R_{>0}$, the kernel of $\RR_n(q)$ carries the derangement representation. 

\subsection{Connection between $\HH_n(q)$ and flags over a finite field}  \label{section:connectionheckeandflags}
We will now discuss an alternative perspective of the Hecke algebra---in fact its original motivation---which realizes $\HH_n(q)$ as the centralizer of the action of $\GL{n}$ on the space of complete flags over $\FF_q^n$, where $\FF_q$ is the field with $q$ elements. We will summarize the relevant aspects of this story; for a comprehensive study, see Halverson--Ram \cite{Halverson-Ram}. 

When working in this setting, we set $q$ to be the power of a prime number, which we shall denote by $q=p^m$. Let 
\[ F_{\bullet} = ( F_1 \subset \cdots \subset F_{n-1} \subset F_n = \FF^n_q)\]
be a complete flag of $\FF_q^n$ with $\dim(F_i) = i$, and write the collection of all such flags as $\F_n$. We will be interested in their $\CC$-span $\CC[\F_n]$, upon which $\GL{n}$ acts diagonally:
\[ g \cdot F_{\bullet} = (g \cdot F_1 \subset \cdots \subset g \cdot  F_{n-1} \subset g \cdot F_n = \FF^n_q). \]

There is a $\GL{n}$-equivariant bijection between $\F_n$ and the coset space $\GL{n} / B$, where $B$ is the Borel subgroup of upper triangular matrices of $\GL{n}$; see \cite[\S 1]{Halverson-Ram} for details.
This induces an isomorphism $\CC[\F_n] \cong \CC[\GL{n} / B]$ of $\GL{n}$-representations. 

The subset of irreducible representations of $\GL{n}$ which appear as the irreducible constituents from the action of $\GL{n}$ on $\CC[\F_n]$ were studied by Steinberg \cite{steinberg1951geometric} and are called the \emph{unipotent representations}; more explanation can be found in \cite[\S 3.2]{brauner2023invariant}. Each irreducible unipotent representation is indexed by a partition, denoted as $G^\lambda$, with dimension $f^{\lambda}(q):= \dim(G^\lambda)$  given by the $q$-\emph{hook formula}, see \cite[Equation (4.2)]{LewisReinerStanton}. In fact, the space $\CC[\F_n]$ decomposes as $\GL{n}$-representation in an analogous way to $\HH_n(q)$:
\[ \CC[\F_n] \cong \bigoplus_{\lambda \vdash n} \left(G^\lambda \right)^{f^\lambda}.\]

In this setting, one can define $\HH_n(q)$  in terms of $B$-double cosets of $\GL{n}$. We omit this description, but note that it implies there is an $\HH_n(q)$-action on $\CC[\GL{n} / B] \cong \CC[\F_n]$ as follows. 

\begin{definition}\label{def:heckeactiononflags}
When $q = p^m,$ there is a right action of $\HH_n(q)$ on $\CC[\F_n]$ by 
\[ F_\bullet \cdot  T_{s_i} = \sum_{\substack{G \neq F_i \\ \dim(G) = i \\ F_{i-1} \subset G \subset F_{i+1}}} (F_1 \subset F_2 \subset \cdots \subset  F_{i-1} \subset G \subset F_{i+1} \subset \cdots \subset F_n = \FF_q^n).\] 
\end{definition}

Importantly, the right action by $\HH_n(q)$ on $\CC[\F_n]$ commutes with the left action by $\GL{n}$, and so we obtain a bimodule decomposition of $\CC[\F_n]$, see \cite[p.253-254]{Halverson-Ram}. Write 
\[ \End_{\GL{n}}(\CC[\F_n]) \]
for the centralizer algebra of the action of $\GL{n}$ on $\CC[\F_n]$.

\begin{theorem}[Double centralizer theorem] \label{thm:doublecentralizerthm}
When $q = p^m,$ 
\[ \HH_n(q) \cong \End_{\GL{n}}(\CC[\F_n]).\]
Moreover, there is a $\GL{n} \times \HH_n(q)$-bimodule decomposition of $\CC[\F_n]$
\[ \CC[\F_n] \ \  \cong  \ \ \bigoplus_{\lambda \vdash n}  \ \ G^\lambda \otimes S^\lambda, \]
where $G^\lambda$ is an irreducible \emph{unipotent} representation of $\GL{n}$ and $S^\lambda$ is a Specht module of $\HH_n(q)$.
\end{theorem}

\begin{example}[Derangement representation of $\GL{n}$]\rm
Just as one can define $\DD_n(q)$ as a representation of $\HH_n(q)$, there is a natural \emph{derangement representation} $\widetilde{\DD}_n(q)$ of $\GL{n}$, defined as the analogous sum over unipotent representations:
\[  \widetilde{\DD}_n(q):=\bigoplus_{\frakt \in \D_n} G^{\ \shape(\frakt)}.
 \]
 This will become relevant in Section \ref{section:r2tflagstohecke}, when we study a Markov chain on $\CC[\F_n]$ whose kernel carries $\widetilde{\DD}_n(q)$.
\end{example}

As in the case of the bimodule decomposition of $\HH_n(q)$ in \eqref{eq:leftrightbimoduledecomposition}, Theorem \ref{thm:doublecentralizerthm} allows us to deduce information about the actions of elements of $\HH_n(q)$ on $\CC[\F_n]$.

\begin{cor}\label{cor:eigenvaluesinbimoduledecomposition}
Consider $\widetilde{\varphi} \in \End_{\GL{n}}(\CC[\F_n]),$ and suppose that for $\varphi \in \HH_n(q),$ 
\[\widetilde{\varphi} \cdot F_{\bullet} =  F_{\bullet} \cdot \varphi   \]
for every $F_{\bullet} \in \F_n$. Then, if $\widetilde{\varphi}$ acts semisimply on $\CC[\F_n]$, Theorem \ref{thm:doublecentralizerthm} implies the following:
    \begin{enumerate}
    \item The element $\varphi$ acts semisimply on $\HH_n(q).$\normalcolor \smallskip
          \item If the $\eigen$-eigenspace of $\widetilde{\varphi}$ acting on $\CC[\F_n]$ has irreducible decomposition as a $\GL{n}$-module
    
         \[  \bigoplus_{\lambda \vdash n}  \ \ \underbrace{G^{\lambda} \oplus G^{\lambda} \oplus \cdots \oplus G^{\lambda} }_{m_\eigen(\lambda)}  \]
        then the right $\eigen$-eigenspace of $\varphi$ acting on $\HH_n(q)$ decomposes as a left $\HH_n(q)$-module 

         \[  \bigoplus_{\lambda \vdash n}  \ \ \underbrace{(S^{\lambda})^{*} \oplus (S^{\lambda})^{*} \oplus \cdots \oplus (S^{\lambda})^{*} }_{m_\eigen(\lambda)}.  \] 
      
        \item With notation as (2) above, the set of eigenvalues of $\widetilde{\varphi}$ acting on $\CC[\F_n]$ will equal that of $\varphi$ acting by right multiplication on $\HH_n(q)$. The multiplicity of the eigenvalue $\eigen$ in $\varphi$ acting on $\HH_n(q)$ will be 
        \[ \sum_{\lambda \vdash n} f^\lambda \  m_\eigen(\lambda) = \sum_{\lambda \vdash n} \dim(S^\lambda) \ m_\eigen(\lambda),\]
        while the multiplicity of $\eigen$ from $\widetilde{\varphi}$ acting on $\CC[\F_n]$ is 
        \[  \sum_{\lambda \vdash n} f^{\lambda}(q) \ m_\eigen(\lambda) = \sum_{\lambda \vdash n} \dim(G^\lambda) \ m_\eigen(\lambda). \]
  
    \end{enumerate}  
\end{cor}

We will use Corollary \ref{cor:eigenvaluesinbimoduledecomposition} in Section \ref{section:r2tflagstohecke} to translate properties between our operators in $\HH_n(q)$ and a flag analog of random-to-top introduced by Brown \cite{BrownOnLRBs} and studied in \cite{brauner2023invariant}.

\subsection{Random-to-top on flags}
\label{section:r2tflagstohecke}
In \cite[Section 5.2, for $w_\ell=1$]{BrownOnLRBs}, Brown introduced a $q$-analogue of $\B_n^*$ which acts on $\CC[\F_n]$. 

\begin{definition}\label{def:xq}
    Define $x^{(q)} \in \End_{\GL{n}}(\CC[\F_n])$ to be the $\CC$-linear map given by
    \begin{align*}
        x^{(q)}(F_\bullet)  = \sum_{i = 1}^n \sum_{\substack{L_i}}  \left( L_i \subset L_i + F_1 \subset \cdots \subset L_i + F_{i - 2} \subset F_i \subset F_{i + 1} \subset \cdots \subset F_n\right)
		\qquad \text{for all } F_\bullet \in \F_n,
    \end{align*}
    where the second sum is over one-dimensional subspaces (or \textit{lines}) $L_i$ for which $L_i \subseteq F_i$ but $L_i \nsubseteq F_{i - 1}.$
\end{definition}

\begin{example}\rm
Let $q = n= 2.$ Use $\langle v \rangle $ to express the $\FF_q$-span of the vector $v \in \FF_q^n$ and let $e_1, e_2$ be the standard basis vectors of $\FF_q^n$. Suppose $F_\bullet  = \left(\left \langle e_1 + e_2 \right \rangle \subset \FF_q^n\right)$. Then
\begin{align*}
    x^{(q)}\left(F_\bullet\right) &= \underbrace{ \left(\left \langle e_1 + e_2 \right \rangle \subset \FF_q^n\right)}_{i = 1} + \underbrace{ \left(\left \langle e_1 \right \rangle \subset \FF_q^n\right) +  \left(\left \langle e_2 \right \rangle \subset \FF_q^n\right)}_{i = 2}.
\end{align*}
\end{example}
In \cite{brauner2023invariant}, Reiner, the second, and the fourth author study the eigenspaces of $x^{(q)}.$ (See \cite[\S 4A]{brauner2023invariant}; in the notation in \cite{brauner2023invariant}, $ {\bf k} \mathscr{F}_{n,n}$ is $\CC[\F_n]$ with ${\bf k} = \CC$.)

\begin{theorem}[\cite{brauner2023invariant}]\label{thm:xeigenvalues}The action of $x^{(q)}$ on $\CC[\F_n]$ is semisimple with eigenvalues $[0]_q, [1]_q, \cdots, [n-2]_q, [n]_q$. The $[n-j]_q$-eigenspace of $x^{(q)}$ carries the representation 
\begin{equation}\label{eq:xeigenspaces} \widetilde{\DD}_{j}(q) * G^{(n-j)} \cong \bigoplus_{\substack{ \lambda \vdash n \\ \frakq \  \in \  \D_{j} \textrm{ with } \shape(\frakq) = \mu \\ \lambda \sm \mu \textrm{ a horizontal strip} }} G^{\lambda}, \end{equation}
where $*$ is \emph{Harish-Chandra induction}, see \cite[\S 3B]{brauner2023invariant}. In particular, the kernel of $x^{(q)}$ carries $\widetilde{\DD}_n(q)$, the derangement representation of $\GL{n}$.

\end{theorem}
\begin{proof}
The $x^{(q)}$ action in \cite{brauner2023invariant} agrees with Definition \ref{def:xq} (see \cite[the first equation after Example 4.1]{brauner2023invariant}). We thus apply \cite[Theorem 4.2]{brauner2023invariant}. The isomorphism in \eqref{eq:xeigenspaces} follows from the Pieri rules after taking the map ${\sf ch}_q$ from \cite[\S 3B]{brauner2023invariant}. Note that $\widetilde{\DD}_{1}(q) = 0$, so the eigenvalue $[n-1]_q$ does not appear.
\end{proof}
We will use the bimodule decomposition in Section \ref{section:connectionheckeandflags} to relate $x^{(q)}$ to our operator $\B_n^*(q)$.
By Corollary \ref{cor:eigenvaluesinbimoduledecomposition}, studying the spectrum of $x^{(q)}$ is equivalent to studying that of an element $\varphi \in \HH_n(q)$ for which 
\[ x^{(q)}\cdot F_\bullet = F_\bullet \cdot \varphi \] for all flags $F_\bullet \in \F_n.$ We will show that this element is $q$-random-to-top $\TT_n^*(q)$, and use this to deduce information about $q$-random-to-bottom $\B_n^*(q)$ using Remark \ref{rmkbottomvtop}. 

\begin{prop}\label{xisrantomtotop}
For any flag $F_\bullet \in \F_n,$
\begin{align*}
    x^{(q)}  \cdot F_\bullet = F_\bullet \cdot \TT_n^\ast(q). 
\end{align*}
\end{prop}
\begin{proof}
Let $E_\bullet = \left(E_1 \subset \cdots \subset E_n\right)$, where each $E_i$ is the span of the first $i$ standard basis vectors of $\FF_q^n.$ Since $\GL{n}$ acts transitively on $\F_n$ and $x^{(q)}$ commutes with the action of $\GL{n},$ it suffices to prove that \[x^{(q)}(E_\bullet) = \sum_{i = 1}^{n} E_\bullet \cdot T_{s_{i - 1}}T_{s_{i - 2}}\cdots T_{s_1}.\]
We will show that for a fixed $i$,
  \[E_\bullet  \cdot T_{s_{i - 1}}T_{s_{i - 2}}\cdots T_{s_1} = \sum_{L_i} \left(L_i \subset L_i + E_1  \subset \cdots \subset L_i + E_{i - 2} \subset E_i \subset \cdots \subset E_n\right),\] where the sum is over lines $L_i$ contained in $E_i$ but not contained in $E_{i - 1}.$ 
  Using Definition \ref{def:heckeactiononflags}, we have that
    \begin{align}\label{eqn:flag-sum}
        E_\bullet  \cdot T_{s_{i - 1}}T_{s_{i - 2}}\ldots T_{s_1} &= \sum F_\bullet,
    \end{align}
    where the sum is over all flags $F_\bullet \in \F_n$ for which both:
    \begin{enumerate}
        \item[(i)] $E_{j-1} \subset F_j \neq E_j$ for all $1 \leq j \leq i - 1,$ and
        \item[(ii)] $E_{j} = F_j$ for all $j \geq i$.
    \end{enumerate}
We claim these conditions force $F_j = F_1 + E_{j - 1}$ for all $2 \leq j \leq i - 1.$ Indeed, since $E_1$ and $F_1$ are two distinct one-dimensional subspaces of the two-dimensional subspace $F_2,$  we must have  $F_1 + E_1 = F_2.$ Inductively, for $j > 2,$ assume $F_{j - 1} = F_1 + E_{j - 2}.$  Since $E_{j - 2} \subset F_{j - 1} \neq E_{j - 1},$ the vector $e_{j - 1}$ is not contained in $F_{j - 1}$, so  
\[
F_{j - 1} \subsetneq F_{j - 1} + \langle e_{j - 1} \rangle =  E_{j - 2} + F_1 + \langle e_{j - 1} \rangle  = E_{j - 1} + F_1 \subseteq  F_j. \]
Hence, $F_j =  E_{j - 1}+ F_1,$ as desired. This also forces that $F_1 \nsubseteq E_{i - 1}$ (since otherwise $F_{i - 1} =F_1 + E_{i - 2}$ would equal $E_{i - 1}$). Thus, each flag $F_\bullet$ appearing in the sum (\ref{eqn:flag-sum}) is determined by the line $F_1$ which is contained in $E_i$ but not $E_{i - 1}.$ Conversely, any choice of such a line gives rise to a flag in (\ref{eqn:flag-sum}).
\end{proof}

\subsection{Analysis of $\B_n^*(q)$}\label{sec:b2r}
Finally, we will use the connection between $x^{(q)}$ and $\TT_n^*(q)$ to compute the spectrum of $\B_n^*(q)$. Our analysis for $\B_n^*(q)$ allows us to deduce the same information for $\B_n(q)$.

\begin{theorem}\label{thm:B-polys}
Let $q \in \CC$ be arbitrary. 
\begin{enumerate}
\item Define the polynomial
        \[
        P(y, q) = \prod_{j\in [0,n] \sm \{1\}} (y - [n-j]_q) \in \CC[y].
		\]
		Then, $P(\B_n^*(q), q) = 0 = P(\B_n(q), q)$. 
  \medskip
\item The characteristic polynomial of $\B_n^*(q)$ and $\B_n(q)$ is
        \[ \chi(y,q) = \prod_{j=0}^{n} (y - [n-j]_q)^{m_{n-j}} \in \CC[y],
		\qquad \text{where } m_{n-j} = \binom{n}{j}d_{j} .
		\] \medskip
\item If $\HH_n(q)$ is semisimple, then $\B_n^*(q)$ and $\B_n(q)$ act semisimply on $\HH_n(q)$. \medskip
\item If $\HH_n(q)$ is semisimple,   then $\dim \ker \B_n^*(q) =  d_n = \dim \ker \B_n(q)$.
\end{enumerate}
\end{theorem}

\begin{proof}
\noindent(1) Let us first assume that $q$ is a prime power $p^m$.
We shall show that $P(\TT_n^*(q), q) = 0$ first.
Indeed, Theorem~\ref{thm:xeigenvalues} shows that we can apply Corollary~\ref{cor:eigenvaluesinbimoduledecomposition} to $\varphi = \TT_n^*(q)$ and $\widetilde{\varphi} = x^{(q)}$.
As a result, we see that the element $\TT_n^*(q)$ acts semisimply on $\HH_n(q)$, and that its eigenvalues are those of $x^{(q)}$.
But Theorem~\ref{thm:xeigenvalues} shows that the latter eigenvalues are the numbers $[0]_q, [1]_q, \ldots, [n-2]_q, [n]_q$. We thus see that
\[
\prod_{j\in [0,n] \sm \{1\}} (\TT_n^*(q) - [n-j]_q) = 0 = P(\TT_n^*(q), q).
\]
Applying the algebra isomorphism $\tau$ from Remark~\ref{rmkbottomvtop}, we obtain $P(\B^*_n(q), q) = 0$, since $\tau(\TT_n^*(q)) = \B_n^*(q)$. This implies that $P(\B^*_n(q), q) = 0$ in the case when $q$ is a prime power $p^m$.
Thus, Lemma~\ref{lem:suffices-to-prove-for-infinite-q} (2) ensures that $P(\B^*_n(q), q) = 0$ holds for all $q \in \CC$, since there are infinitely many prime powers. To obtain the analogous statement for $\B_n(q)$, apply the anti-isomorphism $*$ to $P(\B^*_n(q),q) = 0$.
This proves part (1). \medskip

\noindent(2) 
When $q$ is a prime power $p^m$, we again use Corollary~\ref{cor:eigenvaluesinbimoduledecomposition} parts (2) and (3) with Theorem~\ref{thm:xeigenvalues} to deduce the characteristic polynomial of $\TT_n^*(q)$. Theorem~\ref{thm:xeigenvalues} gives the irreducible decomposition of the $[n-j]_q$-eigenspace of $x^{(q)}$, and so by Corollary \ref{cor:eigenvaluesinbimoduledecomposition}(2), we obtain the analogous decomposition for the $[n-j]_q$-eigenspace of $\TT_n^*(q)$. Corollary \ref{cor:eigenvaluesinbimoduledecomposition}(3) then uses this decomposition to obtain the dimension of each eigenspace of $\TT_n^*(q)$. In particular, using the Pieri rules, the dimension of the $[n-j]_q$-eigenspace is 
\[ \sum_{\substack{ \lambda \vdash n \\ \frakq \  \in \  \D_{j} \textrm{ with } \shape(\frakq) = \mu \\ \lambda \sm \mu \textrm{ a horizontal strip} }} \dim(S^\lambda)  = \binom{n}{j} \dim(\DD_{n-j}(q)) = \binom{n}{j} d_{j} = m_{n-j}.  \]
Thus the characteristic polynomial of $\TT_n^*(q)$ is
\[ \prod_{j=0}^{n} (y - [n-j]_q)^{m_{n-j}}. \]
Again using $\tau$, we conclude that $\B_n^*(q) = \tau(\TT_n^*(q))$ has the same characteristic polynomial. This proves the claim for $\B_n^*(q)$ when $q$ is a prime power $p^m$.
Using Lemma~\ref{lem:suffices-to-prove-for-infinite-q} (1), we extend this to all $q \in \CC$. To see the statement for $\B_n(q)$, realize the action of $x^{(q)}$ on $\CC[\F_n]$ as a $[n]!_q \times [n]!_q$ matrix, and let $(x^{(q)})^*$ be its transpose. Then since each $T_{s_i}$ acts as a symmetric matrix, it is not difficult to deduce that for any $F_{\bullet} \in \CC[\F_n]$, 
\[ (x^{(q)})^* \cdot  F_{\bullet} = F_{\bullet} \cdot  \TT_n(q). \]
We can thus apply Corollary~\ref{cor:eigenvaluesinbimoduledecomposition} to $\varphi = \TT_n(q)$ and $\widetilde{\varphi} = (x^{(q)})^*$. This implies that when $q=p^m$, the characteristic polynomial of $\TT_n(q)$ is also $\chi(y,q)$ by Theorem~\ref{thm:xeigenvalues}, since taking the transpose of a matrix preserves its characteristic polynomial. Again, Lemma~\ref{lem:suffices-to-prove-for-infinite-q} (1) then gives the statement for all $q \in \CC$.  

Since $\tau$ preserves characteristic polynomials, we obtain the same result for $\B_n(q) = \tau(\TT_n(q))$. 
\medskip

\noindent(3) 
If $\HH_n(q)$ is semisimple, $q$ is not a primitive $k$-th root of unity for $2 \leq k \leq n$. 

Hence, the $q$-integers $[0]_q, [1]_q, \ldots, [n]_q$ are distinct,
since their pairwise differences $[j]_q - [i]_q$ for $i < j$ can be rewritten as
\[
[j]_q - [i]_q = q^i \left(1 + q + \cdots + q^{j-i-1}\right)
= q^i \cdot \dfrac{1 - q^{j-i}}{1 - q} \neq 0 .
\]
In particular, the polynomial $P$ in part (1) has no double roots.
Thus, both $\B^*_n(q)$ and $\B_n(q)$ act semisimply,
since part (1) shows they are annihilated by this polynomial.
\medskip

\noindent(4) This follows from parts (2) and (3), since $[0]_q = 0$ and $m_0 = d_n$.
\end{proof}

When $\HH_n(q)$ is semisimple, we also describe the left module structure of each eigenspace of $\B^*_n(q)$, which we apply in Section \ref{section:maintheorem}. We will begin by constructing eigenvectors of $\B_n^*(q)$.

\begin{prop}\label{eigenvector_construction}
  Let $q \in \CC$. Suppose $u \in \ker \B_j^\ast(q)$ for some $0 \leq j \leq n$.  Then, $m_{(1^j, n - j)} \ u$ is an eigenvector for $\B_n^*(q)$ with eigenvalue $[n-j]_q$: 
   \[ \left(m_{(1^j, n - j)} \ u\right)\B_n^\ast(q) = [n-j]_q  \left(m_{(1^j, n - j)} \ u\right). \]
\end{prop}
\begin{proof}
Using the fact that $u \in \HH_{j}(q)$ commutes with $T_{s_i}$ for $i > j$, we have 
\begin{align*}
   \left( m_{(1^j, n - j)} \ u \right)\B_n^\ast(q)  &=  \left(m_{1^j, n - j }\ u\right) \sum_{i=1}^n T_{s_i} \cdots T_{s_{n-1}}\\
   &=  \sum_{i = 1}^j\left(  m_{(1^j, n - j)} \ u\right) T_{s_i} \cdots T_{s_{n-1}} + \sum_{i=j + 1}^n \left( m_{(1^j, n - j)} \ u\right) T_{s_i} \cdots T_{s_{n-1}} \\
    &=  \sum_{i = 1}^j \left(m_{(1^j, n - j)} \ u \right)\ T_{s_i}T_{s_{i + 1}}\cdots T_{s_{j - 1}}  T_{s_j}\cdots T_{s_{n-1}} + \sum_{i=j + 1}^n \left(  m_{(1^j, n - j)} \ u \right) T_{s_i} \cdots T_{s_{n-1}} \\
    &= m_{(1^j, n - j)} \ \left( u  \ \B_j^\ast(q)  \right)  \  T_{s_j}\cdots T_{s_{n-1}} + \sum_{i=j + 1}^n  m_{(1^j, n - j)}\ T_{s_i} \cdots T_{s_{n-1}} \ u. \\
    \end{align*}
 The discussion in Example \ref{ex:m1n-j} and the fact that $u \ \B_n^\ast(q) = 0$ then imply that
    \begin{align*}
   m_{(1^j, n - j)} \ \left( u  \ \B_j^\ast(q)  \right)  \  T_{s_j}\cdots T_{s_{n-1}} + \sum_{i=j + 1}^n  m_{(1^j, n - j)}\ T_{s_i} \cdots T_{s_{n-1}} \ u &= \sum_{i=j + 1}^n q^{n - i}\left( m_{(1^j, n - j)} \ u\right) \\
    &= \left(\sum_{i=0}^{n - j - 1} q^{i}\right)\left( m_{(1^j, n - j)} \ u\right)\\
    &= [n - j]_q\left( m_{(1^j, n - j)} \ u\right).
\end{align*}
\end{proof}

We will use the eigenvectors we have constructed in Proposition \ref{eigenvector_construction}, along with our knowledge of the characteristic polynomial of $\B_n^*(q)$, to describe the left $\HH_n(q)$-representations on each eigenspace of $\B_n^*(q)$.

\begin{theorem}\label{cor:prime-eigenspace}
Let $\HH_n(q)$ be semisimple.
Then the $[n-j]_q$-eigenspace of $\B_n^*(q)$ carries the left $\HH_n(q)$-representation 
\[\ker \left(\B_n^\ast(q) - [n - j]_q \right) \cong \ind_{\HH_{j, n - j}(q)}^{\HH_n(q)}\left( \DD_j(q) \otimes \left(S^{(n - j)}\right)^\ast\right) \quad \text{for each $0 \leq j \leq n$}.\]
\end{theorem}
\begin{proof}
Since $\HH_n(q)$ is semisimple,
$q$ is neither $0$ nor a $k$-th root of unity with $1 < k \leq n$.
Thus, Theorem~\ref{thm:B-polys} (4) shows that
\begin{equation}
\dim \ker \B_j^*(q) = d_j
\qquad \text{for all $0 \leq j \leq n$}.
\label{pf.cor:prime-eigenspace.dimker}
\end{equation}

For each $0 \leq j \leq n$, define
\begin{equation}
N_j:= \mathrm{span}_{\CC}\{m_{(1^j, n - j)} \ u \ : \ u \in \ker \B_j^\ast(q)\} \subseteq \HH_{j,n-j}(q) \subseteq \HH_n(q).
\end{equation}
 Observe that each $N_j$ is closed under the left action of the subalgebra $\HH_{j,n - j}(q)$ and carries the representation \[\ker \B_j^\ast(q) \otimes \left(S^{(n - j)}\right)^\ast.\]

 Define $V_j$ to be the left $\HH_n(q)$-module \[V_j:= \HH_n(q) N_j.\]

  Since each eigenspace $ \ker \left(\B_n^\ast(q) -[n - j]_q \right)$ is closed under the left action of $\HH_n(q),$ Proposition \ref{eigenvector_construction} implies there is a containment 
\begin{align}\label{eqn:eigenspace-subspace}
    V_j \subseteq \ker \left(\B_n^\ast(q) -[n - j]_q \right).
\end{align}

 On the other hand, one can use Proposition \ref{lem:factoring-minimal}(1) to see that $\HH_n(q)$ is free as a left $\HH_{j, n - j}(q)$-module with basis 
 \[ \{T_v: v \in X_{j, n - j}\}. \] 
 Similarly, $\HH_n(q)$ is free as a right $\HH_{j, n - j}(q)$-module with basis $\{T_{v^{-1}}: v \in X_{j, n - j}\}.$
 Thus,
\begin{equation} \label{eq:heckeascoset} \HH_n(q) = \bigoplus_{v \in X_{j, n - j}}T_{v^{-1}} \HH_{j, n - j}(q). \end{equation}
 Since $N_j \subseteq \HH_{j, n - j}(q)$,
 it follows that $V_j$ can be written as
 \begin{equation}\label{eq:vjascoset}
 V_j= \bigoplus_{v \in X_{j, n - j}} T_{v^{-1}} \ N_j.\end{equation}

We will construct an isomorphism 
\[ \gamma: \ind_{\HH_{j, n-j}(q)}^{\HH_n(q)}\left(N_j\right) \longrightarrow V_j\]
as follows. By the definition of induction and \eqref{eq:heckeascoset}, the domain can be written as 
\[ \ind_{\HH_{j, n-j}(q)}^{\HH_n(q)}\left(N_j\right) := \HH_{n}(q) \otimes_{\HH_{j, n - j}(q)} N_j = \bigoplus_{v \in X_{j, n - j}} \CC T_{v^{-1}} \otimes_{\HH_{j, n - j}(q)} N_j.\]
Using the description of $V_j$ as in \eqref{eq:vjascoset}, we define $\gamma$ as 
\[ \gamma: T_{v^{-1}} \otimes z \longmapsto T_{v^{-1}}z,\]
which is an isomorphism by the above discussion of $\{ T_{v^{-1}}: v \in X_{j,n-j} \}$ as an $\HH_{j,n-j}(q)$-basis for $\HH_n(q)$.

Hence, as left $\HH_n(q)$-modules
\begin{equation}\label{eqn:module-structure}
   V_j \cong \ind_{\HH_{j, n-j}(q)}^{\HH_n(q)}\left(N_j\right) \cong \ind_{\HH_{j, n-j}(q)}^{\HH_n(q)}\left( \ker \B_j^\ast (q) \otimes \left(S^{(n - j)}\right)^\ast\right) 
\end{equation}
and \[\dim V_j = \binom{n}{j} \cdot \dim \ker \B_j^\ast(q) = \binom{n}{j} d_j \qquad \text{by \eqref{pf.cor:prime-eigenspace.dimker}}.\]

Thus 
\begin{equation}\label{eqn:counting}
    n! =  \sum_{j = 0}^n  \binom{n}{j}  d_j = \sum_{j = 0}^n \dim V_j \leq \sum_{j = 0}^n \dim \ker \left(\B_n^\ast(q) -[n - j]_q \right)  \leq n!,
\end{equation}
forcing the containments in \eqref{eqn:eigenspace-subspace} to be equalities.

In particular, for each $j$, we have
\[\ker \left(\B_n^\ast(q) - [n - j]_q \right) = V_j \cong \ind_{\HH_{j, n - j}(q)}^{\HH_n(q)}\left( \ker \B_j^\ast (q) \otimes \left(S^{(n - j)}\right)^\ast\right).
\]
It thus remains to show that $\DD_j(q) \cong \ker \B_j^\ast(q)$ as left $\HH_j(q)$-modules. To do so, we make use of the Frobenius characteristic map introduced in Section \ref{sec:frob-char}. Let \[\mathcal{K}_j = \mathrm{ch}\left(\ker \B_j^\ast(q)\right).\] 

Observe that $\bigoplus_{j=0}^n V_j \cong \HH_n(q),$ which has Frobenius characteristic $h_{1^n}$. Taking the Frobenius characteristic of (\ref{eqn:module-structure}) and summing over all $j$ implies that \[\sum_{j = 0}^n \mathcal{K}_jh_{n - j} = h_{1^n}.\]  
By the characterization of $\DD_n(q)$ in \eqref{eq:frobeniusderangement}, we conclude that $\DD_j(q) \cong \ker \B_j^\ast(q).$
\end{proof}

\section{Constructing eigenvectors for $\RR_n(q)$}\label{section:constructionofeigenvectors}
In this section we describe a process to recursively compute eigenvectors of $\RR_n(q)$. In Section \ref{subsec:recursionr2r} we prove Theorem \ref{thm:introrecursion}, which gives a recursive formulation of $\RR_n(q)$ in $\HH_n(q)$. We use this recursion in Section \ref{section:recursiveeigenvectorconstruction} to construct eigenvectors of $\RR_{n}(q)$ from eigenvectors of $\RR_{n-1}(q)$, and in turn, apply this to obtain eigenvectors of $\RR_n(q)$ from $\ker \ \RR_{j}(q) $ for $j < n$.

In what follows, we denote by 
\[ \ker\left(\RR_n(q)\mid_{S^\lambda}\right) \subseteq S^\lambda \quad \quad \textrm{ and }\quad \quad \im\left(\RR_n(q)\mid_{S^\lambda}\right) \subseteq S^\lambda \] the kernel and image from the action of $\RR_n(q)$ on $S^\lambda$, respectively. We will fix a basis of the former space, and construct explicit vectors of the latter.

\begin{definition}\label{def:kappa}
Let $\HH_n(q)$ be semisimple. For any $S^\mu$, we fix a basis
    $\kappa_{\mu}$ of $\ker\left(\RR_{|\mu|}(q)\mid_{S^\mu}\right)$.
\end{definition}

In general, it is an open problem to determine explicit expressions for the elements of $\kappa_\mu$, even when $q=1$. However for our arguments, the existence of a basis inside of $S^\mu$ is sufficient. Later, we will show that $|\kappa_\mu| = d^\mu$ when $q\in \R_{>0}$ (Section \ref{sec:eigenbasisprimeq}), which will allow us to construct an eigenbasis of $S^\lambda$ in those cases. 

\subsection{Proof of Theorem \ref{thm:introrecursion}: the recursion}
\label{subsec:recursionr2r}
Recall that $\RR_n(q) = \B_n^*(q) \B_n(q),$ with the definitions of $\B_n(q)$ and $\B_n^*(q)$ given in equations \eqref{eq:b2rdef} and \eqref{eq:r2bdef}. 

Our goal is to prove a recursion for $\RR_n(q)$ using the Jucys-Murphy elements\footnote{Note that while $J_m(q)$ is not defined for $q =0$, the product $q^n J_m(q) $ is defined for any $q \in \CC$ (and is simply 0 when $q=0$).}:

\begin{theorem}[Theorem \ref{thm:introrecursion}] \label{thm:rec} For any $q \in \CC$, the following holds in $\HH_n(q)$:
\[ \B_{n}(q) \ \RR_{n}(q) = \bigg( q  \ \RR_{n-1}(q) + [n]_q + q^{n} J_{n}(q) \bigg) \  \B_{n}(q). \]
\end{theorem}
To do so, we first need the following lemma. 

\begin{lemma}\label{lem:recursion-lemma}
For any $q \in \CC$, 
    \[\B_{n - 1}^\ast(q) \  T_{s_{n - 1}} \ \B_{n - 1}(q) \ \B_n(q) = q \ \RR_{n - 1}(q) \ \B_n(q).\]
\end{lemma}

\begin{proof}
We claim that
\begin{equation}\label{eqn:simpler-pre-recursion-recursion}
    T_{s_{n - 1}} \ \B_{n - 1}(q) \ \B_{n}(q) = q \ \B_{n - 1}(q) \ \B_{n}(q).
\end{equation}
To see why, note that by Proposition \ref{prop:factoring-Cj-symm-coset}, we have
\[\B_{n - 1}(q) \B_n(q) = m_{(1^{n - 2}, 2)} \ x_{(n - 2, 2)} = (1 + T_{s_{n - 1}}) \ x_{(n - 2, 2)}.\]
Combining this equality with the fact that $T_{s_{n - 1}}(1 + T_{s_{n - 1}}) = q \ (1 + T_{s_{n - 1}})$, we obtain
\[ T_{s_{n - 1}} \ \B_{n - 1}(q) \ \B_n(q) = T_{s_{n - 1}}(1 + T_{s_{n - 1}}) \ x_{(n - 2, 2)} =  q(1 + T_{s_{n - 1}}) \ x_{(n - 2, 2)} = q\B_{n - 1}(q) \ \B_n(q).\]

Using Equation \eqref{eqn:simpler-pre-recursion-recursion} then gives
\begin{align*}
    \B_{n - 1}^\ast(q) \ T_{s_{n - 1}} \ \B_{n - 1}(q) \ \B_n(q) = q\B_{n - 1}^\ast(q) \  \B_{n - 1}(q) \ \B_n(q) = q\RR_{n - 1}(q) \ \B_n(q).
\end{align*}
\end{proof}

We are now ready for the proof of Theorem \ref{thm:rec}:

\begin{proof}[Proof of Theorem \ref{thm:rec}]
Scaling $J_n(q)$ by $q^n$ removes all negative exponents of $q,$ making both sides of the equation defined for all $q.$ Note that if 
\begin{equation} \label{eqn:simpler-recursion}
    \B_n(q) \ \B_n^{\ast}(q) =  \B_{n - 1}^\ast(q) \  T_{s_{n - 1}} \ \B_{n - 1}(q) + [n]_q + q^nJ_n(q),
\end{equation}
then the claim follows by multiplying both sides by $\B_n(q)$ on the right and applying Lemma \ref{lem:recursion-lemma}. Hence, it suffices to prove Equation (\ref{eqn:simpler-recursion}), which we do by induction on $n.$ 

First, we rewrite $\B_n(q) \ \B_n^\ast(q):$
\begin{align*}
    \B_n(q) \ \B_n^\ast(q) &= \sum_{i = 1}^{n}\sum_{j = 1}^n (T_{s_{n - 1}} T_{s_{n - 2}} \cdots T_{s_i}) (T_{s_j} \cdots T_{s_{n - 2}}T_{s_{n - 1}})\\
        &= \sum_{j=1}^{n}T_{s_j} \dots T_{s_{n - 1}} + \sum_{i=1}^{n - 1} T_{s_{n - 1}}T_{s_{n - 2}} \dots T_{s_i} + T_{s_{n - 1}}\left(\sum_{i=1}^{n - 1}\sum_{j=1}^{n - 1}(T_{s_{n-2}}\dots T_{s_i})(T_{s_j} \dots T_{s_{n-2}})\right) T_{s_{n - 1}}
        \\
        &= \sum_{j=1}^{n}T_{s_j} \dots T_{s_{n - 1}} + \sum_{i=1}^{n - 1} T_{s_{n - 1}} \dots T_{s_i} + T_{s_{n - 1}}\left(\B_{n - 1}(q) \ \B_{n - 1}^\ast(q)\right) T_{s_{n - 1}}.
        \end{align*}
   We rewrite $\B_{n - 1}(q) \ \B_{n - 1}^\ast(q)$ with the induction hypothesis:
        \begin{align*}
        \sum_{j=1}^{n}T_{s_j} \dots T_{s_{n - 1}} &+ \sum_{i=1}^{n - 1} T_{s_{n - 1}} \dots T_{s_i} + T_{s_{n - 1}}\bigg(\B_{n - 2}^\ast(q) T_{s_{n - 2}}\B_{n - 2}(q) + [n - 1]_q + q^{n - 1}J_{n - 1}(q)\bigg) T_{s_{n - 1}}\\
        &= \sum_{j=1}^{n}T_{s_j} \dots T_{s_{n - 1}} + \sum_{i=1}^{n - 1} T_{s_{n - 1}} \dots T_{s_i} + \underbrace{[n - 1]_qT_{s_{n - 1}}^2}_{=:v_1} + \underbrace{\sum_{i = 1}^{n - 2}q^i T_{s_{n - 1}}T_{(i, n - 1)}T_{s_{n - 1}}}_{=:v_2}\\&+\underbrace{\sum_{i=1}^{n-2}\sum_{j=1}^{n-2}T_{s_{n - 1}}(T_{s_i}\dots T_{s_{n - 3}})T_{s_{n-2}}(T_{s_{n - 3}}\dots T_{s_j})T_{s_{n - 1}}}_{=:v_3}.
        \end{align*}
    We expand $v_1$, $v_2,$ and $v_3$ separately. First,
    \begin{align*}
        v_1 = [n - 1]_q \left(q + (q - 1)T_{s_{n - 1}}\right)= [n]_q - 1 + [n - 1]_q(q - 1)T_{s_{n - 1}}= [n]_q - 1 + (q^{n - 1}- 1)T_{s_{n - 1}}.
    \end{align*}
    Next, since $(i,n-1)$ has reduced expression 
    \[ (i,n-1) = s_{i} \  s_{i+1} \cdots s_{n-3} \  s_{n-2} \  s_{n-3} \cdots s_{i+1} \ s_i,\]
    we have $(i,n) = s_{n-1} \cdot (i,n-1) \cdot s_{n-1}$ is a reduced expression and so
    \begin{align*}
        v_2 := \sum_{i = 1}^{n - 2}q^i T_{s_{n - 1}}T_{(i, n - 1)}T_{s_{n - 1}}= \sum_{i=1}^{n-2}q^{i}T_{(i, n)}.
    \end{align*}
    Finally, we use the generating relations for $\HH_n(q)$ to rewrite $v_3$:

    \begin{align*}
        v_3  &:= \sum_{i=1}^{n-2}\sum_{j=1}^{n-2}\color{blue}{T_{s_{n - 1}}}\normalcolor (T_{s_i}\dots T_{s_{n - 3}})\color{blue}{T_{s_{n-2}}}\normalcolor (T_{s_{n - 3}}\dots T_{s_j})\color{blue}T_{s_{n - 1}}  \\
        &=\sum_{i=1}^{n-2}\sum_{j=1}^{n-2}(T_{s_i}\dots T_{s_{n-3}}) \blue{T_{s_{n - 1}}} T_{s_{n-2}} \blue{T_{s_{n - 1}}} \normalcolor (T_{s_{n - 3}} \dots T_{s_j}) \tag{$T_{s_i} T_{s_j} = T_{s_j} T_{s_i}$ when $|i-j|\geq 2$}\\
        &= \sum_{i=1}^{n-2}\sum_{j=1}^{n-2}(T_{s_i}\dots T_{s_{n-3}}) \blue{T_{s_{n-2}} T_{s_{n - 1}} T_{s_{n-2}}} \normalcolor (T_{s_{n-3}} \dots T_{s_j}).\tag{$T_{s_i}T_{s_{i + 1}}T_{s_i} = T_{s_{i + 1}}T_{s_i}T_{s_{i + 1}}$}
    \end{align*}
 
    Substituting these simplifications into our full expression and rearranging, we obtain
        \begin{align*}
        \B_{n}(q) \ \B_n^\ast(q) &= \left(1 + \sum_{j=1}^{n - 1}T_{s_j} \dots T_{s_{n - 1}} \right) + \sum_{i=1}^{n - 1} T_{s_{n - 1}} \dots T_{s_i} + \left( [n]_q - 1 + q^{n - 1} T_{s_{n - 1}} - T_{s_{n - 1}} \right) + \sum_{i=1}^{n-2}q^{i}T_{(i, n)}\\
        &\quad + \sum_{i=1}^{n-2}\sum_{j=1}^{n-2}T_{s_i}\dots T_{s_{n-2}}T_{s_{n - 1}}T_{s_{n-2}} \dots T_{s_j}\\
        &=[n]_q + \left( q^{n - 1} T_{s_{n - 1}} + \sum_{i=1}^{n-2}q^{i}T_{(i, n)} \right)+ \sum_{i=1}^{n-1}\sum_{j=1}^{n-1}T_{s_i}\dots T_{s_{n-2}}T_{s_{n - 1}}T_{s_{n-2}} \dots T_{s_j}
        \\
        &= [n]_q + q^nJ_n(q) + \B_{n - 1}^\ast(q) \  T_{s_{n - 1}}\B_{n-1}(q).
    \end{align*}
    \end{proof}

\subsection{Building eigenvalues from Theorem \ref{thm:introrecursion}}\label{section:recursiveeigenvectorconstruction}
Our next task is to turn the recursion in Theorem \ref{thm:introrecursion} into a method for constructing eigenvectors of $\RR_n(q)$. The idea behind Theorem \ref{thm:recursive_eigenvector_construction} below is to use the Tower Rule (Proposition \ref{lem:tower-rule}) and the properties of the seminormal units to build an eigenvector for $\RR_{n}(q)$ from one of $\RR_{n-1}(q)$.

\begin{theorem}\label{thm:recursive_eigenvector_construction}
Let $\HH_n(q)$ be semisimple and $\lambda' \lessdot \lambda$ with $|\lambda| = n$.
Assume that $u' \in S^{\lambda'}$ is an eigenvector for $\RR_{n-1}(q)$ with eigenvalue $\eigen$.
Then,
\[ u:= u' \ \Phi_{\lambda \sm \lambda'} \  \B_n(q) \ p_{\lambda}\]
belongs to $S^\lambda$, and is
either zero or an eigenvector for $\RR_n(q)$ with eigenvalue 
\[ q \eigen + [n]_q + q^{n} \content_{\lambda \sm \lambda'}(q). \]
 
\end{theorem}
\begin{proof} 
 Consider the $u$ defined above. Thus,
\begin{equation}
 u = u' \ \Phi_{\lambda \sm \lambda'} \  \B_n(q) \ p_{\lambda}
 = u' \ \Phi_{\lambda \sm \lambda'} \ p_{\lambda}\ \B_n(q) ,
 \label{eq:recursive_eigenvector_construction:pf.0}
\end{equation}
since $p_\lambda$ is central in $\HH_n(q)$.

Write $u'$ as 
\begin{equation}
u' = \sum_{\frakt' \in \syt(\lambda')} c_{\frakt'} w_{\frakt'}
\label{eq:recursive_eigenvector_construction:pf.1}
\end{equation}
for scalars $c_{\frakt'}$. This can be rewritten as
\begin{equation}
u' = \sum_{\substack{\frakt \in \syt(\lambda):\\ \shape(\frakt \vert_{n - 1}) = \lambda'}} c_{\frakt \vert_{n - 1}} w_{\frakt \vert_{n - 1}},
\label{eq:recursive_eigenvector_construction:pf.2}
\end{equation}
since there is a bijection from the set of all $\frakt \in \syt(\lambda)$ satisfying $\shape(\frakt \vert_{n - 1}) = \lambda'$ to $\syt(\lambda')$ that sends each $\frakt$ to $\frakt \vert_{n-1}$.
Now, for each $\frakt \in \syt(\lambda)$ satisfying $\shape(\frakt \vert_{n - 1}) = \lambda'$, we have $ w_{\frakt \vert_{n - 1}}  \ \Phi_{\lambda \sm \lambda'} \ p_{\lambda} = w_{\frakt} $
(by \eqref{eq:restrictionrewrite-wt}). Hence, subjecting
\eqref{eq:recursive_eigenvector_construction:pf.2} to the action of $\Phi_{\lambda \sm \lambda'} \ p_{\lambda}$, we obtain
\begin{equation}
u' \ \Phi_{\lambda \sm \lambda'} \ p_{\lambda}
= \sum_{\substack{\frakt \in \syt(\lambda):\\ \shape(\frakt \vert_{n - 1}) = \lambda'}} c_{\frakt \vert_{n - 1}} w_{\frakt} \in S^\lambda .
\label{eq:recursive_eigenvector_construction:pf.3}
\end{equation}
Thus, \eqref{eq:recursive_eigenvector_construction:pf.0} shows that $u \in S^\lambda \ \B_n(q) \subseteq S^\lambda$.

Next, again using the definition of $u$ and the fact that $p_\lambda$ is central,
and then using Theorem \ref{thm:rec}, we find
\begin{align*}
    u \  \RR_n(q) & =  u' \ \Phi_{\lambda \sm \lambda'} \  \B_n(q) \ p_{\lambda} \ \RR_n(q) \\
    &= \bigg( u' \ \Phi_{\lambda \sm \lambda'} \bigg) \bigg( q \RR_{n-1}(q) + [n]_q + q^{n} J_{n}(q) \bigg) \B_{n}(q) \  p_\lambda \\
    &= \underbrace{ q\big( u' \ \Phi_{\lambda \sm \lambda'}  \ \RR_{n-1}(q) \ \B_n (q)\ p_{\lambda}\big)}_{=:v_1} +  \underbrace{[n]_q\big( u' \ \Phi_{\lambda \sm \lambda'} \B_n(q) \  p_{\lambda} \big) }_{=:v_2}  + \underbrace{q^n\big(u' \ \Phi_{\lambda \sm \lambda'} J_n(q) \ \B_n(q) \ p_{\lambda} \big)}_{=:v_3}.
\end{align*}
We will go through each summand in turn. For $v_1$, note that $\RR_{n-1}(q)$ commutes with $\Phi_{\lambda \sm \lambda'}$ since $\RR_{n-1}(q) \in \HH_{n-1}(q)$. 
Since $u'$ is an eigenvector for $\RR_{n-1}(q)$ by assumption, we thus have
\[ v_1 = q\big( u' \ \RR_{n-1}(q) \Phi_{\lambda \sm \lambda'}\ \B_n(q)  \ \ p_{\lambda}\big) = (q \eigen) \ u' \ \Phi_{\lambda \sm \lambda'} \ \B_n(q) \  p_{\lambda} = (q \eigen) u. \]
The second term is $v_2 = [n]_q u$.

The third term is the most interesting.

Using the centrality of $p_\lambda$ once again, we can rewrite $v_3$ as
\begin{align*}
    v_3
	&= q^n\big(u' \ \Phi_{\lambda \sm \lambda'} \ p_{\lambda} \ J_n(q) \ \B_n(q) \big) \\
	&= q^n \ \left( \sum_{\substack{\frakt \in \syt(\lambda):\\ \shape(\frakt \vert_{n - 1}) = \lambda'}} c_{\frakt \vert_{n - 1}} w_{\frakt} \right)  \ J_n(q) \ \B_n (q)
	\ \tag{by \eqref{eq:recursive_eigenvector_construction:pf.3}}\\
    &= q^n \content_{\lambda \sm \lambda'}(q) \left( \sum_{\substack{\frakt \in \syt(\lambda):\\ \shape(\frakt \vert_{n - 1}) = \lambda'}} c_{\frakt \vert_{n - 1}} w_{\frakt} \right)  \ \B_n(q) \ \tag{Theorem \ref{thm:dipperjamesyoungbasis} (2)}\\
    &= q^n \content_{\lambda \sm \lambda'}(q) \left( u' \ \Phi_{\lambda \sm \lambda'} \ p_{\lambda} \right)  \ \B_n(q)
	\ \tag{by \eqref{eq:recursive_eigenvector_construction:pf.3}}\\
    &= q^n \content_{\lambda \sm \lambda'}(q) \  u .
    \ \tag{by \eqref{eq:recursive_eigenvector_construction:pf.0}}
\end{align*}
Thus in total we can conclude that 
\[ u \RR_n(q) = v_1 + v_2 + v_3 = (q \eigen)\ u + [n]_q \ u +q^n \content_{\lambda \sm \lambda'}(q) \  u = \big(q\eigen + [n]_q + q^n \content_{\lambda \sm \lambda'}(q)\big) \ u. \]
\end{proof}

Theorem \ref{thm:recursive_eigenvector_construction} gives an inductive method of constructing eigenvectors of $\RR_n(q)$. The base case of this inductive process are the 0-eigenvectors of $\RR_n(q)$, i.e. the elements of $\kappa_\mu$. Theorem \ref{thm:eigenvectors} below iterates the construction in Theorem \ref{thm:recursive_eigenvector_construction}, starting with this base case.
Recall the elements $p_{\frakt}$ introduced in Definition~\ref{def:skew-pt}.

\begin{theorem}\label{thm:eigenvectors}
Suppose $\HH_n(q)$ is semisimple. Fix a partition $\lambda$ of $n$, and a partition $\mu$ of $j$ with $\mu \subseteq \lambda$.
Let $u \in \kappa_\mu$  and $\frakt \in \syt(\lambda \sm \mu)$.
Then, the element
\[ u \ \Phi_{\frakt } \ p_{\frakt } \  \B_{n, n-j}(q) \]
belongs to $S^\lambda$, and is either zero or is an eigenvector for $\RR_n(q)$ with eigenvalue 
\[ q^{n} \content_{\lambda\sm \mu}(q)+ \sum_{k=j + 1}^{n} q^{n-k}[k]_q.  \]
 \end{theorem}
 \begin{proof}

For any $\fraks \in \syt(\mu)$ and $j \leq k \leq n$ define 
$\lambda^{(k)}:= \shape\left(\frakt(\fraks) \vert_k \right)$ (this is independent of $\fraks$).
Thus,
\[ \mu = \lambda^{(j)} \subseteq \lambda^{(j + 1)} \subseteq \cdots \subseteq \lambda^{(n)} = \lambda. \]

We have
\begin{align*}
\Phi_{\frakt} &= \Phi_{\lambda^{(j+1)} \sm \lambda^{(j)}} \Phi_{\lambda^{(j+2)} \sm \lambda^{(j+1)}} \cdots \Phi_{\lambda^{(n)} \sm \lambda^{(n-1)}} \qquad \text{by Definition~\ref{def:phi};} \\
p_{\frakt} &= p_{\lambda^{(j + 1)}} \ p_{\lambda^{(j + 2)}} \cdots p_{\lambda^{(n)}}
\qquad \text{by Definition~\ref{def:skew-pt};} \\
\B_{n, n-j}(q) &= \B_{j+1}(q) \ \B_{j+2}(q) \cdots \B_n(q) \qquad \text{by Definition~\ref{def:Cjn}.}
\end{align*}
Thus,
 \begin{align*}
      u \ \Phi_{\frakt } \ p_{\frakt} \  \B_{n, n-j}(q)
	  &=  u \ \Phi_{\lambda^{(j+1)} \sm \lambda^{(j)}} \Phi_{\lambda^{(j+2)} \sm \lambda^{(j+1)}} \cdots \Phi_{\lambda^{(n)} \sm \lambda^{(n-1)}} \ \ p_{\lambda^{(j + 1)}} \ p_{\lambda^{(j + 2)}} \cdots p_{\lambda^{(n)}} \  \B_{j+1}(q) \ \B_{j+2}(q) \cdots \B_n(q) \\ 
      &=  u \  \left(\Phi_{\lambda^{(j+1)} \sm \lambda^{(j)}} \  \B_{j+1} (q) \ p_{\lambda^{(j + 1)}} \right) \  \left(\Phi_{\lambda^{(j+2)} \sm \lambda^{(j+1)}} \  \B_{j+2}(q) \ p_{\lambda^{(j + 2)}}\right)  \  \cdots \  \left(\Phi_{\lambda^{(n)} \sm \lambda^{(n-1)}} \  \B_n(q) \ p_{\lambda^{(n)}} \right)
 \end{align*} 
by Lemma~\ref{lem:commuting-phi} and because the $p_{\lambda^{(j)}}$ are central in the respective $\HH_{|\lambda^{(j)}|}(q)$.

For $j \leq m \leq n$, define $v_m$ as
 \[ v_m:=u  \ \left(\Phi_{\lambda^{(j+1)} \sm \lambda^{(j)}} \ \B_{j+1}(q)\ p_{\lambda^{(j + 1)}} \right)   \cdots \left( \Phi_{\lambda^{(m)} \sm \lambda^{(m-1)}} \ \B_{m}(q) \  p_{\lambda^{(m)}} \right),  \]
so that $u \ \Phi_\frakt \ p_{\frakt} \ \B_{n, n-j}(q) = v_{n}.$
Assume by induction that $v_m$ belongs to $S^{\lambda^{(m)}}$ and is an eigenvector for $\RR_m(q)$ with eigenvalue 
\[ q^{m} \content_{\lambda^{(m)} \sm \mu}(q)+ \sum_{k=j + 1}^{m} q^{m - k}[k]_q.  \]
(Note this holds trivially in the base case $m = j$.) By Theorem \ref{thm:recursive_eigenvector_construction}, the element
\[ v_{m+1}= v_m\left( \Phi_{\lambda^{(m+1)} \sm \lambda^{(m)}} \ \B_{m+1}(q) \  p_{\lambda^{(m + 1)}} \right) \]
belongs to $S^{\lambda^{(m+1)}}$, and is zero or an eigenvector for $\RR_{m+1}(q)$ with eigenvalue
\begin{align*}
    q \left(q^{m} \content_{\lambda^{(m)} \sm \mu}(q)+ \sum_{k=j + 1}^{m} q^{m - k}[k]_q\right)& + [m + 1]_q + q^{m + 1}\content_{\frakt, m + 1}(q)\\
    &= q^{m + 1}\left(\content_{\lambda^{(m)} \sm \mu}(q) + \content_{\frakt, m + 1}(q)\right) + \sum_{k = j + 1}^{m + 1}q^{m + 1 - k}[k]_q\\
    &= q^{m + 1}\content_{\lambda^{(m + 1)}\sm \mu}(q) + \sum_{k = j + 1}^{m + 1}q^{m + 1 - k}[k]_q.
\end{align*}
 \end{proof}
 \begin{example}\rm
 Suppose we begin with $\mu = (1,1).$ Then $d^{(1,1)} = 1$, so $\kappa_{(1,1)}$ contains one element. One choice of such an element is $u = 12 - 21 \in S^{(1,1)}.$ 
 To build an eigenvector in $S^{(2,1,1)}$, take
 \[ \frakt^{(2,1,1) \sm (1,1)}= \begin{ytableau} *(lightgray) & 3\\ *(lightgray) \\ 4
\end{ytableau}.\]
 One can check that the element 
 \begin{equation}\label{ex:eigenvector}(12 - 21) \ \Phi_{\frakt^{(2,1,1) \sm (1,1)}}  \ p_{\frakt^{(2,1,1) \sm (1,1)}} \  \B_{4, 2}(q) = (1  2 13- 2 113)\ p_{(2,1)} \ p_{(2,1,1)} \ \B_3(q) \ \B_4(q) \neq 0.\end{equation} Hence, by Theorem \ref{thm:eigenvectors}, it is an eigenvector for $\RR_4(q)$ with eigenvalue \[q^4 c_{(2,1,1) \sm (1,1)}(q) + \sum_{k = 3}^4 q^{4  - k}[k]_q = q^4 \left([1]_q + [-2]_q\right) + q[3]_q + [4]_q = [5]_q + q.\]
 We will see in Section \ref{section:spanning}
why the vector in \eqref{ex:eigenvector} is a ``good'' choice of eigenvector.
 \end{example}

\section{A spanning set of $\RR_n(q)$-eigenvectors when $q$ is positive}\label{section:spanning}
We will use the recursive construction of eigenvectors of $\RR_n(q)$ acting on $S^\lambda$ in Theorem \ref{thm:eigenvectors} to define a spanning set $\basis_\lambda$ of $S^\lambda$ when $q \in \R_{>0}$. This is in some ways the most important---but also the most technical---section of the paper. We will break the argument up into three parts. We first explain the necessity of the assumption that $q \in \R_{>0}$ in Section \ref{sec:useofpositivity}. In Section \ref{sec:spanningargumentsubsection}, we define the set $\basis_\lambda$ and prove it spans $S^\lambda$ (Theorem~\ref{thm:orderedspanningset}) using a \emph{Straightening Lemma} (Lemma~\ref{lem:horiz-strip-diff-by-cnostant}). Then in Section \ref{section:simpledescriptionbasis} we give a simpler description of $\basis_\lambda$ that shows it can be written with elements indexed by horizontal strips. 

\subsection{Use of positivity}\label{sec:useofpositivity}
We prove a lemma that that is responsible for the case $q \in \R_{>0}$ being simpler to handle than the general case.

\begin{lemma}[Use of positivity]\label{lemma:whenqisreal}
Suppose $q \in \R_{>0}$. Then the following are true:
\begin{enumerate}
    \item If $a \in \HH_n(q)$ has real coefficients and satisfies $a^* a = 0$, then $a = 0$. \medskip
    \item We have $\ker \RR_n(q) = \ker \B_n^*(q)$. \medskip
	\item We have $\HH_n(q) = \ker \RR_n(q) \oplus \im \RR_n(q)$. \medskip
	\item For any partition $\lambda \vdash n$, we have $S^\lambda = \ker \left(\RR_n(q)\vert_{S^\lambda}\right) \oplus \im \left(\RR_n(q)\vert_{S^\lambda}\right)$.
\end{enumerate}
\end{lemma}

\begin{proof}
Throughout the proof, we restrict ourselves to the elements of $\HH_n(q)$ that have real coefficients.
These form an $\R$-subalgebra of $\HH_n(q)$, since $q \in \R$.
Restricting ourselves to this subalgebra does not change the claims of (2) and (3), since the linear algebra does not change when we extend the base field. \smallskip

\noindent (1) To prove the first claim, we use a bilinear form on $\HH_n(q)$ defined by
\begin{equation}\label{bilinearform} ( T_x , T_y ) := \begin{cases}
    q^{\ell(x)} & x = y^{-1} \\
    0 & \textrm{otherwise}.
\end{cases} \end{equation}
This form is known to be associative (see, e.g., \cite[Prop 1.16]{Mathas}).
Now, let $a \in \HH_n(q)$ have real coefficients and satisfy $a^* a = 0$.
Let $[T_w] a$ denote the coefficient of $T_w$ in $a$.
Then, by the associativity of our form,
\[ (a^*, a) = (a^*a,1) = 0,  \]
since $a^* a = 0$. Therefore,
\[ 0 = (a^*, a) = \sum_{w \in \symm_n} q^{\ell(w)} ([T_w]a)^2
\qquad \text{by \eqref{bilinearform}.}
\]
Since $q \in \R_{>0}$ and all squares $([T_w]a)^2$ are in $\R_{\geq 0}$ by assumption,
we thus conclude that $[T_w]a = 0$ for all $w \in \symm_n$.
Hence, $a =0$. \medskip

\noindent (2) For the second claim, it is clear that $\ker \B_n^*(q) \subseteq \ker \RR_n(q)$ since $\RR_n(q)$ acts on the right. Thus it is enough to prove the reverse containment.
Let $x \in \ker \RR_n(q)$. Thus, $x \in \HH_n(q)$ and $x \ \RR_n(q) = 0$, and therefore
\[ x \ \RR_n(q) \ x^* = x \ \B^*_n(q)\ \B_n(q) \ x^* = (\B_n(q) \ x^*)^* \ (\B_n(q) \ x^*) = 0.\]
Setting $a = \B_n(q) x^*$ in (1) above lets us conclude that $\B_n(q) x^* = 0$, and applying $*$ again transforms this into $x \B^*_n(q) = 0$, meaning that $x \in \ker \B_n^*(q)$ as desired.
\medskip

\noindent (3) Recall from Remark~\ref{rmk:antiiso} that $*$ is an involutive anti-isomorphism. Note that 
\[ \RR_n^*(q):=\left( \RR_n(q) \right)^* = \left( \B^*_n(q) \ \B_n(q) \right)^*  = \B_n^*(q) \ \B_n(q) = \RR_n(q).\]

We must prove that $\HH_n(q) = \ker \RR_n(q) \oplus \im \RR_n(q)$. The dimensions of $\ker \RR_n(q)$ and $\im \RR_n(q)$ add up to $\dim \HH_n(q)$ by the rank-nullity theorem.
Thus, it suffices to show that $\ker \RR_n(q) \cap \im \RR_n(q) = 0$.

Assume $v \in \ker \RR_n(q) \cap \im \RR_n(q)$. Since $v \in \im \RR_n(q)$ it can be written as $v = x \RR_n(q)$ for some $x \in \HH_n(q)$. On the other hand, because $v \in \ker \RR_n(q)$, 
\[ 0 = v  \ \RR_n(q) = x  \ \RR_n(q)  \ \RR_n(q) = x  \ \RR_n(q) \  \RR^*_n(q) .\]
Thus we also have $ x \ \RR_n(q) \  \RR_n^*(q)  \ x^* = 0$. Hence, by part (1), we obtain $v = x \RR_n(q) = 0$. \medskip

\noindent (4) We can assume that $S^\lambda$ is a right $\HH_n(q)$-submodule of $\HH_n(q)$, since any irreducible representation of a semisimple finite-dimensional algebra can be embedded into the algebra.
From part (3), we know that $\ker \RR_n(q) \cap \im \RR_n(q) = 0$. Thus, a forteriori, we have $\ker \left(\RR_n(q)\vert_{S^\lambda}\right) \cap \im \left(\RR_n(q)\vert_{S^\lambda}\right) = 0$. Since the dimensions of the kernel and image add up to $\dim S^\lambda$, we obtain  $S^\lambda = \ker \left(\RR_n(q)\vert_{S^\lambda}\right) \oplus \im \left(\RR_n(q)\vert_{S^\lambda}\right)$.
\end{proof}
\subsection{The spanning argument}\label{sec:spanningargumentsubsection}
We are now ready to define the set $\basis_\lambda$.
Recall that $\kappa_\mu \subseteq S^\mu$ is a basis of $\ker \left(\RR_{|\mu|}(q) \vert_{S^\mu}\right)$. We also have from Section \ref{section:factorization} that 
\[ \B_{n, n-|\mu|}(q):=  \ \B_{|\mu|+1}(q) \ \cdots \B_{n}(q),\]
while the definitions of $p_{\frakt^{\lambda \sm \mu}}$ and $\Phi_{\frakt^{\lambda \sm \mu}}$ can be found in Definitions \ref{def:skew-pt} and \ref{def:phi}.

\begin{definition}
For a fixed $\lambda \vdash n$, any $\mu \subseteq \lambda,$ and $u \in \kappa_\mu$ define the element 
\[ x_{\mu(u)}^\lambda:= u \  \Phi_{\frakt^{\lambda \sm \mu}} \  p_{\frakt^{\lambda \sm \mu}} \  \B_{n, n-|\mu|}(q) \in S^\lambda. \]
Denote the set of all such $x_{\mu(u)}^\lambda$ as
    \[ \basis_{\lambda}:= \{ x_{\mu(u)}^\lambda: u \in \kappa_\mu, \ \  \mu \subseteq \lambda \}  . \]
\end{definition}

Our goal is to prove that $\basis_\lambda$ spans $S^\lambda$ (Theorem \ref{thm:orderedspanningset}) when $q \in \R_{>0}$.  Recall that the general construction in Theorem \ref{thm:eigenvectors} showed that for any $u \in \kappa_\mu$ with $|\mu| = j$ and $\frakt \in \syt(\lambda \sm \mu)$,
\begin{equation}\label{eq:arbitraryeigenvector}
u \ \Phi_{\frakt } \ p_{\frakt } \  \B_{n, n-|\mu|}(q) \in S^\lambda \end{equation}
was either 0 or an eigenvector for $\RR_n(q)$. The idea behind this construction was to build the element in \eqref{eq:arbitraryeigenvector} iteratively via the Tower Rule (Proposition \ref{lem:tower-rule}) using the chain of inclusions 
\begin{equation}\label{eq:tableauchain} \mu \sm \mu \subseteq \shape(\frakt|_{j+1}) \subseteq \shape(\frakt|_{j+2}) \cdots \subseteq \shape(\frakt) = \lambda \sm \mu. \end{equation}

An element $x_{\mu(u)}^\lambda$ of $\basis_\lambda$ is a special case of an element of the form \eqref{eq:arbitraryeigenvector}. The key difference is that for $x_{\mu(u)}^\lambda$ we have fixed a specific tableau $\frakt^{\lambda \sm \mu}$, and thus chosen a specific inclusion chain of the form in \eqref{eq:tableauchain}.

At the heart of proving that $\basis_\lambda$ spans $S^\lambda$ is showing that this is possible; namely, that at every step in the inductive process given in Theorem \ref{thm:eigenvectors}, we can \emph{straighten} our tableau to ensure that we are using the tableau of largest dominance order $\frakt^{\shape(\frakt|_k) \sm \mu}$. We demonstrate this idea in Example \ref{ex:straighteningrule} below.

\begin{example}\label{ex:straighteningrule} \rm
   We illustrate a subtlety that arises in proving Theorem \ref{thm:orderedspanningset}. In particular, we will need to show that for any $\lambda$, whenever $\mu \subseteq \lambda'$ for $\lambda' \lessdot \lambda$ and $u \in \kappa_\mu,$
   we have 
      \begin{equation}\label{eqn:suffices-to-show}
       x_{\mu(u)}^{\lambda'} \ \Phi_{\lambda \sm \lambda'} \ p_\lambda \ \B_n(q) \in \spann \{x_{\mu(u)}^\lambda \}.
   \end{equation}
We explain why this is not obvious below. To illustrate,  consider \eqref{eqn:suffices-to-show} for $\lambda = (5, 3, 1)$, $\lambda' = (4, 3, 1),$ and $\mu = (3,1).$ First, note that
\[ x^{\lambda'}_{\mu(u)} = u \ \Phi_{\frakt^{\lambda' \sm \mu}} \ p_{\lambda' \sm \mu} \ \B_{8, 4}(q) \quad \quad  \textrm{ where } \quad \quad \frakt^{\lambda' \sm \mu}  = \ytableausetup{centertableaux}
 \ytableaushort
  {\none\none\none5,\none67, 8}
 * {4, 3, 1}
 * [*(lightgray)]{3, 1}.\]

Thus, 
\[ x^{\lambda'}_{\mu(u)}  {\Phi_{\lambda \sm \lambda'} \ p_{\lambda} \ \B_9(q)} \normalcolor =  \left(u \ \Phi_{\frakt^{\lambda'\sm \mu}} \ p_{\frakt^{\lambda' \sm \mu}} \ \B_{8, 4}(q) \right){\Phi_{\lambda \sm \lambda'} \ p_{\lambda} \ \B_9(q)} \normalcolor = u \ \red{\Phi_{\frakt}}\normalcolor \ p_{\frakt} \ \B_{9, 5}(q)\] where 
  \[\textcolor{red}{\frakt} = \begin{ytableau}
      *(lightgray)  & *(lightgray) & *(lightgray) & 5 & *(blue){\color{white}{9}}\\
      *(lightgray) & 6 & 7\\
      8
  \end{ytableau} \quad \quad \neq \quad \quad \frakt^{\lambda \sm \mu} = \begin{ytableau}
      *(lightgray)  & *(lightgray) & *(lightgray) & 5 & 6\\
      *(lightgray) & 7 & 8\\
      *(blue){\color{white}{9}}
  \end{ytableau}.\]

Since $\frakt \neq \frakt^{\lambda \sm \mu}$, it is not clear that $u \ \Phi_\frakt \  p_\frakt \ \B_{9, 5}(q)$ is a scalar multiple of  $x_{\mu(u)}^\lambda  = u\ \Phi_{\frakt^{\lambda \sm \mu}} \  p_{\frakt^{\lambda \sm \mu}} \ \B_{9, 5}(q)$. The goal of our Straightening Lemma (Lemma \ref{lem:horiz-strip-diff-by-cnostant}) is to show that it actually is. 
\end{example}

\begin{remark}\rm
The straightening process is missing from \cite{DiekerSaliola}. In particular, the proof of \cite[Proposition 58]{DiekerSaliola} assumes they can inductively construct an eigenvector that is already ``straightened,'' but without an argument as in Lemma \ref{lem:horiz-strip-diff-by-cnostant}, this assumption is not valid, as illustrated in Example \ref{ex:straighteningrule}. By setting $q=1$, our Lemma \ref{lem:horiz-strip-diff-by-cnostant} thus corrects this mistake.
\end{remark}

We will now formalize the straightening performed in Example \ref{ex:straighteningrule} in Lemmas \ref{missinglemma}, \ref{lemma:factordelta} and \ref{lem:horiz-strip-diff-by-cnostant}. Though these arguments are somewhat technical, the key idea behind them is that the miraculous properties of the seminormal units given in Theorem \ref{thm:dipperjamesyoungbasis}\eqref{thm:dipperjamesactiononseminormalunits} provide a way to straighten our tableau. 

\begin{lemma}\label{missinglemma}
Assume $\HH_n(q)$ is semisimple. Suppose $\frakt, \frakq \in \syt(\lambda)$ are such that  
\begin{itemize}
    \item  $\frakt \cdot s_{i} = \frakq$ so that $\frakt|_{i-1} = \frakq|_{i-1}$, and
    \item $\frakq \gtdom \frakt$.
\end{itemize}
 Then $ w_\frakt \left(1 + T_{s_{i}}\right) $ is a scalar multiple of $w_\frakq \left(1 + T_{s_{i}}\right)$. Specifically:
 \begin{align*}
       w_\frakt \left(1 + T_{s_{i}}\right)
	   = w_\frakq \left(1 + T_{s_{i}}\right) \cdot \left(q + \frac{1}{[\rho]_q}\right),
 \end{align*}
 where $\rho:= \content_{\frakq, i} - \content_{\frakq, i + 1} = \content_{\frakq, i} - \content_{\frakt, i}.$
\end{lemma}
\begin{proof}
By Theorem \ref{thm:dipperjamesyoungbasis}(3),
    \begin{align*}
        w_\frakq T_{s_i} = -\frac{1}{[\rho]_q}w_\frakq + w_\frakt.
    \end{align*}
    Solving the above for $w_\frakt$,  \[w_\frakt = w_\frakq \left(T_{s_i} + \frac{1}{[\rho]_q}\right).\]

    Multiplying both sides on the right by $1 + T_{s_i}$ then reveals

\begin{align*}
    w_\frakt(1 + T_{s_i}) &= w_\frakq \left(T_{s_i} + \frac{1}{[\rho]_q}\right)(1 + T_{s_i})\\
    &= w_\frakq \left(T_{s_i} + (q - 1)T_{s_i} + q + \frac{1}{[\rho]_q} + \frac{1}{[\rho]_q}T_{s_i}\right)\\
    &= w_{\frakq}(1 + T_{s_i}) \left(q + \frac{1}{[\rho]_q}\right).
\end{align*}
\end{proof}
In order to use Lemma \ref{missinglemma}, we will need to factor the term $(1+T_{s_i})$ out of $\B_{n, n-j}(q) = m_{(1^j, n - j)} \ x_{(j,n-j)} $; this is the goal of Lemma \ref{lemma:factordelta}.
\begin{lemma}\label{lemma:factordelta} For any $j < i < n,$ the element $m_{(1^j, n - j)}$ factors as

    \[ m_{(1^j, n - j)} = (1+T_{s_{i}}) \sum_{\substack{w \in \symm_{\left(1^j, n - j\right)}: \\w^{-1}(i) < w^{-1}(i + 1) }} T_w.\] 
\end{lemma}
\begin{proof}
We have
\begin{align*}
    m_{\left(1^j, n - j\right)} &= \sum_{\substack{w \in\symm_{\left(1^j, n - j\right)}}}T_w = \sum_{\substack{w \in \symm_{\left(1^j, n - j\right)}: \\w^{-1}(i) < w^{-1}(i + 1) }} T_w + \sum_{\substack{w \in \symm_{\left(1^j, n - j\right)}: \\w^{-1}(i) > w^{-1}(i + 1) }} T_w.
\end{align*}

Observe that for any $w, u \in \symm_n$, as sets 
\[\big\{w:  w  \in \symm_{(1^j, n - j)},  w^{-1}(i) > w^{-1}(i + 1) \big\} = \big \{s_{i} \ u \ : u \in \symm_{(1^j, n - j)}, \  u^{-1}(i) < u^{-1}(i + 1)\big\}.\] Further, if $u^{-1}(i) < u^{-1}(i + 1),$ then $T_{s_{i}}T_{u} = T_{s_{i}u}.$ Hence, we can continue rewriting of $m_{(1^j, n - j)}$ as
\begin{align*}
    m_{\left(1^j, n - j\right)} = \sum_{\substack{w \in \symm_{(1^j, n - j)}: \\w^{-1}(i) < w^{-1}(i + 1) }} T_w + \sum_{\substack{u \in \symm_{\left(1^j, n - j\right)}: \\u^{-1}(i) < u^{-1}(i + 1) }} T_{s_{i}}T_{u} = (1 + T_{s_{i}})\sum_{\substack{w \in \symm_{\left(1^j, n - j\right)}: \\w^{-1}(i) < w^{-1}(i + 1) }} T_w.
\end{align*}
    \end{proof}

Finally, we will combine Lemma \ref{missinglemma} and Lemma \ref{lemma:factordelta} to give our straightening procedure.

\begin{lemma}[Straightening Lemma]\label{lem:horiz-strip-diff-by-cnostant}
Assume $\HH_n(q)$ is semisimple. Fix $\lambda \vdash n$ and $\mu \subseteq \lambda$ with $|\mu| = j$.
Let $\frakt \in \syt(\lambda \sm \mu)$.
Then there exists a constant $\alpha \in \CC$ such that for any $\fraks \in \syt(\mu)$,
\[ w_{\frakt(\fraks)} \ \B_{n, n-j}(q) = \alpha  \  w_{\frakt^{\lambda \sm \mu}(\fraks)} \ \B_{n, n-j}(q).\]

\end{lemma}
\begin{proof}
By induction, we assume that the claim is already proved for all tableaux $\frakq \gtdom \frakt$ in place of $\frakt$.

If $\frakt = \frakt^{\lambda \sm \mu}$, then the claim is clear (take $\alpha = 1$). Now assume $\frakt \neq \frakt^{\lambda \sm \mu}$.
Thus, the entries of $\frakt$ do not increase from north to south.
Hence, there exists some $k \in [j+1, n-1]$ such that the entry $k + 1$ appears in a row north of the entry $k$ in $\frakt$.
Letting $\frakq: = \frakt \cdot s_k$, we thus have $\frakq \in \syt(\lambda \sm \mu)$. Thus, by our induction hypothesis there exists a constant $\alpha'\in \CC$ (independent of $\fraks$) for which
\begin{align}
    w_{\frakq(\fraks)} \ \B_{n, n-j}(q) = \alpha' \ w_{\frakt^{\lambda \sm \mu}(\fraks)} \ \B_{n, n-j}(q)
\label{pf.lem:horiz-strip-diff-by-cnostant.1}
\end{align}
for all $\fraks \in \syt(\mu).$

 By Proposition \ref{prop:factoring-Cj-symm-coset} and Lemma \ref{lemma:factordelta}
\begin{align*}
    w_{\frakt(\fraks)} \ \B_{n, n-j}(q) = w_{\frakt(\fraks)} \ m_{\left(1^{j}, n - j\right)} \ x_{(j, n - j)} = w_{\frakt(\fraks)} \left(1 + T_{s_{k}}\right)\left(\sum_{\substack{w \in \symm_{\left(1^j, n - j\right)}: \\w^{-1}(k) < w^{-1}(k + 1) }} T_w\right)x_{(j, n - j)}.
\end{align*}
Since $k > j$, the contents $\content_{\frakt(\fraks), k}$ and $ \content_{\frakt(\fraks), k + 1}$ do not depend on $\fraks$. Further, since $\frakq \gtdom \frakt$,  $\frakq(\fraks) \gtdom \frakt(\fraks)$. Hence, Lemma \ref{missinglemma} implies that there exists a constant $\beta \in \CC$ (independent of $\fraks$) for which
\begin{align*}
   w_{\frakt(\fraks)} \left(1 + T_{s_{k}}\right) = \beta  \cdot w_{\frakq(\fraks)} \left(1 + T_{s_{k}}\right)
\end{align*}
for all $\fraks \in \syt(\mu)$. Hence, 
\begin{align*}
    w_{\frakt(\fraks)} \ \B_{n, n-j}(q) &=  w_{\frakt(\fraks)} \left(1 + T_{s_{k}}\right)\left(\sum_{\substack{w \in \symm_{\left(1^j, n - j\right)}:\\ w^{-1}(k) < w^{-1}(k + 1)}}T_w \right)\ x_{(j, n - j)}\\
    &= \beta \  w_{\frakq(\fraks)} \left(1 + T_{s_{k}}\right) \left(\sum_{\substack{w \in \symm_{\left(1^j, n - j\right)}:\\ w^{-1}(k) < w^{-1}(k + 1)}}T_w \right) \ x_{(j, n - j)}\\
    &= \beta \ w_{\frakq(\fraks)} \ \B_{n, n-j}(q).
\end{align*}
Setting $\alpha = \beta\alpha'$ and recalling \eqref{pf.lem:horiz-strip-diff-by-cnostant.1} completes the proof.
\end{proof}

We are finally ready to prove that $\basis_\lambda$ spans $S^\lambda$ when $q \in \R_{>0}$.
\begin{theorem}\label{thm:orderedspanningset}
When $q \in \R_{>0}$, the set $\basis_{\lambda}$ spans $S^\lambda$.\end{theorem}
\begin{proof}
    We proceed by induction on $|\lambda| = n,$ where the base case $|\lambda| = 0$ is trivial.
	For $|\lambda| > 0$, note first that $\HH_n(q)$ is semisimple, since $q \in \R_{>0}$.
	Moreover, by Lemma~\ref{lemma:whenqisreal} (4), we can decompose $S^\lambda$ as a direct sum:
	\[
	S^\lambda = \ker \left(\RR_n(q)\vert_{S^\lambda}\right) \oplus \im \left(\RR_n(q)\vert_{S^\lambda}\right).
	\]
	By construction, $\kappa_\lambda \subseteq S^\lambda$ is a basis for $\ker(\RR_n(q)\mid_{S^\lambda})$. Consider the following subset of $\basis_\lambda$: 
        \[ \basis^{\circ}_{\lambda}:= \{ x_{\mu(u)}^{\lambda}: \mu \subsetneq \lambda, \ \ u \in \kappa_\mu \} \subseteq \basis_\lambda. \]
    Then 
    \[ \basis_\lambda = \kappa_\lambda \  \cup \ \basis^\circ_\lambda.\]

It thus remains to show that $\basis^\circ_\lambda$ spans $\im(\RR_n(q)\mid_{S^\lambda})$. Note that $\im\left(\RR_n(q\right)\vert_{S^\lambda}) \subseteq \im\left(\B_n(q)\vert_{S^\lambda}\right)$ since $\RR_n(q) = \B^*_n(q) \B_n(q)$. Therefore, it is sufficient to show that $\basis^\circ_\lambda$ spans $\im \left(\B_n(q)\vert_{S^\lambda}\right)$. In other words, we must prove that
$S^\lambda \B_n(q) \subseteq \spann \basis^\circ_\lambda$.

Since $S^\lambda$ is spanned by the $w_\frakt$ for all $\frakt \in \syt(\lambda)$, the statement $S^\lambda \B_n(q) \subseteq \spann \basis^\circ_\lambda$ will follow once we have shown that $w_\frakt \B_n(q) \in \spann \basis^\circ_\lambda$ for each $\frakt \in \syt(\lambda)$. Fix $\frakt \in \syt(\lambda)$, and recall that \eqref{eq:restrictionrewrite-wt} says $w_{\frakt} = w_{\frakt'} \ \Phi_{\lambda \sm \lambda'} \ p_{\lambda}$, where $\lambda' := \shape(\frakt')$ and $\frakt' := \frakt\vert|_{n-1}$.
It remains to show that $w_{\frakt'} \ \Phi_{\lambda \sm \lambda'} \ p_{\lambda} \ \B_n(q) \in \spann \basis^\circ_\lambda$, 

By induction, $w_{\frakt'} \in S^{\lambda'}$ can be written as a linear combination of elements $x^{\lambda'}_{\mu(u)} \in \basis_{\lambda'}$.
Thus it suffices to show that $x^{\lambda'}_{\mu(u)} \ \Phi_{\lambda \sm \lambda'} \ p_{\lambda} \ \B_n(q) \in \spann \basis^\circ_\lambda$ for each such $x^{\lambda'}_{\mu(u)} \in \basis_{\lambda'}$.

Consider such an element $x^{\lambda'}_{\mu(u)} \in \basis_{\lambda'}$ and the corresponding $\mu \subseteq \lambda'$ and $u \in \kappa_\mu$.
By its definition,
\[
x^{\lambda'}_{\mu(u)} = u \ \Phi_{\frakt^{\lambda' \sm \mu}} \ p_{\frakt^{\lambda' \sm \mu}} \ \B_{n -1, n-1 -|\mu|}(q) .
\]
Hence,
\begin{align*}
x^{\lambda'}_{\mu(u)} \ \Phi_{\lambda \sm \lambda'} \ p_{\lambda} \ \B_n(q)
&= u \ \Phi_{\frakt^{\lambda' \sm \mu}} \ p_{\frakt^{\lambda' \sm \mu}} \ \B_{n-1, n-1 -|\mu|}(q) \ \Phi_{\lambda \sm \lambda'} \ p_{\lambda} \ \B_n(q) \\
&= u \ \Phi_{\frakt^{\lambda' \sm \mu}} \ \Phi_{\lambda \sm \lambda'} \ p_{\frakt^{\lambda' \sm \mu}} \ \B_{n-1, n-1 -|\mu|}(q) \ p_{\lambda} \ \B_n(q)
\tag{by Lemma \ref{lem:commuting-phi}} \\
&= u \ \Phi_{\frakt^{\lambda' \sm \mu}} \ \Phi_{\lambda \sm \lambda'} \ p_{\frakt^{\lambda' \sm \mu}} \ p_{\lambda} \ \B_{n, n-|\mu|}(q)
\tag{since $p_\lambda$ is central and $ \B_{n, n-1 -|\mu|}(q) \ \B_n(q) = \B_{n, n - |\mu|}(q)$} \\
&= u \ \Phi_{\frakq} \ p_{\frakq} \ \B_{n, n - |\mu|}(q) ,
\end{align*}
where $\frakq \in \syt(\lambda \sm \mu)$ is the skew tableau obtained from $\frakt^{\lambda' \sm \mu}$ by adding the entry $n$ at the box of $\lambda \sm \lambda'$.

Our goal has thus become to show that $u \ \Phi_{\frakq} \ p_{\frakq} \ \B_{n, n - |\mu|} \in \spann \basis^\circ_\lambda$. Expand $u \in S^\mu$ as $\sum_{\fraks \in \syt(\mu)} c_{\fraks} \ w_{\fraks}$ with coefficients $c_\fraks \in \CC$.
Then,
\begin{align*}
u \ \Phi_{\frakq} \ p_{\frakq} \ \B_{n, n -|\mu|}(q)
&= \sum_{\fraks \in \syt(\mu)} c_{\fraks} \ w_{\fraks} \ \Phi_{\frakq} \ p_{\frakq} \  \B_{n, n -|\mu|}(q)\\
&= \sum_{\fraks \in \syt(\mu)} c_{\fraks} \ w_{\frakq(\fraks)} \ \B_{n, n -|\mu|}(q)
\tag{by Lemma \ref{lem:wts} for $\frakt = \frakq$} \\
&= \sum_{\fraks \in \syt(\mu)} c_{\fraks} \ \alpha \  w_{\frakt^{\lambda \sm \mu}(\fraks)} \  \B_{n, n -|\mu|}(q),
\tag{by Lemma \ref{lem:horiz-strip-diff-by-cnostant} for $\frakt = \frakq$}
\end{align*}
where $\alpha \in \CC$ is a constant independent of $\fraks$.
Hence,
\begin{align*}
u \ \Phi_{\frakq} \ p_{\frakq} \  \B_{n, n -|\mu|}(q)
&= \alpha \sum_{\fraks \in \syt(\mu)} c_{\fraks} \  w_{\frakt^{\lambda \sm \mu}(\fraks)} \  \B_{n, n -|\mu|}(q)\\
&= \alpha \sum_{\fraks \in \syt(\mu)} c_{\fraks} \  w_\fraks \ \Phi_{\frakt^{\lambda \sm \mu}} \ p_{\frakt^{\lambda \sm \mu}} \  \B_{n, n -|\mu|}(q)
\tag{by Lemma \ref{lem:wts} for $\frakt = \frakt^{\lambda \sm \mu}$} \\
&= \alpha u\ \Phi_{\frakt^{\lambda \sm \mu}} \ p_{\frakt^{\lambda \sm \mu}} \  \B_{n, n -|\mu|}(q)
\tag{by the expansion of $u$} \\
&= \alpha x^{\lambda}_{\mu(u)} \in \spann \basis^\circ_\lambda ,
\end{align*}
as desired.
\end{proof}

\subsection{A simplified description of $\basis_\lambda$}\label{section:simpledescriptionbasis}
We originally defined $\basis_\lambda$ as 
    \[ \basis_{\lambda}:= \{x_{\mu(u)}^{\lambda}:= u \  \Phi_{\frakt^{\lambda \sm \mu}} \  p_{\frakt^{\lambda \sm \mu}} \  \B_{n, n-|\mu|}(q): u \in \kappa_\mu, \ \  \mu \subseteq \lambda \}. \]
    We now show that the elements of $\basis_\lambda$ are either zero, or can be rewritten as a set with elements indexed by horizontal strips.
    
\begin{lemma}\label{lemma:xtoy}
Assume $\HH_n(q)$ is semisimple.
Fix $\lambda \vdash n$ and let $\mu \subseteq \lambda$. For any $u \in \kappa_\mu,$ 
   \[ x_{\mu(u)}^{\lambda}:=u \  \Phi_{\frakt^{\lambda \sm \mu}} \  p_{\frakt^{\lambda \sm \mu}} \  \B_{n, n-|\mu|}(q) = \begin{cases} 0 & \lambda \sm \mu \textrm{ is not a horizontal strip}\\
    y_{\mu(u)}^{\lambda}:= u\  \Phi_{\frakt^{\lambda \sm \mu}}  \  \B_{n, n-|\mu|}(q) \ p_{\lambda} & \lambda \sm \mu \textrm{ is a horizontal strip}.
    \end{cases}\] 
\end{lemma}

 Recall that by Lemma \ref{cor:nonhorizontalstripgives0}, when $\lambda \sm \mu$ is not a horizontal strip and $u \in S^\mu$, we have
\[    u \  \Phi_{\frakt^{\lambda \sm \mu}}  \ \B_{n, n-|\mu|}(q) \ p_{\lambda} = 0. \]
To prove Lemma \ref{lemma:xtoy}, we would like to apply the \emph{Horizontal Strip Lemma} (Lemma \ref{cor:nonhorizontalstripgives0}). However, in order to do so, we must first transform the $x_{\mu(u)}^{\lambda}$ to be of the form  $u \  \Phi_{\frakt^{\lambda \sm \mu}}  \  \B_{n, n-|\mu|}(q) \ p_{\lambda}$ rather than  $u \  \Phi_{\frakt^{\lambda \sm \mu}} \  p_{\frakt^{\lambda \sm \mu}} \  \B_{n, n-|\mu|}(q)$. This is the purpose of the Lemma \ref{technicallemma}, whose proof relies on two basic properties of $p_\lambda$ and $p_\frakt$, proved below in Lemmas \ref{lem:dominance-1} and \ref{lem:dominance-ws-pt}.

\begin{lemma}\label{lem:dominance-1}
        Assume $\HH_n(q)$ is semisimple. Let $\lambda$ and $\mu$ be two partitions of $n$
	such that $\lambda \not\geqdom \mu$.
	Then 
 \[ W^\mu \ p_\lambda = 0. \]
\end{lemma}

\begin{proof}
As discussed in Section \ref{sec:idempotents}, $W^\mu \  p_\lambda$ is a right $\HH_n(q)$-module isomorphic to a direct sum of some number of copies of $S^\lambda.$ On the other hand, by \cite[Corollary 4.12 (i)]{DipperJames}, the module $W^\mu \supseteq W^\mu \  p_\lambda$ does not contain a copy of $S^\nu$ if $\nu \not\geqdom \mu$. (Note that our $W^\mu$ is called $M^\mu$ in \cite{DipperJames}, whereas our $S^\lambda$ agrees with the $D^\lambda$ of \cite{DipperJames} by \cite[Theorem 4.15 (i)]{DipperJames}). Thus, $W^\mu \  p_\lambda$ must be $0.$
\end{proof}

\begin{lemma}\label{lem:dominance-ws-pt}
    Assume $\HH_n(q)$ is semisimple. Let $\lambda \vdash n$ and $\frakq, \frakt \in \syt(\lambda)$ be such that $\frakt \not\geqdom \frakq$.
	Then
 \[ \word(\frakq) \ p_\frakt = 0. \]
\end{lemma}

\begin{proof}
By assumption, $\frakt \not\geqdom \frakq$.
Thus, there exists some $1 \leq j \leq n$ such that
$\shape \left(\frakt \vert_j\right) \not\geqdom \shape \left(\frakq \vert_j\right)$.
Consider this $j$, and set
$\lambda' := \shape \left(\frakt \vert_j\right)$
and $\mu' := \shape \left(\frakq \vert_j\right)$,
so that $\lambda' \not\geqdom \mu'$.
Hence, Lemma~\ref{lem:dominance-1} yields
$W^{\mu'} p_{\lambda'} = 0$.
In particular,
$\word(\frakq\vert_j) p_{\lambda'} = 0$,
since $\word(\frakq\vert_j) \in W^{\mu'}$.

Recall that
\[
p_\frakt
= p_{\shape\left(\frakt\vert_1\right)} p_{\shape\left(\frakt\vert_2\right)} \cdots p_{\shape\left(\frakt\vert_n\right)} .
\]
The $j$-th factor $p_{\shape\left(\frakt\vert_j\right)} = p_{\lambda'}$ in this product is central in $\HH_j(q)$ and thus commutes with all the preceding factors, allowing us to rewrite the above as
\[
p_\frakt = p_{\lambda'}
\left(p_{\shape\left(\frakt\vert_1\right)} p_{\shape\left(\frakt\vert_2\right)} \cdots p_{\shape\left(\frakt\vert_{j-1}\right)}\right) \left(p_{\shape\left(\frakt\vert_{j+1}\right)} p_{\shape\left(\frakt\vert_{j+2}\right)} \cdots p_{\shape\left(\frakt\vert_n\right)}\right).
\]
Thus, in order to prove $\word(\frakq) \ p_\frakt = 0$, it suffices to show that $\word(\frakq) \ p_{\lambda'} = 0$, which we now do.

Write $\frakq$ as $\frakq'(\frakq\vert_j)$, where $\frakq' \in \syt(\lambda \sm \mu')$. Then, $\word(\frakq) = \word(\frakq\vert_j) \cdot \Phi_{\frakq'}$.
Hence,
\[
\word(\frakq) \ p_{\lambda'}
= \left(\word(\frakq\vert_j) \cdot \Phi_{\frakq'}\right) p_{\lambda'}
= \left(\word(\frakq\vert_j) \ p_{\lambda'}\right) \cdot \Phi_{\frakq'}
\qquad \text{by Lemma \ref{lem:commuting-phi}} .
\]
From $\word(\frakq\vert_j) \ p_{\lambda'} = 0$, we thus obtain
$\word(\frakq) \ p_{\lambda'} = 0$.
\end{proof}

We may now prove Lemma \ref{technicallemma}, which will allow us to apply the \emph{Horizontal Strip Lemma}.
\begin{lemma}\label{technicallemma}
    Assume $\HH_n(q)$ is semisimple. Let $\mu \subseteq \lambda$ and $u \in S^\mu$. Then,
    \[ u \ \Phi_{\frakt^{\lambda \sm \mu}} \ p_{\frakt^{\lambda \sm \mu}}  = u \ \Phi_{\frakt^{\lambda \sm \mu}} \  p_\lambda. \]
\end{lemma}

\begin{proof}
The equality we are proving is linear in $u$, and the Specht module $S^\mu$ is spanned by the elements $w_\fraks$ for $\fraks \in \syt(\mu)$.
Thus, it suffices to prove it for $u = w_\fraks$.

Let $\fraks \in \syt(\mu)$. We must show that $w_\fraks \ \Phi_{\frakt^{\lambda \sm \mu}} \ p_{\frakt^{\lambda \sm \mu}}  = w_\fraks \ \Phi_{\frakt^{\lambda \sm \mu}} \  p_\lambda$. Combining \eqref{tower2} with \eqref{def:skew-pt}, we can easily see that
\begin{equation}
p_{\frakt^{\lambda \sm \mu}(\fraks)} = p_\fraks \ p_{\frakt^{\lambda \sm \mu}}.
\label{pf.technicallemma.1}
\end{equation}

Let  $|\mu| = j$ and $|\lambda| = n$.
By the tower rule (Proposition \ref{lem:tower-rule}: \eqref{tower1}),
\begin{align*}
w_\fraks  \  \Phi_{\frakt^{\lambda \sm \mu}} \  p_\lambda
&= \sum_{\substack{\frakq \in \syt(\lambda)} } w_{\fraks} \  \Phi_{\frakt^{\lambda \sm \mu}} \  p_{\frakq} \\
&= \sum_{\substack{\frakq \in \syt(\lambda)} } \word(\fraks) \  \Phi_{\frakt^{\lambda \sm \mu}} \  p_{\fraks} \  p_{\frakq}
\tag{by \eqref{eq.def:youngseminormalunits.wt=} and Lemma \ref{lem:commuting-phi}} \\
&= \sum_{\substack{\frakq \in \syt(\lambda)\\ \frakq|_{j} = \fraks}} \word(\fraks) \  \Phi_{\frakt^{\lambda \sm \mu}} \  p_{\frakq}
\tag{by Proposition \ref{lem:interaction-of-jucys-murphys-with-pts} (1)} \\
&= \sum_{\substack{\frakq \in \syt(\lambda) \\ \frakq|_{j} = \fraks}} \word(\frakt^{\lambda\sm\mu}(\fraks)) \  p_{\frakq}.
\tag{by \eqref{eq.def:phi.wordPhi}}
\end{align*}
Note that for all $\frakq \neq \frakt^{\lambda\sm \mu}(\fraks)$ with $\frakq \vert_j = \fraks,$ we have
$\frakt^{\lambda\sm \mu}(\fraks) \gtdom \frakq$
and therefore
$\frakq \not\geqdom \frakt^{\lambda\sm \mu}(\fraks)$,
so that $\word(\frakt^{\lambda\sm\mu}(\fraks)) \  p_{\frakq} = 0$
by Lemma \ref{lem:dominance-ws-pt}.
The above sum thus becomes 
\begin{align*}
    w_\fraks  \  \Phi_{\frakt^{\lambda \sm \mu}} \  p_\lambda &= \word(\frakt^{\lambda \sm \mu}(\fraks)) \ p_{\frakt^{\lambda \sm \mu}(\fraks)} \\
    &= \word(\fraks)  \ \Phi_{\frakt^{\lambda \sm \mu}} \ p_\fraks \ p_{\frakt^{\lambda \sm \mu}}
	\tag{by \eqref{pf.technicallemma.1} and \eqref{eq.def:phi.wordPhi}}\\
    &=  w_\fraks  \  \Phi_{\frakt^{\lambda \sm \mu}} \  p_{\frakt^{\lambda \sm \mu}} ,
	\tag{by \eqref{eq.def:youngseminormalunits.wt=} and Lemma \ref{lem:commuting-phi}}
\end{align*}
as desired.
\end{proof}

Lemma \ref{lemma:xtoy} now easily follows from Lemma \ref{technicallemma}. 

\begin{proof}[Proof of Lemma \ref{lemma:xtoy}]
By Lemma \ref{technicallemma},
 \[ u \ \Phi_{\frakt^{\lambda \sm \mu}} \ p_{\frakt^{\lambda \sm \mu}}  = u \ \Phi_{\frakt^{\lambda \sm \mu}} \  p_\lambda. \]
Using the fact that $p_\lambda$ is central, we have 
\[ x_{\mu(u)}^{\lambda} = u \ \Phi_{\frakt^{\lambda \sm \mu}} \ p_{\frakt^{\lambda \sm \mu}} \ \B_{n, n-|\mu|}(q)= u \  \Phi_{\frakt^{\lambda \sm \mu}}  \ p_{\lambda} \  \B_{n, n-|\mu|}(q) = u \  \Phi_{\frakt^{\lambda \sm \mu}} \ \B_{n, n-|\mu|}(q) \ p_{\lambda}. \]

Thus by Lemma \ref{cor:nonhorizontalstripgives0}, if $\lambda \sm \mu$ is not a horizontal strip,
\[ u \  \Phi_{\frakt^{\lambda \sm \mu}}  \  \B_{n, n-|\mu|}(q) \ p_{\lambda} = 0. \]
\end{proof}

Having shown that $x_{\mu(u)}^{\lambda} = 0$ if $\lambda \sm \mu$ is not a horizontal strip, we will now slightly abuse notation and remove the copies of 0 from $\basis_\lambda$, so that
\[ \basis_\lambda := \{  y_{\mu(u)}^{\lambda} : = u\  \Phi_{\frakt^{\lambda \sm \mu}}  \  \B_{n, n-|\mu|}(q) \ p_{\lambda}: u \in \kappa_\mu \textrm{ for } \lambda \sm \mu \textrm{ a horizontal strip} \}. \]

\section{Proof of the Main Theorem: Spectrum for random-to-random}\label{section:maintheorem}

Our final task is to apply our work in Sections \ref{sec:flagintroduction}---\ref{section:spanning} to compute the full spectrum of $\RR_n(q)$ for any $q \in \CC$. We show in Section \ref{sec:eigenbasisprimeq} that when $q \in \R_{>0}$, the set $\basis_\lambda$ defines an eigenbasis of $S^\lambda$ for every $\lambda \vdash n$. We use this to prove our main result, Theorem \ref{thm:summarythm} (Theorem \ref{thm:intromainthm} from the Introduction) in Section \ref{sec:proofofmainthm}. Finally, in Section \ref{sec:further-results-for-positive-q}, we discuss special cases of Theorem \ref{thm:summarythm}, including Corollary \ref{cor:introsecondeigen} describing $\eigen_{(n) \sm \emptyset}(q)$ and $\eigen_{(n-1,1) \sm (1,1)}(q)$ from the introduction. 

\subsection{An $\RR_n(q)$-eigenbasis when $q$ is positive}\label{sec:eigenbasisprimeq}

We are at last ready to show that $\basis_\lambda$  is a basis for $S^\lambda$ when $q \in \R_{>0}$. Recall that for a fixed $\lambda$ of size $n$, by Lemma \ref{lemma:xtoy} the nonzero elements of $\basis_\lambda$ are of the form 
\[ \big \{ y_{\mu(u)}^{\lambda}= u\  \Phi_{\frakt^{\lambda \sm \mu}}  \  \B_{n, n-|\mu|}(q) \ p_{\lambda}: u \in \kappa_\mu \textrm{ and } \lambda \sm \mu \textrm{ is a horizontal strip}  \big \}, \]
where $\kappa_\mu$ is a basis for $ \ker \left(\RR_{|\mu|}(q) \vert_{S^\mu} \right)$. 

 We have already shown in Theorem \ref{thm:orderedspanningset} that for any $\lambda \vdash n$, the elements of $\basis_\lambda$ span $S^\lambda$ when $q \in \R_{>0}$. In Proposition \ref{cor:dim-restr-s-la} and Theorem \ref{thm:eigenbasisSlambda} below, we argue that the only additional ingredients needed to conclude that $\basis_\lambda$ is a  basis of $S^\lambda$ are that 
\[ \dim \ker \RR_j(q) = d_j \quad \quad \textrm{ for } \quad \quad 0 \leq j \leq n. \]

\begin{prop}\label{cor:dim-restr-s-la}
   Suppose $q \in \R_{>0}$.
   Then, for all $0 \leq j \leq n$, we have $\ker \RR_j(q) \cong \DD_j(q)$, and thus
   \[\dim \ker \left(\RR_{j}(q) \vert_{S^\mu}\right) = d^\mu
   \qquad \text{for any partition } \mu \vdash j .\]
\end{prop}
\begin{proof}
Since $q \in \R_{>0}$, the algebra $\HH_j(q)$ is semisimple, so that Theorem \ref{cor:prime-eigenspace} yields $\ker \B^*_j(q) \cong \DD_j(q)$.
Furthermore, Lemma \ref{lemma:whenqisreal} (2) shows that $\ker \RR_j(q) = \ker \B_j^\ast(q)$.
Hence,
\[\ker \RR_j(q) = \ker \B_j^\ast(q) \cong \DD_j(q) \cong \bigoplus_{\frakt \in \D_j} \left(S^{\shape(\frakt)}\right)^*.\]
Now apply Corollary \ref{cor:righteigenleftmodule} (2).
\end{proof}

\begin{theorem}\label{thm:eigenbasisSlambda}
If $q \in \R_{>0}$, then the set $\basis_{\lambda}$ is an eigenbasis for the right action of $\RR_n(q)$ on $S^\lambda$.
\end{theorem}

\begin{proof}
    By Theorem \ref{thm:orderedspanningset} the set $\basis_\lambda$ spans $S^\lambda$ and by Lemma \ref{lemma:xtoy} consists of elements
    \[ \basis_\lambda = \{ y_{\mu(u)}^{\lambda}: u \in \kappa_\mu \textrm{ for } \lambda \sm \mu \textrm{ a horizontal strip} \}. \]
    By Theorem \ref{thm:eigenvectors}, each element of $\basis_\lambda$ is an eigenvector of $\RR_n(q)$. 
    
    Thus it remains to prove that the $y_{\mu(u)}^{\lambda}$ are linearly independent. 
    We will show that $\lvert \basis_\lambda\rvert  = f^\lambda$, from which the result will follow. Since $\basis_\lambda$ spans $S^\lambda$,
\[ \dim(S^\lambda) = f^\lambda \leq |\basis_\lambda|. \]

By Proposition \ref{cor:dim-restr-s-la}, we have $\dim \ker \left(\RR_{|\mu|}(q) \vert_{S^\mu}\right) = d^\mu$, and thus $\basis_\lambda$ has at most $d^\mu$ nonzero $y_{\mu(u)}^{\lambda}$ for each horizontal strip $\lambda \sm \mu$ of $\lambda$, so
\[ \lvert \basis_\lambda \rvert \leq \sum_{\lambda \sm \mu \textrm{ a horizontal strip}} d^\mu. \]
Comparing $s_\lambda$-coefficients in \eqref{eq:frobeniusderangement} reveals
\[ \sum_{\lambda \sm \mu \textrm{ a horizontal strip}} d^\mu = f^\lambda,\]
and so $|\basis_\lambda| = f^\lambda$.  
\end{proof}

\begin{remark}[Eigenbasis for $\HH_n(q)$]
\label{rmk:eigenbasisforH}\rm
Note that it is not difficult to construct an $\RR_n(q)$-eigenbasis of $\HH_n(q)$ using the eigenbasis for each $S^\lambda$ by embedding each $y_{\mu(u)}^\lambda$ in $\HH_n(q)$. 

For any composition $\nu$, define the injective linear map
\begin{align*}
    \iota: W^\nu &\hookrightarrow \HH_n(q)\\
    {\bf w} &\longmapsto m_\nu T_{\sigma_{\bf w}},
\end{align*}
where $\sigma_{\bf w}$ is the unique element in $X_\nu$ for which ${\bf w} = \word(\frakt^\nu) \sigma_{\bf w} .$ We abuse notation and write 
\[ \iota(y_{\mu(u)}^{\lambda}) := \iota( u \cdot \Phi_{\frakt^{\lambda \sm \mu}}) \cdot \B_{n, n-|\mu|}(q)\cdot p_\lambda  \in \HH_n(q).\]

Then using the construction in \cite{Mathas} (see Remark \ref{remark:mathasdictionary}) a basis for the $S^\lambda$-isotypic component of $\HH_n(q)$ is 
\[ \left\{ p_{\frakt} \cdot m_\lambda \cdot   \iota(y_{\mu(u)}^{\lambda}) : \frakt \in \syt(\lambda) \textrm{ and } y_{\mu(u)}^{\lambda} \in \basis_\lambda \right \}.\] 
Varying over all $\lambda$ gives a $\RR_n(q)$-eigenbasis of $\HH_n(q)$.

\end{remark}
We immediately obtain the following.

\begin{cor}\label{thm:realpositivecase}
    Suppose $q \in \R_{>0}$. Then the following holds:
    \begin{enumerate}
		\item The right action of $\RR_n(q)$ on $\HH_n(q)$ is diagonalizable. \smallskip
        \item The $\eigen(q)$-eigenspace of $\RR_n(q)$ carries the left $\HH_n(q)$-representation \[\bigoplus_{\substack{\lambda \sm \mu \textrm{ a horizontal strip:} \\ \eigen_{\lambda \sm \mu}(q) = \eigen(q)}}\underbrace{\left(\left(S^\lambda\right)^\ast \right) \oplus \cdots \oplus\left(\left(S^\lambda\right)^\ast \right)}_{d^\mu \text{ copies}}.\] \smallskip
     
    \end{enumerate}
\end{cor}

\begin{proof}
    \noindent (1) The first claim follows immediately from Theorem~\ref{thm:eigenbasisSlambda}. \smallskip
	
	\noindent (2) For a fixed $S^\lambda$, the construction of the eigenbasis $\basis_\lambda$ and Proposition \ref{cor:dim-restr-s-la} show that $\basis_\lambda$ contains $d^\mu$ many $\eigen_{\lambda \sm \mu}(q)$-eigenvectors of $\RR_n(q)$ coming from $\kappa_\mu$.
	Now recall the bimodule decomposition \eqref{eq:leftrightbimoduledecomposition}.
\end{proof}

\subsection{Proof of Theorem \ref{thm:intromainthm}: Spectrum of $\RR_n(q)$}\label{sec:proofofmainthm}

Theorem \ref{thm:eigenbasisSlambda} gives a complete understanding of how $\RR_n(q)$ acts on $S^\lambda$ when $q \in \R_{>0}$. We apply this knowledge to compute the spectrum for $\RR_n(q)$ for $q \in \CC$.

\begin{theorem}[Theorem \ref{thm:intromainthm}]\label{thm:summarythm} For $q \in \CC$, the action of $\RR_n(q)$ on $\HH_n(q)$ has the following properties:
\begin{enumerate}
    \item All eigenvalues of $\RR_n(q)$ are of the form 
    \[  \eigen_{\lambda \sm \mu}(q) = q^{n} \content_{\lambda \sm \mu}(q) + \sum_{k=|\mu|+1}^{n} q^{n-k} \  [k]_q,  \]
    where $\lambda \sm \mu$ is a horizontal strip with $|\lambda| = n$ and $0 \leq |\mu| \leq n$.
\medskip
    \item 
The (algebraic) multiplicity of a fixed eigenvalue $\eigen(q)$ is given by
\[ \sum_{\substack{\lambda \sm \mu \textrm{ a horizontal strip}: \\ \eigen_{\lambda \sm \mu}(q) = \eigen(q)}} f^\lambda d^\mu, \]
where $d^\mu = |\kappa_\mu|$ is the number of desarrangement tableaux of shape $\mu$. \medskip
\item Every $\eigen_{\lambda \sm \mu}(q)$ is a polynomial in $q$ with non-negative integer coefficients.
\end{enumerate}
\end{theorem}
Note that when $\lambda = \mu$, the formula $\eigen_{\lambda \sm \lambda}(q)$ is vacuous, and so
 $\eigen_{\lambda \sm \lambda}(q) = \eigen_{\emptyset}(q) = 0$ .

\begin{proof} 
The matrix representing the action of $\RR_n(q)$ on the $T_w$ basis of $\HH_n(q)$ has entries in $\Z[q].$ Hence, by Lemma \ref{lem:suffices-to-prove-for-infinite-q} (1), it suffices to prove (1) and (2) for infinitely  many $q$. We will do so in the case that $q \in \R_{>0}$. \smallskip

  \noindent  (1) Let $q \in \R_{>0}$. By Theorem \ref{thm:eigenbasisSlambda}, for any $\lambda \vdash n$ the set $\basis_\lambda$ is a basis of $S^\lambda$. By Theorem \ref{thm:eigenvectors}, each $y_{\mu(u)}^{\lambda}$ is an eigenvector of $\RR_n(q)$ with eigenvalue $\eigen_{\lambda \sm \mu}(q)$. The decomposition of $\HH_n(q)$ as a direct sum of Specht modules then gives the claim.  \medskip

    \noindent (2) Again let $q \in \R_{>0}$. For a fixed $S^\lambda$, the construction of the eigenbasis $\basis_\lambda$ and Proposition \ref{cor:dim-restr-s-la} show that $\basis_\lambda$ contains $d^\mu$ many $\eigen_{\lambda \sm \mu}(q)$-eigenvectors of $\RR_n(q)$ coming from $\kappa_\mu$.
	Since $\HH_n(q)$ is a direct sum of $f^\lambda$ copies of each Specht module $S^\lambda$, we thus obtain an eigenbasis of $\HH_n(q)$ that contains $f^\lambda d^\mu$ many $\eigen_{\lambda \sm \mu}(q)$-eigenvectors of $\RR_n(q)$ coming from each horizontal strip $\lambda \sm \mu$.
	Summing over all $\lambda \sm \mu$ where $\eigen(q) = \eigen_{\lambda \sm \mu}(q)$ then gives the claim. \medskip

 \noindent    (3) Finally, we will prove that for all $q \in \CC$ \begin{equation}\label{eq:eigenvalueformula} 
       q^{n} \content_{\lambda \sm \mu}(q) + \sum_{k=|\mu|+1}^{n} q^{n-k} \  [k]_q \in \Z_{\geq 0}[q].
     \end{equation}
     Since 
     \begin{align}
          q^{n} \content_{\lambda \sm \mu}(q) + \sum_{k=|\mu|+1}^{n} q^{n-k} \  [k]_q  &= \sum_{k = |\mu| + 1}^n q^{n - k}\left(q^k \content_{\frakt^{\lambda \sm \mu}, k}(q) + [k]_q\right) \nonumber\\
          &= \sum_{k = |\mu| + 1}^n q^{n - k}\left[\content_{\frakt^{\lambda \sm \mu}, k} + k \right]_q,
		  \label{eq:positive-poly}
     \end{align}
     the claim follows by observing that $\content_{\frakt^{\lambda \sm \mu}, k} \geq -k$ for all $k.$ 
\end{proof}

\begin{remark}\rm
Note that Theorem \ref{thm:summarythm} neither claims that $\RR_n(q)$ acts semisimply on $\HH_n(q)$ nor explains the left $\HH_n(q)$-representation structure on eigenspaces for general $q \in \CC$. In fact, it is not true that $\RR_n(q)$ is always diagonalizable, even when $\HH_n(q)$ is semisimple. 
\end{remark}

\begin{remark}[Degree of $\eigen_{\lambda \sm \mu}(q)$ as a polynomial in $q$] \label{rem:degree}\rm
For $\mu \subsetneq \lambda,$ define
\[ C_{\lambda \sm \mu} = \max_{|\mu|+1 \leq k \leq n} \left \{ \content_{\frakt^{\lambda \sm \mu},k} \right\}.\]
Write $\deg(\eigen_{\lambda \sm \mu}(q))$ to be the $q$-degree of $\eigen_{\lambda \sm \mu}(q)$. Then 
\eqref{eq:positive-poly} reveals that 
\[ \deg(\eigen_{\lambda \sm \mu}(q)) =   n + C_{\lambda \sm \mu} - 1.\] 
\end{remark}

\subsection{Special cases of Theorem \ref{thm:intromainthm}}\label{sec:further-results-for-positive-q}

We will use Theorem \ref{thm:summarythm} and to compute the eigenvalues $\eigen_{\lambda \sm \mu}(q)$ in some special cases of interest.

\vskip.1in

\paragraph{{\bf Hook eigenvalues $\eigen_{(n-k,1^k) \sm (j-k,1^k)}(q)$}}
In the case that $\lambda = (n-k,1^k)$ and $\mu = (j-k,1^k)$, our formula $\eigen_{\lambda \sm \mu}(q)$ simplifies dramatically. 
\begin{cor}\label{cor:introsecondeigen}
The case $\lambda = (n-k,1^k)$ and $\mu = (j-k,1^k)$ gives the eigenvalue
\[ \eigen_{(n-k,1^k) \sm (j-k,1^k)}(q) = [n-j]_q \ [n+j-k]_q.\]
\end{cor}
\begin{proof}
Using the formula in Theorem \ref{thm:summarythm}(1) and rearranging gives:
\begin{align*}
    \eigen_{(n-k,1^k) \sm (j-k,1^k)}(q) &= \sum_{i=j}^{n-1} q^n   [i-k]_q  + q^{n-(i+1)} [i+1]_q \\
    &= \sum_{i=j}^{n-1} q^{n} [i-k]_q + q^{i-j} [n+j-i]_q \\
    &= \sum_{i=j}^{n-1} q^{i-j} \big( [n+j-k]_q \big) \\
    &= [n-j]_q [n+j-k]_q.
\end{align*}
\end{proof}

We can specialize Corollary \ref{cor:introsecondeigen} to obtain the largest and second largest eigenvalues of $\RR_n(q)$ when $q \in \R_{>0}$.



\begin{cor}\label{prop:topeigenvalue}
For any $q \in \CC$, when $\lambda = (n)$ and $\mu = \emptyset$, we have 
\[ \eigen_{(n) \sm \emptyset}(q) = ([n]_q)^2,\]
which has left and right eigenvector 
\[ m_{(n)} = \sum_{w \in \symm_n} T_w. \]
When $q \in \R_{>0}$, the eigenvalue $\eigen_{(n) \sm \emptyset}(q) = ([n]_q)^2$ is the largest eigenvalue of $\RR_n(q)$.
\end{cor}
\begin{proof}
The formula $ \eigen_{(n) \sm \emptyset}(q) = ([n]_q)^2$ comes from Corollary \ref{cor:introsecondeigen} specialized to the case that $j=k=0$.

It is not hard to check that for any $q \in \CC$, the element $m_{(n)}$ is always an eigenvector of $\B_n(q), \B_n^*(q)$ and $\RR_n(q)$ with eigenvalues $[n]_q$ and $\eigen_{(n) \sm \emptyset}(q) = ([n]_q)^2$, respectively. This follows from the fact that $\RR_n(q) = \B_n^*(q) \B_n(q)$, the element $m_{(n)}$ is central in $\HH_n(q)$ and by Example \ref{ex:triv-rep}
\[ m_{(n)} \B_n^*(q) = [n]_q m_{(n)}  =\B_n^*(q)  m_{(n)};   \quad \quad m_{(n)} \B_n(q) = [n]_q m_{(n)} =  \B_n(q) m_{(n)}. \]
It is clear that $\eigen_{(n) \sm \emptyset}(q) = ([n]_q)^2$ is the largest eigenvalue when $q \in \R_{>0}$ since the contents of $(n) \sm \emptyset$ are $\{ 0, 1, \cdots, n-1 \}$, which dominate the contents of any other horizontal strip $\lambda \sm \mu$.
\end{proof}

Recall from the introduction that we defined in \eqref{eq:probr2rinto} the random walk $\widetilde{\RR}_n(q)$ when $q^{-1} \in (0,1]$. 
In the theory of Markov chains, a \emph{stationary distribution} is an eigenvector for eigenvalue $1$ (after normalization). In the case of $\widetilde{\RR}_n(q)$, this corresponds to the eigenvalue $\eigen_{(n) \sm \emptyset}(q)$. 

Bufetov showed that any random walk on $\HH_n(q)$ defined by right multiplication with elements of $\HH_n(q)$ has stationary distribution given by the \emph{Mallows measure} \cite[Prop 2.3]{bufetov2020interacting}, described below.

\begin{definition}[Mallows Measure]
The \emph{Mallows measure} of the symmetric group is      
\[ \mathcal{M}(\symm_n, q^{-1}):= \sum_{w \in \symm_n} q^{\ell(w)} \tT_{w}  .\]
\end{definition}
That $\widetilde{\RR}_n(q), \widetilde{\B}_n^*(q)$ and $\widetilde{\B}_n(q)$ define random walks on $\HH_n(q)$ thus immediately implies the following. 

\begin{cor}\label{cor:uniformdist}
The Mallows measure $\mathcal{M}(\symm_n, q^{-1})$ is a stationary distribution for $\widetilde{\B}_n^*(q)$, $\widetilde{\B}_n(q)$ and $\widetilde{\RR}_n(q)$.
\end{cor}
Note that in our earlier notation, 
\[ \mathcal{M}(\symm_n, q^{-1}) = m_{(n)}  = \sum_{w \in \symm_n} T_w. \]
Thus Corollary \ref{cor:uniformdist} is consistent with Corollary \ref{prop:topeigenvalue}. 


The second largest eigenvalue of a Markov chain is of interest when computing the mixing time of the process. In \cite{DiekerSaliola}, Dieker--Saliola show that the second largest eigenvalue in the $q=1$ case is $\eigen_{(n-1,1) \sm (1,1)} = (n-2)(n+1)$. We show that this formula generalizes to $\HH_n(q)$ as follows. 
{\begin{cor}\label{thm:second-largest-eigenval}
For any $q \in \CC$ and $n \geq 2$, when $\lambda = (n-1,1)$ and $\mu = (1,1)$, we have
    \[\eigen_{(n-1,1) \sm (1,1)}(q) = [n-2]_q \ [n+1]_q.\] 
    Moreover, when $q \in \R_{>0},$ the eigenvalue $\eigen_{(n-1,1) \sm (1,1)}(q)$ is the second largest eigenvalue of $\RR_n(q)$ and occurs with multiplicity $n-1.$
\end{cor}\begin{proof}
Setting $j = 2$ and $k=1$ in Corollary \ref{cor:introsecondeigen} gives $ \eigen_{(n-1,1) \sm (1,1)}(q) = [n-2]_q [n+1]_q$.

  For the second claim, fix a non-empty horizontal strip $\lambda \sm \mu$  with $d^\mu \neq 0$ and $|\lambda| = n$ and $|\mu|= j$ such that
   \begin{equation}\label{eq:corsecondlargest.1} \lambda \sm \mu \neq (n) \sm \emptyset \quad \quad \textrm{ and}\quad \quad \lambda \sm \mu \neq (n-1, 1) \sm (1,1).\end{equation}
   It suffices to show that 
   $\eigen_{(n-1,1)\sm (1,1)}(q) - \eigen_{\lambda \sm \mu}(q)$ is strictly positive.
    
    First, assume that $j = 2.$ By \eqref{eq:corsecondlargest.1}, the horizontal strip $\lambda \sm \mu$ is forced to be $(n-2,1,1) \sm (1,1).$ It is straightforward to use Theorem \ref{thm:summarythm}(1) to check that \[\eigen_{(n-1,1) \sm (1,1)}(q) - \eigen_{(n-2,1, 1) \sm (1,1)}(q) = q^n([n-2]_q - [-2]_q)\] which is positive for $q > 0.$
    
    Now, assume $j > 2.$ Observe that \eqref{eq:corsecondlargest.1} forces $\mu$ to contain the partition $(1,1)$. Hence the contents of the $n-j$ cells in $\lambda \sm \mu$ are bounded above by $j-1, j, j + 1, \cdots, n - 2,$ meaning that  
    \[ \content_{\lambda \sm \mu}(q) \leq [j-1]_q + [j]_q + \cdots + [n-2]_q. \] 
    Therefore, 
    \begin{align*}
       \eigen_{(n-1,1) \sm (1,1)}(q) - \eigen_{\lambda \sm \mu}(q) &= q^n \left([1]_q + [2]_q + \cdots + [n-2]_q - \content_{\lambda \sm \mu}(q)\right) + \sum_{k = 3}^{j} q^{n-k}[k]_q\\
       &\geq q^n \left([1]_q + [2]_q + \cdots + [j-2]_q\right)+ \sum_{k = 3}^{j } q^{n-k}[k]_q
    \end{align*}
which is strictly positive for $q \in \R_{>0}$.

Finally, to see the multiplicity of $\eigen_{(n-1,1) \sm (1,1)}(q)$ is $n-1$, note that $d^{(1,1)} = 1$ and $f^{(n-1,1)} = n-1$. By the above argument, when $q \in \R_{>0}$ there are no other eigenvalues $\eigen(q)$ which coincide with $\eigen_{(n-1,1) \sm (1,1)}(q)$.
\end{proof}

\vskip.1in
\noindent
\paragraph{\bf{The kernel of $\RR_n(q)$}} 
When $\lambda = \mu,$ both sums in the formula for $\eigen_{\lambda \sm \mu}(q)$ in Theorem \ref{thm:summarythm}(1) are empty, hence $\eigen_{\lambda \sm \lambda}(q) = \eigen_{\emptyset}(q) = 0$ for all $q \in \CC.$ However, any solution to $\eigen_{\lambda \sm \mu}(q) = 0$ will also contribute to the kernel of $\RR_n(q)$. When $q \in \R_{>0}$, we have that $\eigen_{\lambda \sm \mu}(q) > 0$ whenever $\mu \neq \lambda$, confirming that the multiplicity of $0$ is $d_n$ for $q \in \R_{>0}$.

\vskip.1in
\noindent
\paragraph{\bf{The eigenvalue on $S^{1^n}$}} 

When $n$ is even, the unique $\frakt \in \syt(1^n)$ is a desarrangement tableau, since ${\sf Des}(\frakt) = \{ 1, \cdots ,n-1 \}$, so $[n] \sm {\sf Des}(\frakt) = \{ n\}$. Thus in this case $\eigen_{(1^n) \sm (1^n)} (q) =0$.

When $n$ is odd, $\frakt \in \syt(1^n)$ is not a desarrangement tableau. The only nontrivial horizontal strip of $(1^n)$ is $(1^n) \sm (1^{n-1}).$ By Theorem \ref{thm:summarythm}(1),
\begin{align*}
        \eigen_{(1^n) \sm (1^{n - 1})}(q) = q^n[-(n - 1)]_q + [n]_q = [n - (n - 1)]_q = [1]_q = 1.
    \end{align*}
Since $n-1$ is even in this case, $d^{(1^{n-1})} = 1$, so $\eigen_{(1^n) \sm (1^{n - 1})}(q)$ occurs with multiplicity 1.
\vskip.1in
\noindent
\paragraph{\bf{Eigenvalues with multiplicity zero: $d^\mu = 0$}} 
The eigenvalue $\eigen_{\lambda \sm \mu}(q)$ is constructed from $u \in \kappa_\mu$. Hence if $\kappa_\mu = \emptyset,$ the eigenvalue $\eigen_{\lambda \sm \mu}(q)$ will appear with multiplicity 0. This is reflected in Theorem \ref{thm:summarythm}(2), since if $\kappa_\mu = \emptyset$, then $d^\mu = 0$.

\bibliographystyle{abbrv}
\bibliography{bibliography}
\appendix
\section{Data}\label{appendix:data}

Below we include the spectrum of $\RR_n(q)$ for $n = 2,3,4,5$, where: 
\begin{itemize}
    \item Each row in the table corresponds to an eigenvalue $\eigen_{\lambda \sm \mu}(q)$ of $\RR_n(q)$ in $S^\lambda$. For the important cases of ${\lambda \sm \mu}$  explained in Section \ref{sec:further-results-for-positive-q}, we write $\eigen_{\lambda \sm \mu}(q)$ as computed in that section.
    \item The first column shows the corresponding $\lambda$ with $\mu \subseteq \lambda$ shaded in gray. The boxes in $\lambda \sm \mu$ are filled with the content they contribute to $\content_{\lambda \sm \mu}.$

    \item The third column indicates the (algebraic) multiplicity of the eigenvalue from $\RR_n(q)$ acting on $S^\lambda$ as indicated by Theorem \ref{thm:eigenbasisSlambda}.
    \item   The last column indicates the (algebraic) multiplicity of the eigenvalue from $\RR_n(q)$ acting on $\HH_n(q)$, as explained by Theorem \ref{thm:summarythm}.

\end{itemize}

\setlength{\extrarowheight}{.5cm}
\ytableausetup{boxsize=1.38em}
\subsubsection*{$n = 2$}
    \begin{center}
        \begin{longtable}{cccc}
       Horizontal strip  $\lambda \sm \mu$& $\eigen_{\lambda \sm \mu}(q)$ & Multiplicity in $S^\lambda$: $d^\mu$& Multiplicity in $\HH_n(q)$: $d^\mu f^\lambda$ \\[.4cm]\hline 
           \begin{ytableau}
             \color{blue}{0} &\color{blue} 1  
           \end{ytableau} & 
           $([2]_q)^2$ & $1$ & $1$\\[.4cm]
            \hline
           $\ydiagram[*(lightgray)]{1,1}*[*(white)]{1,1}$&  $0$ & $1$ & $1$
        \end{longtable}
    \end{center}
\subsubsection*{$n = 3$}
\setlength{\extrarowheight}{.7cm}

 \begin{center}
        \begin{longtable}{cccc}
           Horizontal strip  $\lambda \sm \mu$& $\eigen_{\lambda \sm \mu}(q)$ & Multiplicity in $S^\lambda$: $d^\mu$& Multiplicity in $\HH_n(q)$: $d^\mu f^\lambda$ \\[.4cm]\hline
           $\begin{ytableau}
               \color{blue} 0 & \color{blue}1 & \color{blue}2
           \end{ytableau}$ & 
           $([3]_q)^2$ & $1$ & $1$\\[.4cm]
            \hline
            $\ydiagram[*(lightgray)]{2,1}$  & $0$ & $1$ & $2$\\[.4cm]
           $\begin{ytableau}
               *(lightgray)& *(white) \color{blue}{1} \normalcolor\\ *(lightgray)
           \end{ytableau}$ & $q^3\color{blue}{[1]_q}\normalcolor + [3]_q = [1]_q \cdot [4]_q $ & $1$ & $2$\\[.4cm]
            \hline
            $\begin{ytableau}
                *(lightgray)\\
                *(lightgray)\\
            \color{blue}{-2}
            \end{ytableau}$ & $1$ & $1$ & $1$
        \end{longtable}
    \end{center}
    
\subsubsection*{$n = 4$}
\setlength{\extrarowheight}{1cm}
\begin{center}
        \begin{longtable}{cccc}
            Horizontal strip  & $\eigen_{\lambda \sm \mu}(q)$ & Multiplicity in & Multiplicity in \\[-.7cm]
           $\lambda \sm \mu$&   &  $S^\lambda$: $d^\mu$   & $\HH_n(q)$: $d^\mu f^\lambda$\\[.4cm]\hline
           $\begin{ytableau}
               \color{blue} 0 & \color{blue}{1} & \color{blue} 2 & \color{blue} 3
           \end{ytableau}$ & 
           $([4]_q)^2$ & $1$ & $1$\\[.4cm]
            \hline
            $\ydiagram[*(lightgray)]{3, 1}$  & $0$ &  $1$ & $3$\\
          $\begin{ytableau}
              *(lightgray) & *(lightgray) & \color{blue}{2}\\
              *(lightgray)
          \end{ytableau}$  & $q^4\color{blue}{[2]_q}\normalcolor + [4]_q = [1]_q \cdot [6]_q$ & $1$ & $3$\\
          $\begin{ytableau}
              *(lightgray) & \color{blue}{1} & \color{blue}{2} \normalcolor\\
              *(lightgray)
          \end{ytableau}$ &$q^4\left(\color{blue}[1]_q + [2]_q \normalcolor\right) + \left(q[3]_q + [4]_q\right) = [2]_q \cdot [5]_q$ & $1$ & $3$\\[.6cm]
            \hline
           $\ydiagram[*(lightgray)]{2,2}$ & $0$ & $1$ & $2$\\
           $\begin{ytableau}
               *(lightgray) & *(lightgray)\\
               *(lightgray) & \color{blue}{0}\normalcolor
           \end{ytableau}$  & $q^4\color{blue}[0]_q \normalcolor + [4]_q$ & $1$ & $2$\\[.6cm]
            \hline 
            $\ydiagram[*(lightgray)]{2,1,1}$  & $0$ & $1$ & $3$\\
           $\begin{ytableau}
               *(lightgray) & *(lightgray)\\
               *(lightgray)\\
               {\color{blue}{-2}}\normalcolor
           \end{ytableau}$  & $q^4\color{blue}{[-2]_q} + \normalcolor{[4]_q}$ & $1$ & $3$\\
          $\begin{ytableau}
              *(lightgray) & \color{blue}{1}\\
              *(lightgray)\\
              \color{blue}{-2}
          \end{ytableau}$   & $q^4\left(\color{blue}[-2]_q + [1]_q\normalcolor\right)\normalcolor + \left(q[3]_q + [4]_q\right)$ & $1$ & $3$\\[.8cm]
            \hline 
            $\ydiagram[*(lightgray)]{1, 1, 1, 1}$ & $0$ & $1$ & $1$
        \end{longtable}
    \end{center}

\subsubsection*{$n = 5$}
\setlength{\extrarowheight}{1.2cm}

\begin{center}
        \begin{longtable}{cccc}
            Horizontal strip  & $\eigen_{\lambda \sm \mu}(q)$ & Multiplicity in & Multiplicity in \\[-.8cm]
           $\lambda \sm \mu$&   &  $S^\lambda$: $d^\mu$   & $\HH_n(q)$: $d^\mu f^\lambda$\\[.4cm]\hline
           $\begin{ytableau}
               \color{blue} 0 & \color{blue} 1 & \color{blue} 2 & \color{blue} 3 & \color{blue} 4
           \end{ytableau}$ & 
           $([5]_q)^2$ & $1$ & $1$\\[.4cm]
            \hline
            $\ydiagram[*(lightgray)]{4,1}$ & $0$ & $1$ & $4$\\
           $\begin{ytableau}
               *(lightgray) & *(lightgray) & *(lightgray) & \color{blue}{3}\\
               *(lightgray)
           \end{ytableau}$ & $q^5 \color{blue}{[3]_q}\normalcolor + [5]_q = [1]_q \cdot [8]_q$ & $1$ & $4$\\
            $\begin{ytableau}
               *(lightgray) & *(lightgray) & \color{blue}{2} & \color{blue}{3}\\
               *(lightgray)
           \end{ytableau}$ & $q^5 \left(\color{blue}[2]_q + [3]_q \normalcolor \right) + \left(q[4]_q + [5]_q\right) = [2]_q \cdot [7]_q$& $1$ & $4$\\
            $\begin{ytableau}
               *(lightgray) & \color{blue}{1} & \color{blue}{2} & \color{blue}{3}\\
               *(lightgray)
           \end{ytableau}$ &$q^5 \left(\color{blue}[1]_q + [2]_q + [3]_q\normalcolor\right) + \left(q^2[3]_q + q[4]_q + [5]_q\right) = [3]_q \cdot [6]_q$ & $1$ & $4$\\[.6cm]
            \hline
            $\ydiagram[*(lightgray)]{3, 2}$ & $0$ & $2$ & $10$\\
            $\begin{ytableau}
                *(lightgray) & *(lightgray) & *(lightgray)\\
                *(lightgray) & \color{blue}{0}
            \end{ytableau}$& $q^5\left(\color{blue}[0]_q \normalcolor\right) + [5]_q$ & $1$ & $5$\\
            $\begin{ytableau}
                *(lightgray) & *(lightgray) & \color{blue}{2}\\
                *(lightgray) & *(lightgray)
            \end{ytableau}$ & $q^5\left(\color{blue}[2]_q \normalcolor\right) + [5]_q$ & $1$ & $5$\\
            $\begin{ytableau}
                *(lightgray) & *(lightgray) & \color{blue}{2}\\
                *(lightgray) & \color{blue}{0}
            \end{ytableau}$  & $q^5 \left(\color{blue}[0]_q + [2]_q\normalcolor\right) + \left(q[4]_q + [5]_q\right)$ & $1$ & $5$\\[.6cm]
            \hline
            $\ydiagram[*(lightgray)]{3,1,1}$ & $0$ & $2$ & $12$\\
            $\begin{ytableau}
                *(lightgray) & *(lightgray) & *(lightgray)\\ *(lightgray)\\ \color{blue}{-2}
            \end{ytableau}$& $q^5 \color{blue}[-2]_q\normalcolor + [5]_q$ & $1$ & $6$\\
           $\begin{ytableau}
                *(lightgray) & *(lightgray) & \color{blue}{2}\\ *(lightgray)\\ *(lightgray)
            \end{ytableau}$ & $q^5\color{blue}[2]_q\normalcolor + [5]_q = [1]_q \cdot [7]_q$ & $1$ & $6$\\
           $\begin{ytableau}
                *(lightgray) & *(lightgray) & \color{blue}{2}\\ *(lightgray)\\ \color{blue}{-2}
            \end{ytableau}$ & $q^5 \left(\color{blue}[-2]_q + [2]_q\normalcolor\right) + q[4]_q + [5]_q$ & $1$ & $6$\\
            $\begin{ytableau}
                *(lightgray) & \color{blue}{1} & \color{blue}{2}\\ *(lightgray)\\ \color{blue}{-2}
            \end{ytableau}$& $q^5 \left(\color{blue}[-2]_q + [1]_q + [2]_q\normalcolor\right) + \left(q^2[3]_q + q[4]_q + [5]_q\right)$ & $1$ & $6$\\[.8cm]
            \hline
            $\ydiagram[*(lightgray)]{2, 2, 1}$ & $0$ & $2$ & $10$\\
            $\begin{ytableau}
                *(lightgray) & *(lightgray)\\
                *(lightgray) & *(lightgray)\\
                \color{blue}{-2}
            \end{ytableau}$& $q^5 \color{blue}[-2]_q \normalcolor + [5]_q$ & $1$ & $5$\\
           $\begin{ytableau}
                *(lightgray) & *(lightgray)\\
                *(lightgray) & \color{blue}{0}\\
                *(lightgray)
            \end{ytableau}$ & $q^5 \color{blue}[0]_q \normalcolor + [5]_q$ & $1$ & $5$\\
           $\begin{ytableau}
                *(lightgray) & *(lightgray)\\
                *(lightgray) & \color{blue}{0}\\
                \color{blue}{-2}
            \end{ytableau}$ & $q^5 \left(\color{blue}[0]_q + [-2]_q \normalcolor\right) + q[4]_q + [5]_q$ & $1$ & $5$\\[.8cm]
            \hline
            $\ydiagram[*(lightgray)]{2, 1, 1, 1}$ & $0$ & $2$ & $8$\\
          $\begin{ytableau}
              *(lightgray) & *(lightgray)\\
              *(lightgray)\\
              *(lightgray)\\
              \color{blue}{-3}
          \end{ytableau}$  & $q^5\color{blue}[-3]_q \normalcolor + [5]_q$ & $1$ & $4$\\
           $\begin{ytableau}
              *(lightgray) & \color{blue}{1}\\
              *(lightgray)\\
              *(lightgray)\\
              *(lightgray)
          \end{ytableau}$  & $q^5 \color{blue}[1]_q \normalcolor + [5]_q = [1]_q \cdot [6]_q$ & $1$ & $4$\\[1cm]
            \hline
            $\begin{ytableau}
                *(lightgray)\\
                *(lightgray)\\
                *(lightgray)\\
                *(lightgray)\\
                \color{blue} -4
            \end{ytableau}$ & $1$ & $1$ & $1$
        \end{longtable}
    \end{center}

\end{document}